\documentclass[12pt]{amsart}
\usepackage{amsmath,amsthm,amscd,amssymb,color}
\usepackage{latexsym}
\usepackage[colorlinks, citecolor = blue]{hyperref}
\usepackage[capitalize, nameinlink]{cleveref}
\usepackage[alphabetic]{amsrefs}

\crefname{lem}{Lemma}{Lemmas}
\crefname{prop}{Proposition}{Propositions}

\usepackage{enumerate}

\usepackage[margin=1.4in]{geometry}
\usepackage{bbm}

\usepackage{amsfonts}
\usepackage{mathdots}
\usepackage[mathscr]{eucal}
\usepackage{mathrsfs}

\def\Tr{{\mathop{\mathrm{Tr}}}}
\def\d{{\mathrm{d}}}
\def\bsl{\backslash}
\def\inf{\infty}

\def\fO{\mathfrak o} 

\def\cH{\mathcal H}

\def\un{\mathrm{un}}

\def\scS{\mathscr S}

\newcommand{\R}{\mathbb R}
\newcommand{\Q}{\mathbb Q}
\newcommand{\C}{\mathbb C}
\newcommand{\Z}{\mathbb Z}
\newcommand{\F}{\mathbb F}
\newcommand{\eps}{\varepsilon}
\newcommand{\calO}{\mathcal O}
\newcommand{\A}{\mathbb A}

\newcommand{\calA}{\mathcal A}

\newcommand{\frako}{\mathfrak o}
\newcommand{\frakp}{\mathfrak p}
\newcommand{\Cl}{\mathrm{Cl}}
\newcommand{\JL}{\mathrm{JL}}

\newcommand{\one}{\mathbbm{1}}
\newcommand{\bs}{\backslash}
\newcommand{\<}{\langle}
\renewcommand{\>}{\rangle}
\newcommand{\bmx}{\left( \begin{matrix}}
\newcommand{\emx}{\end{matrix} \right)}

\DeclareMathOperator{\GL}{GL}  
\DeclareMathOperator{\SL}{SL} 
\DeclareMathOperator{\GSp}{GSp} 
\DeclareMathOperator{\UU}{U}
\DeclareMathOperator{\SU}{SU}

\DeclareMathOperator{\Ram}{Ram}

  \DeclareMathOperator{\Sp}{Sp}
\DeclareMathOperator{\St}{St} \DeclareMathOperator{\diag}{diag} \DeclareMathOperator{\vol}{vol}
\DeclareMathOperator{\PGSp}{PGSp} \DeclareMathOperator{\PGL}{PGL} \DeclareMathOperator{\SO}{SO}

\newtheorem{lem}{Lemma}
\numberwithin{lem}{section}
\newtheorem{prop}[lem]{Proposition}
\newtheorem{thm}[lem]{Theorem}

\newtheorem{cor}[lem]{Corollary}
\newtheorem{conj}[lem]{Conjecture}
\newtheorem{ex}[lem]{Example}

\newtheorem*{thma}{Theorem A*}
\newtheorem*{thmb}{Theorem B}

\theoremstyle{definition}
\newtheorem{rem}[lem]{Remark}

\numberwithin{equation}{section}

\title{Mass formulas and Eisenstein congruences in higher rank}

\author{Kimball Martin}
\address{Department of Mathematics, University of Oklahoma, Norman, OK 73019 USA}
\email{kimball.martin@ou.edu}

\author{Satoshi Wakatsuki}
\address{Faculty of Mathematics and Physics, Institute of Science and Engineering, Kanazawa University, Kakumamachi, Kanazawa, Ishikawa,
920-1192, Japan}
\email{wakatsuk@staff.kanazawa-u.ac.jp}

\date{\today}

\begin{document}

\maketitle


\begin{abstract}
We use mass formulas to
construct minimal parabolic Eisenstein congruences for algebraic modular forms on 
reductive groups compact at infinity, and study when these
yield congruences between cusp forms and Eisenstein series on the quasi-split
inner form.  This extends recent work of the first author on weight 2 Eisenstein congruences for GL(2) to higher rank.

Two issues in higher rank are that the transfer
to the quasi-split form is not always cuspidal and
sometimes the congruences come
from lower rank (e.g., are ``endoscopic'').  We show our
construction yields Eisenstein congruences with non-endoscopic
cuspidal automorphic forms on quasi-split unitary groups by using
certain unitary groups over division algebras.
On the other hand, when using unitary groups over fields, or
other groups of Lie type, these Eisenstein congruences typically appear to
be endoscopic.  This suggests a new way to see higher weight Eisenstein congruences for GL(2), and leads to various conjectures about GL(2)
Eisenstein congruences.

In supplementary sections, we also generalize previous weight 2
Eisenstein congruences
for Hilbert modular forms, and prove some special
congruence mod $p$ results between cusp forms on U($p$).
\end{abstract}


\setcounter{tocdepth}{1}
\tableofcontents

\section{Introduction}


In \cite{me:cong}, we gave a construction for mod $p$ congruences of weight 2 cusp forms
with Eisenstein series on PGL(2) using the Eichler mass formula for a definite
quaternion algebra and the Jacquet--Langlands correspondence.  This approach
has certain advantages over previous approaches to Eisenstein congruences for
elliptic modular forms (e.g., \cite{mazur}, \cite{yoo}): 
one can treat more general levels and 
primes $p$, as well as Hilbert modular forms, without much difficulty.

In this paper, we extend this approach to groups of higher rank.  

Suppose $\pi, \pi'$ are irreducible automorphic representations of a reductive group $G$ over a number field $F$, 
and that outside of a finite set of places $S$ there is a hyperspecial
maximal compact subgroup $K_v \subset G(F_v)$ such that $\pi_v$ and $\pi'_v$
are both $K_v$-spherical.  Then we say $\pi$ and $\pi'$ are Hecke congruent mod
$p$ (away from $S$) if there exists a prime $\frakp$ above $p$ in a sufficiently large
number field such that, for $v \not \in S$, the spherical Hecke eigenvalues for 
$\pi_v^{K_v}$ and those for $(\pi'_v)^{K_v}$ are congruent mod $\frakp$.

Here is a sample result for certain unitary groups over $\Q$.  

\begin{thma} (\cref{ex:Qi}) \hypertarget{thma}
Let $n=2m+1$ be an odd prime, $\chi$ the idele class character for $\Q$ associated to
the quadratic extension $E=\Q(i)$, and $G=\UU(n)$ the quasi-split
unitary group associated to $E/\Q$.  Let $\one_G$ denote the trivial 
representation of $G$.  Fix a prime $\ell \equiv 1 \bmod 4$. 
Suppose $p > n$ is a prime such that either $p | (\ell^r-1)$ for some
$1 \le r \le n-1$ or that $p$ divides the numerator of the product $\prod_{r=1}^m B_{2r} \cdot
\prod_{r=1}^m B_{2r+1, \chi}$ of generalized Bernoulli numbers.

Then there exists a holomorphic weight $n$ cuspidal representation
$\pi$ of $G(\A)$ such that (i) $\pi_v$ is unramified at each finite odd $v \ne \ell$,
(ii) $\pi_2$ is spherical, (iii) $\pi_\ell$ is an unramified twist of the Steinberg representation,
(iv) the base change $\pi_{E}$ of $\pi$ to $\GL_n(\A_E)$ is cuspidal, and
(v) $\pi$ is Hecke congruent to $\one_G \bmod p$.
\end{thma}

The asterisk in the theorem refers to an underlying assumption of
the endoscopic classification for unitary groups when $n > 3$, to be
discussed below.

We can regard this as an Eisenstein congruence as follows.  
Suppose $G/F$ is semisimple with a Borel subgroup $B$.  For a character
$\chi$ of the Levi of $B$, consider the principal series representation $I(\chi)$
induced from $\chi$.  Choosing standard sections of $I(\chi)$ yields Eisenstein
series, which are not in general $L^2$.  In particular, if $\chi = \delta_G^{-1/2}$
where $\delta_G$ denotes the modulus character, then $I(\chi)$ contains
$\one_G$ as a subrepresentation, and $\one_G$ contributes to the residual
part of the discrete $L^2$ spectrum.
Note one can reformulate the weight 2 Eisenstein series
congruence for elliptic modular forms from \cite{mazur}, \cite{me:cong}
as congruences with $\one_G$ for $G = \PGL(2)$.

The Hecke eigenvalues for $\one_G$ are relatively simple to describe,
being simply the degrees of the corresponding Hecke operators.
For instance, if 
$G=\UU(2)$ or $\UU(3)$, the local spherical Hecke algebra 
is generated by a single Hecke operator $T_q$.
For $G=\UU(2)$, the spherical eigenvalue of $\one_G$ for $T_q$ is
$q+1$ if $q$ is split in $E/\Q$ and $q^2+q$ if $q$ is inert in $E/\Q$.
For $G=\UU(3)$, the spherical eigenvalue of $\one_G$ for $T_q$ is
$q^2+q+1$ if $q$ is split in $E/\Q$ and $q^4+q$ if $q$ is inert in $E/\Q$.
In general, there are many local Hecke operators at $q$.  
See \eqref{eq:congvol} and \eqref{eq:congvol-Un} for
a general description of the unramified Hecke eigenvalues for $\one_G$.

As a specific example of the Bernoulli number divisibility condition, for any $\ell \equiv 1 \bmod 4$, we may take $p=61$ if $n=7$ or $p \in
\{ 19, 61, 277, 691 \}$ if $n=13$.

We also remark that the condition $p > n$ is not needed in our general result.

To our knowledge, these are the first general Eisenstein congruence results
in higher rank for Eisenstein series attached to minimal parabolic
subgroups.  Some instances of Eisenstein congruences in higher rank
are known for Eisenstein series attached to maximal parabolic subgroups,
e.g., see \cite{bergstrom-dummigan}.

There are two main steps to the proof: 
(1) construct appropriate
congruences on certain compact inner forms, and (2)
transfer these congruences to the quasi-split forms via functoriality.  

We begin by explaining (1) in a quite general setting.
Let $F$ be a totally real number field and $G/F$ be a reductive group which is
compact at infinity. 
Gross \cite{gross:amf} defined a notion of algebraic modular forms on $G$.  
 Let $K = \prod K_v$ be a suitably nice compact open subgroup of $G(\A)$.
In particular, 
we assume $K_v$ is a hyperspecial maximal compact subgroup almost everywhere and 
$K_v = G_v$ for $v | \infty$.
Let $\calA(G,K)$ denote
the space of algebraic modular forms with level $K$ and trivial weight.  We may view
$\calA(G,K)$ as the space of $\C$-valued functions on the finite set
$\Cl(K) = G(F) \bs G(\A) / K$.  
Let $x_1, \dots, x_h \in G(\A)$ be a set of representatives for $\Cl(K)$ and 
put $w_i = |G(F) \cap x_i K x_i^{-1}|$.  On $\calA(G,K)$, we consider 
the inner product $(\phi, \phi') = \sum \frac 1{w_i} \phi(x_i) \overline{\phi'(x_i)}$.
This space has a basis of orthogonal eigenforms for the unramified Hecke algebra.
The constant function $\one$ is an eigenform, which we think of as a 
compact analogue of an Eisenstein series 
associated to the minimal parabolic of the quasi-split form. 
Let $\calA_0(G,K)$ be the orthogonal
complement on $\one$ in $\calA(G,K)$.  The mass of $K$ is defined to be
\[ m(K) = (\one, \one) =  \frac 1{w_1} + \cdots + \frac 1{w_h}. \]
We say two eigenforms are congruent mod $p$ if their automorphic representations are.

\begin{thmb} (\cref{gen-prop}) \hypertarget{thmb} If $p | m(K)$, then there exists an eigenform 
$\phi \in \calA_0(G,K)$ which is Hecke congruent to $\one$ mod $p$.
\end{thmb}

Explicit mass formulas have been computed in a wide variety of settings---e.g.,
see \cite{shimura:bull} and \cite{GHY}.  We explicate these mass formulas in 
a number of cases below.  For the relevant inner forms and compact open
subgroups for \hyperlink{thma}{Theorem A*}, the mass is $(2^n n!)^{-1} \prod_{r=1}^{n-1} (\ell^r -1)$
times the product of Bernoulli numbers that appears in the theorem.

Now we briefly explain step (2).  \hyperlink{thmb}{Theorem B} gives us Eisenstein
congruences on definite unitary groups.  To obtain \hyperlink{thma}{Theorem A*},
we work with an inner form $G$ of $\UU(n)$ which is compact at infinity and compact
mod center at $\ell$, i.e., $G$ is unitary group over a division algebra. 
By comparing the endoscopic classification of discrete $L^2$
automorphic representations of $G$ with those of the quasi-split form $\UU(n)$, 
one gets a transfer of automorphic representations of $G$ to those of $\UU(n)$.
Since $G$ is compact mod center at $\ell$ and $n$ is prime, 
if $\pi$ is a non-abelian (not $1$-dimensional) 
automorphic representation of $G$, the transfer to $G'$ must be non-endoscopic
and have cuspidal base change to $\GL_n(\A_E)$.  In this case, there are no
abelian automorphic representations occurring in $\calA_0(G,K)$, which gives
\hyperlink{thma}{Theorem A*}.  For definite unitary groups associated to a general 
CM extension $E/F$, one can refine \hyperlink{thmb}{Theorem B} to obtain
a non-abelian $\phi$ under the condition that $p | \frac{m(K)}{n|\Cl(\UU_1(E/F))|}$.
See \cref{thm:main-Un} for a precise statement.

The endoscopic classification results that we use were obtained (conditional on
stabilization of trace formulas) in
\cite{mok} for $\UU(n)$ and were announced in \cite{KMSW} for inner forms.
However, the proof for the case of inner forms, while known in many situations, is still
work in progress, and we assume this classification in \hyperlink{thma}{Theorem A*}.  For $n=3$, the endoscopic classification was completed
for all inner forms in \cite{rogawski}, and thus our results are unconditional at least
for $n=3$.

On the other hand, if one carries out this procedure for unitary groups over
fields (so not compact at a finite place for $n > 2$), in all cases we have
checked it appears that all Eisenstein congruences coming from the mass
formula can also be explained as endoscopic lifts of GL(2)
Eisenstein congruences in higher weight.  In particular, we suggest
that one \emph{only} sees weight $\le k$ GL(2) Eisenstein congruences in 
mass formulas for suitable $\UU(k)$ ($k \ge 2$).
We use this to formulate precise conjectures
about weight $k$ Eisenstein congruences on GL(2) which generalize 
existing results.  See \cref{conj:wtk,rem:wtk}.  
However, at least for small $p$, one does not see all weight
$k$ GL(2) Eisenstein congruences in mass formulas for $\UU(k)$
(see \cref{ex:endo-nomass}).

We also consider Eisenstein congruences coming from mass formulas
for groups not of type $A_n$: specifically symplectic and odd orthogonal groups (with a focus on SO(5)) and $G_2$.  Again, in these, it appears
that such congruences are also explained by endoscopic lifts.
 (These groups have no inner forms which are compact at a finite place.)
For instance, if $G'=\SO(5) \simeq \PGSp(4)$, 
the Eisenstein congruences we obtain correspond to scalar weight 3
Siegel modular forms, and they appear to simply be Saito--Kurokawa lifts of Eisenstein congruences of weight 4 modular forms on GL(2).  
Again, this suggests some refinements of known results on Eisenstein
congruences for GL(2).  See for instance \cref{conj:wt4}.

Now we briefly outline the contents of our paper.

In \cref{sec2}, we give a general treatment of 
(trivial weight) algebraic modular forms on reductive
groups compact at infinity and congruences via mass formulas.
In \cref{sec:local}, we discuss local Hecke algebras and the local
versions of Eisenstein congruences for the groups we will consider.
In \cref{sec:lifted-cong}, we explain how, under certain conditions, these 
local results let us obtain
Eisenstein congruences on unitary groups via functorial liftings
(endoscopic and symmetric powers) of Eisenstein congruences of smaller
groups.

In \cref{sec:unitary}, we first explain how to obtain Eisenstein congruences on
$\UU(n)$ for $n$ prime which are not endoscopic.  
Then we investigate the case of definite unitary groups over fields
with some examples leading to \cref{conj:wtk}.  
In \cref{sec:orthog} we discuss
orthogonal and symplectic groups, with a focus on $\SO(5)$,
and then consider $G_2$ in \cref{sec:G2}.
In \cref{sec:GL2} we use the techniques here to slightly refine
our results on weight 2 Eisenstein congruences for GL(2) from \cite{me:cong}.

Finally, in \cref{sec:special} we show that if $\pi$ is a cuspidal representation of
$\UU(p)$ with trivial central character such that $\pi_v$ is an unramified twist of Steinberg
at some finite $v$, there exists a cuspidal $\pi'$ on $\UU(p)$ with the same level structure as
$\pi$ which is Hecke congruent to $\pi$ mod $p$ and $\pi'_v$ is Steinberg
at $v$.  This is a higher rank analogue of a mod 2 congruence result on 
$\GL(2)$ from \cite{me:cong}.


\subsection*{Notation}

Excluding the local section \cref{sec:local}, $F$ will denote a number field, $\frako = \frako_F$ its ring of integers, $\A = \A_F$ 
its adele ring, and $v$ a place of $F$.  We also denote the finite adeles by
$\A_f$ and put $\hat \frako = \prod_{v < \infty} \frako_v$.  At a finite place $v$,
we denote by $\frakp_v$ the prime ideal and $q_v$ the size of the residue field.

For a group $G$, we denote its center by $Z(G)$, or just by $Z$ if $G$ is 
understood.  For an algebraic group $G$ over $F$, we
often write $G_v$ for $G(F_v)$.  By an automorphic representation, by
default we mean an irreducible $L^2$-discrete automorphic representation.

Finally $p$ will typically denote our congruence prime.
To denote other primes, we generally use $v$ to denote other primes, 
or $\ell$ or $q$ when $F=\Q$.


\subsection*{Acknowledgements} We thank Wei Zhang for suggesting the use 
of unitary groups compact at a finite place, and Sug Woo Shin for answering
a question about their endoscopic classification.  We are grateful to Markus
Kirschmer for providing a ``pre-release'' version of his Magma code with David Lorch to carry out Hermitian lattice calculations.
We also thank Hiraku Atobe, Tobias Berger,
Ralf Schmidt, Takashi Sugano and Shunsuke Yamana for helpful discussions.  This project was
supported by the University of Oklahoma's VPR Faculty Investment Program and 
CAS Senior Faculty Research Fellowship, as well
as a JSPS Invitation Fellowship (Short-term, S18025).  K.M.\ was partially
supported by grants from the Simons Foundation/SFARI (240605 and 512927, KM).
S.W.\ was partially
supported by JSPS Grant-in-Aid for Scientific Research (No.~18K03235).


\section{Congruences from mass formulas} \label{sec2}


Let $F$ be a totally real number field.  Let $G$ be a connected reductive linear 
algebraic group over $F$ such that $G_\infty$ is compact.  
Let $K = \prod K_v$ be
an open compact subgroup of $G(\A)$ such that 
$K_v = G_v$ for $v | \infty$.  For later use of the theory of Hecke operators, 
we will also assume $K_v$ is a hyperspecial maximal compact subgroup for 
all $v$ outside of a finite set of places $S$, which contains all infinite places.

Fix a nonzero Haar measure $dg$ on $G(\A)$ which is a product of local Haar 
measures $dg_v$.
The mass of $K$ is defined to be
\begin{equation}
 m(K) = \frac{\vol(G(F) \bs G(\A), dg)}{\vol(K, dg)}.
\end{equation}
(As usual, we give the discrete subgroup $G(F)$ the counting measure and
the volume of the quotient $G(F) \bs G(\A)$ really means with respect to the quotient
measure.)  This is nonzero, finite, and independent of the choice of $dg$.
Note that if $K' \subset K$ is also a compact open subgroup, then
$m(K') = [K:K'] m(K)$.

Consider the classes $\Cl(K) = G(F) \bs G(\A) / K$.  We identify $\Cl(K)$ with a set
of representatives $\{ x_1, \ldots, x_h \}$, where $x_i \in G(\A)$.  Note 
$\vol(G(F) \bs G(F) x_i K, dg) = \frac 1{w_i} \vol(K, dg)$ where $w_i = | G(F) \cap
x_i K x_i^{-1} |$.  Thus we can also express the mass as
\begin{equation}
m(K) = \frac 1{w_1} + \dots + \frac 1{w_h}.
\end{equation}
Consequently, $m(K) \in \Q$.

If $G$ is a unitary, symplectic or orthogonal group, and $K$ is the stabilizer of a 
lattice $\Lambda$, then this mass corresponds to the classical mass of $\Lambda$.
Mass formulas have been calculated in a considerable amount of generality
in many works, e.g.\ see \cite{GHY} or  \cite{shimura:bull}.  We will explicate
these in some cases below.


\subsection{Algebraic modular forms}

The basic theory of algebraic modular forms was developed in \cite{gross:amf}.
Below, we will review aspects necessary for our applications.

We define the space of algebraic modular forms 
on $G(\A)$ with level $K$
and trivial weight to be
\begin{equation}
\calA(G,K) = \{ \phi : \Cl(K) \to \C \}.
\end{equation}
As $\calA(G, K) \subset L^2(G(F) \bs G(\A))$, we can decompose this
space as
\begin{equation} \label{eq:AGK-rep-decomp}
 \calA(G, K) = \bigoplus \pi^K,
\end{equation}
where $\pi$ runs over irreducible automorphic representations of $G(\A)$ with 
trivial infinity type.  If $\pi^K \ne 0$, we will say $\pi$ occurs in $\calA(G, K)$.
Since $G(F) \bs G(\A)$ is compact,
$L^2(G(F) \bs G(\A))$ decomposes discretely and each $\pi$ above is
finite dimensional.
The usual inner product on $L^2(G(F) \bs G(\A))$ restricts
to an inner product on $\calA(G,K)$, which after suitable normalization we can
take to be
\[ (\phi, \phi') = \sum \frac 1{w_i} \phi(x_i) \overline{\phi'(x_i)}. \]

Let $Z$ denote the center of $G$ and $K_Z = K \cap Z(\A)$.  
Note that $\Cl(K_Z) = Z(F) \bs Z(\A) / K_Z$
acts on elements of $\calA(G,K)$ by (left or right) multiplication. 
Let $\omega : \Cl(K_Z) \to \C$ be a ``class character''.\footnote{If we relax our 
compact at infinity condition to compact mod
center at infinity, and suppose $Z = \GL(1)$ and $K_Z = \hat \frako_F^\times \times F_\infty^\times$, this is just an ideal class character of $F$.}  
Define the space of algebraic modular forms with central character $\omega$,
level $K$ and trivial weight to be
\[ \calA(G, K; \omega) = \{ \phi \in \calA(G, K) : \phi(zg) = \omega(z) \phi(g) \text{ for }
z \in Z(\A), \, g \in G(\A) \}. \]
By decomposing $\calA(G,K)$ with respect to the action of $\Cl(K_Z)$, we obtain
a decomposition
\[ \calA(G, K) = \bigoplus_\omega \calA(G,K; \omega), \]
where $\omega$ runs over characters of $\Cl(K_Z)$.  We also have decompositions
of the form \eqref{eq:AGK-rep-decomp} for each $\calA(G, K; \omega)$, where
now one runs over $\pi$ with central character $\omega$.

If $\chi : G(\A) \to \C$ is a 1-dimensional representation and $\ker \chi \supset
G(F) K$, then we may view $\chi$ as an element of $\calA(G, K)$.  
In particular, the space for trivial representation is the span of the constant function
$\one \in \calA(G, K)$.  Define the codimension 1 subspace
\[ \calA_0(G, K) = \{ \phi \in \calA(G, K) : (\phi, \one) = 0 \}, \]
and put $\calA_0(G, K; \omega) = \calA(G,K, \omega) \cap \calA_0(G, K)$.
We say
$\phi \in \calA(G, K)$ is non-abelian if it is not a linear combination of 1-dimensional
representations of $G(\A)$.

In the special case $G = B^\times$, where $B$ is a definite quaternion algebra
over $F=\Q$ and $K$ is the multiplicative group of an 
Eichler order of level $N$, then $\one \in \calA(G, K)$ corresponds to a 
weight 2 Eisenstein series on $\GL(2)$ and the Jacquet--Langlands 
correspondence gives a Hecke isomorphism of $\calA_0(G, K)$ with 
the subspace of $S_2(N)$ which are $p$-new for $p$ ramified in $B$.
However, in general $\calA_0(G, K)$ may contain many abelian forms,
as well as many non-abelian forms $\phi$ (even in $\pi^K$ for some $\pi$
occurring in $\calA_0(G, K)$) which do not correspond to cusp
forms via a generalized Jacquet--Langlands correspondence.  For higher rank 
$G$, it is a difficult problem to describe the set of $\phi$ which correspond
to cusp forms on the quasi-split form of $G$.

Finally, for a subring $\calO$ of $\C$, and a space of algebraic modular forms 
(e.g., $\calA(G, K)$) we denote with a superscript $\calO$ 
(e.g., $\calA^\calO(G,K)$)  the subring of $\calO$-valued algebraic modular forms.


\subsection{Hecke operators}
For $g \in G(\A)$ and $\phi \in \calA(G, K)$, we define the Hecke operator
\begin{equation} \label{eq:hecke-def}
(T_g \phi)(x) = \sum \phi(xg_i), \quad KgK = \coprod g_i K.
\end{equation}
By right $K$-invariance of $\phi$, this is independent of the choice of representatives
$g_i$ in the coset decomposition $KgK = \coprod g_i K$, and $T_{g'} = T_g$
if $g' \in KgK$.  Clearly each $\pi^K$ is stable under $T_g$ for
each $\pi$ occurring in $\calA(G, K)$.  In particular, $T_g$ acts on the subspaces
$\calA_0(G, K)$ and $\calA_0(G, K, \omega)$.

We also note that each $T_g$ is integral
in the sense that, viewing each $\phi$ as a column vector $(\phi(x_i)) \in \C^h$,
the action is given by left multiplication by an integral matrix in $M_h(\Z)$.  
Consequently, for any
subring $\calO \subset \C$, $T_g$ restricts to an operator on $\calA^\calO(G, K)$
(and similarly, $\calA^\calO_0(G, K)$, etc.).  Moreover, all eigenvalues for
$T_g$ are algebraic integers.

Consider a representation $\pi = \bigotimes' \pi_v$ occurring in $\calA(G, K)$.
Take $v \not \in S$.  Then $\pi_v$ is $K_v$-spherical, and $\dim \pi_v^{K_v} = 1$.
Viewing $g_v \in G(F_v)$ as an element of $G(\A)$ which is $g_v$ at $v$
and 1 at all other places, we can consider the (global) Hecke operator
$T_{g_v}$.   Then $T_{g_v}$ acts by a scalar on $\pi_v^{K_v}$, and hence
is diagonalizable on $\calA(G, K)$.  

In fact, the $T_{g_v}$'s are simultaneously diagonalizable for all $v \not \in S$ 
and all $g_v \in G(F_v)$.  Specifically, let us call any nonzero $\phi \in \calA(G,K)$
such that $\phi \in \pi^K$ for some $\pi$ an eigenform.  Such a $\phi$ is a simultaneous
eigenform for all $T_{g_v}$'s with $v \not \in S$.  We denote the corresponding eigenvalue by $\lambda_{g_v}(\phi)$.  Then any basis of $\calA(G,K)$ of eigenforms 
simultaneously diagonalizes the $T_{g_v}$'s ($v \not \in S$).

Note that $\one$ is always an eigenform, and $\lambda_{g_v}(\one)$
is the degree of $T_{g_v}$, i.e., the number of $g_i$'s occurring in the decomposition
$Kg_vK = \coprod g_i K$, which equals $\vol(K_v g_v K_v) / \vol(K_v)$.


\subsection{Congruences}

Let $\phi, \phi' \in \calA(G, K)$ be eigenforms.
We say $\phi$ and $\phi'$ are Hecke congruent mod $p$ (away from $S$) 
if, for all $v\not \in S$ and all $g_v \in G(F_v)$, 
$\lambda_{g_v}(\phi) \equiv \lambda_{g_v}(\phi') \bmod \frakp$,
where $\frakp$ is a prime of some finite extension of $\Q$.

For a subring $\calO \subset \C$, ideal $\mathfrak n$ in $\calO$ and
$\phi_1, \phi_2 \in \calA^\calO(G,K)$, we write $\phi_1 \equiv \phi_2
\bmod \mathfrak n$ if $\phi_1(x_i) \equiv \phi_2(x_i) \bmod \frakp$ for all
$x_i \in \Cl(K)$.  Note if $\phi_1$ and $\phi_2$ are eigenforms
(or just mod $\frakp$ eigenforms), then $\phi_1 \equiv \phi_2 \bmod \frakp$
implies $\phi_1$ and $\phi_2$ are Hecke congruent mod $\frakp$.

In addition, if $\alpha \in \Q$, by $p | \alpha$ we mean $p$ divides the numerator
of $\alpha$.

\begin{prop} \label{gen-prop}
 Suppose $p | m(K)$.  Then there exists an eigenform $\phi
\in \calA_0(G, K)$ which is Hecke congruent to $\one \bmod p$.
\end{prop}

\begin{proof} One can use the same arguments as those given for GL(2) in 
\cite{me:cong} and \cite{me:cong2}.  In fact we give a slightly more refined argument
than what we need for this proposition in order to use it later in 
\cref{sec:eis-Un}.

Let $r = v_p(m(K)) \ge 1$.
The first step is note that there exists a $\Z$-valued $\phi' \in \calA_0^\Z(G, K)$
such that $\phi' \equiv \one \bmod p^r$, i.e., $\phi'(x_i) \equiv 1 \bmod p^r$ for
$i = 1, \dots, h$.  To see this, consider $\phi' \in \calA^\Z(G, K)$ 
such that each $\phi'(x_i) = 1+p^r a_i$ for some $a_i \in \Z$.  We claim we can choose the $a_i$'s so that $(\phi', \one) = 0$, i.e., 
$p^r \sum \frac{a_i}{w_i} = - \sum \frac 1{w_i} = -m(K)$.
Let $w = \prod w_i$ and $w_i^* = \frac w{w_i}$.  Then we want 
$a_i \in \Z$ such that $\sum a_i w_i^* = -w\frac{m(K)}{p^r}$.  Note that
$p^j | w_i^*$ for some $i$ implies $p^{j} | w$ and thus $p^{j+r} | w m(K)$.  Thus
$\gcd(w_1^*, \ldots, w_h^*) | w\frac{m(K)}{p^r}$, and we may choose the $a_i$'s
as claimed.

Take such a $\phi'$, which is a mod $p$ eigenform.
Now we want to pass from $\phi'$ to an eigenform $\phi$ which
is Hecke congruent to $\phi'$ mod $p$.  
For this, one can either use the Deligne--Serre lifting
lemma as in the proof of \cite[Theorem 5.1]{me:cong2} or the reduction 
argument as in proof of \cite[Theorem 2.1]{me:cong}.
Specifically, the subsequent \cref{lift-lem} is a slight refinement of the
latter, and applying it with $\calO = \Z$, $\phi_1 = \one$, $\phi_2 = \phi'$
and $W = \calA_0(G,K)$ gives the desired $\phi$.
\end{proof}

\begin{lem} \label{lift-lem}  Let $\calO$ be the ring of integers of a number field $L$,
and $\frakp$ a prime of $\calO$ above a rational prime $p$.
Let $\phi_1 \in \calA^\calO(G,K)$ be an eigenform.  Let $W$ be
a Hecke-stable subspace of $\calA(G,K)$.  Suppose
there exists $\phi_2 \in \calA^\calO(G, K)$ such that $\phi_2 \equiv
\phi_1 \bmod \frakp$ and $\phi_2$ has nonzero projection to $W$.
Then there exists an eigenform $\phi \in W$ such that $\phi$
is Hecke congruent to $\phi_1 \bmod p$ for all Hecke operators $T_g$.
\end{lem}

\begin{proof} 
Enlarge $L$ if necessary to assume that $\calA^\calO(G,K)$
contains a basis of eigenforms $\psi_1, \dots, \psi_h$.  
Let $\Phi$ denote the collection of $\phi \in \calA^\calO(G,K)$ such
that $\phi$ is congruent to a nonzero multiple of $\phi_1 \bmod \frakp$
and $\phi$ has nonzero projection to $W$.  The hypothesis on $\phi_2$ means
$\Phi \ne \emptyset$.
Let $m$ be minimal
such that, after a possible reordering of $\psi_1, \dots, \psi_h$,
there exists $\phi = c_1 \psi_1 + \dots +c_m \psi_m \in \Phi$ with each $c_i \in L^\times$
and $\psi_1 \in W$.   
Take such a $\phi$.  

Fix any Hecke operator $T_{g}$, and put $\phi' = [T_{g} - \lambda_{g}(\psi_1)]\phi$.
Then note that
\[ \phi' \equiv (\lambda_{g}(\phi_1)
- \lambda_{g}(\psi_1)) \phi \mod \frakp. \]
Hence $\phi' \in \Phi$ unless 
$\lambda_{g}(\psi_1) \equiv \lambda_{g}(\phi_1) \bmod \frakp$.
But $\phi'$ is of the form $c_2' \psi_2 + \dots + c_m' \psi_m$ for some $c_i' \in L^\times$.
Thus $\phi' \not \in \Phi$ by minimality of $m$.  Consequently,
$\lambda_{g}(\psi_1) \equiv \lambda_{g}(\phi_1) \bmod \frakp$ for all $g$, and
we may take $\psi_1$ for our desired $\phi$.
\end{proof}

\begin{rem} \label{rem:depth}
Let $\frakp$ be the prime above $p$ in a sufficiently large
extension of $\Q_p$, with ramification index $e$.  The work
\cite{BKK} considers the notion of \emph{depth} of congruences,
which is $\frac 1e$ times the number of Hecke eigensystems satisfying a 
congruence mod $\frakp$ counted with multiplicity 
(a congruence mod $\frakp^r$ means multiplicity $r$).  Combining this
theorem with Proposition 4.3 of \emph{op.\ cit.}\ gives a lower bound on the depth of
congruences of $v_p(m(K))$.
\end{rem}

We can also guarantee the existence of such a $\phi$ with trivial central
character.

\begin{cor} \label{gen-cor}
Set $\bar G = G/Z$.
Suppose that $\bar G(k)=G(k)/Z(k)$ holds for any field $k$ of characteristic zero,
 and $p | \frac{m(K)}{m(K_Z)}$.  Then there exists an 
eigenform $\phi \in \calA_0(G, K; 1)$ which is Hecke congruent to $\one
\bmod p$.
\end{cor}

\begin{proof} Let $\bar K = Z(\A) K / Z(\A)$.  Then
$\calA_0(G, K; 1)$ may be identified with $\calA_0(\bar G, \bar K)$.
Now note that $m(\bar K) = \frac{m(K)}{m(K_Z)}$, and apply the proposition
to $\calA(\bar G, \bar K)$.
\end{proof}
The assumption for $\bar G$ in Corollary \ref{gen-cor} is satisfied when $G$ is a unitary group of odd degree. 

Ideally, we would like to be able to say when there is such a $\phi$
which is non-abelian, or more generally non-endoscopic.  This topic is the focus
of the remainder of the paper.


\section{Local congruences} \label{sec:local}


\subsection{Algebraic groups, $L$-parameters and Satake parameters}

We recall well-known facts and results for $L$-parameters of discrete series of real groups, and Stake transforms and parameters of spherical representations of $p$-adic groups.
For details of $L$-parameters over the real field, we refer to \cite{Borel}, \cite[Section 7]{Kottwitz}, \cite{lansky-pollack}, and \cite{CR}.
For details of spherical representations over $p$-adic fields, we refer to \cite[Chapter II]{Borel}, \cite[Chapter III]{Cartier}, \cite{Gross}, \cite{Minguez} and \cite{Satake}.


\subsubsection{Some quasi-split reductive groups}\label{sec:reductive}

Let $F$ be a field of characteristic zero.
Assume that $E$ is either $F\times F$ or a quadratic extension of $F$.
When $E$ is a quadratic extension of $F$, we write $\sigma$ for the non-trivial element in $\mathrm{Gal}(E/F)$.
When $E= F\times F$, we define $\sigma(x,y)=(y,x)$ for $x,y\in F$.
Set $\overline{z}=\sigma(z)$ $(z\in E)$ and let $\Phi_n$ denote an $n\times n$ matrix with alternating $\pm 1$'s on the anti-diagonal and zeros elsewhere. 
A quasi-split unitary group $\UU(n)=\UU_{E/F}(n)$ of degree $n$ over $F$ is defined by
\begin{equation}\label{eq:unitarygroup}
\UU(n)=\UU_{E/F}(n):= \{  g\in \mathrm{Res}_{E/F}\GL(n) \mid g \Phi_n {}^t\!\overline{g}=\Phi_n \}.  
\end{equation}
Note that $\UU(n,E) \simeq \GL(n,E)$.
If $E=F\times F$, then $\UU(n)$ is isomorphic to $\GL(n)$ over $F$.

A split  symplectic similitude group $\GSp(4)$ over $F$ is defined by
\[
\GSp(4):=\{ g\in\GL(4) \mid \exists \mu(g)\in\mathbb{G}_m \;\; {\rm s.t.} \;\;  gJ {}^t\!g=\mu(g)J \} ,\qquad J:=\begin{pmatrix} O_2& -I_2 \\ I_2&O_2 \end{pmatrix}. 
\]

An exceptional split simple algebraic group $G_2$ over $F$ is defined by an automorphism group of a split octonion over $F$, see \cite{SV}. 
It is known that $G_2$ is realized as a closed subgroup in a split special orthogonal group $\SO(Q)$, where $Q$ is a symmetric matrix of degree seven.
We give such a realization which is a slight modification of that in \cite[Section 4.2]{CNP}.
Let $Q:=E_{71}+E_{62}+E_{53}+2E_{44}+E_{35}+E_{26}+E_{17}$ where $E_{ij}$ denotes the matrix unit of $(i,j)$ in $M_7$.  Consider the vector space 
$\mathfrak{t}$ (a Cartan subalgebra) given by
\[
\mathfrak{t}:=\langle H_1:=\diag(0,-1,1,0,-1,1,0), \quad  H_2:=\diag(-1,-1,0,0,0,1,1) \rangle .
\]
We may assume that the fundamental roots $\alpha$ and $\beta$ satisfy $\alpha(H_1)=1$, $\alpha(H_2)=0$, $\beta(H_1)=-2$, $\beta(H_2)=-1$ ($\alpha$ is short, $\beta$ is long).
Then $\mathrm{Lie}(G_2)$ is generated by $H_1$, $H_2$, $X_\alpha:=E_{12}-E_{34}+2E_{45}-E_{67}$, $X_\beta:=-E_{23}+E_{56}$, and $X_{-3\alpha-2\beta}:=E_{61}-E_{71}$.
In particular, these give a Chevalley basis that we can use to construct $G_2$ as a closed subgroup of $\SO(Q)$.
Let $T$ denote the split torus of $G_2$ over $F$ such that $\mathfrak{t}=\mathrm{Lie}(T)$.
We choose an isomorphism $T\cong \mathbb{G}_m\times\mathbb{G}_m$ over $F$ 
corresponding to basis elements $H_1$ and $H_2$, namely $\alpha(a,b)=a$, $\beta(a,b)=a^{-2}b^{-1}$, where $(a,b)$ is identified with $\diag(b^{-1},a^{-1}b^{-1},a,1,a^{-1},ab,b)$.

\subsubsection{Discrete series of $\UU_{\C/\R}(n,\R)$}\label{sec:20190531}

The group $W_\R$ is given by $W_\R=\C^\times\sqcup \C^\times j$, $j^2=-1$, $zj=j\overline{z}$ $(z\in\C^\times)$.
Let $\psi:W_\R\to \GL(n,\C)\rtimes W_\R$ denote a tempered discrete relevant $L$-parameter of $\UU_{\C/\R}(n,\R)$ which is of the form
\begin{equation}\label{eq:LDS}
\psi(z)=\diag(  (z/\overline{z})^{l_1+\frac{n+1}{2}-1}, (z/\overline{z})^{l_2+\frac{n+1}{2}-2},\dots,   (z/\overline{z})^{l_n+\frac{n+1}{2}-n} )\rtimes z  ,\quad \psi(j)=\Phi_n\rtimes j 
\end{equation}
where $l_1,\dots,l_n\in\Z$ and $l_1\geq l_2\geq \cdots\geq l_n$.
Indeed, $\psi$ corresponds to a $L$-packet of discrete series of $\UU_{\C/\R}(n,\R)$.
When $l_1+\cdots+l_n=0$, the image of $\psi$ is contained in $\SL(n,\C)$, namely the central character of the corresponding discrete series is trivial.
In particular, if $(l_1,\dots,l_n)=(0,\dots,0)$, then the $L$-parameter $\psi$ corresponds to the $L$-packet of discrete series of $\UU_{\C/\R}(n,\R)$ with the same infinitesimal character as that of the trivial representation.

\subsubsection{Discrete series of $\GL(2,\R)$}

For each $l\in\Z$, we write $\psi_l:W_\R\to \GL(2,\C)$ for a tempered discrete relevant $L$-parameter of $\GL(2,\R)$ which is of the form
\begin{equation}\label{eq:gl2r}
\psi_l(z)=\begin{pmatrix}  (z/\overline{z})^{l/2} & 0 \\  0 & (z/\overline{z})^{-l/2} \end{pmatrix}  ,\quad \psi_l(j)=\begin{pmatrix} 0& (-1)^l \\ 1 & 0 \end{pmatrix}.
\end{equation}
Note that $\psi_l\cong\psi_{-l}$ up to conjugation.
If $l\neq 0$, then $\psi$ corresponds to the discrete series of $\GL(2,\R)$ including $[e^{i\theta}\mapsto e^{i(l+1)\theta}]$ as a minimal $\SO(2)$-type.

\subsubsection{Discrete series of $\GSp(4,\R)$}

Choose an element $(a,b)\in \Z^{\oplus 2}$ such that
\[
a\equiv b \mod 2.
\]
We denote by $\psi_{a,b}:W_\R\to \GSp(4,\C)$ a tempered discrete relevant $L$-parameter of $\GSp(4,\R)$ which is of the form
\begin{equation}\label{eq:gsp4r}
\psi_{a,b}(z)=\begin{pmatrix}  (z/\overline{z})^{a/2} & 0 & 0 & 0 \\  0 & (z/\overline{z})^{b/2} & 0 & 0 \\  0 & 0 & (z/\overline{z})^{-a/2} & 0  \\  0 & 0 & 0 & (z/\overline{z})^{-b/2} \end{pmatrix}  ,\;\; \psi_{a,b}(j)=\begin{pmatrix} 0 & 0& (-1)^a & 0 \\ 0 & 0 & 0 &  (-1)^b \\ 1 & 0& 0& 0 \\ 0 & 1& 0& 0 \end{pmatrix}.
\end{equation}
Note that $\psi_{a,b}\cong\psi_{b,a}$ and $\psi_{a,b}\cong\psi_{\pm a, \pm b}$ up to conjugation.
If $a>b>0$, then the $L$-parameter $\psi_{a,b}$ corresponds to a holomorphic discrete series of $\GSp(4,\R)$ with minimal $\UU_{\C/\R}(2,\R)$-type $\det^{\frac{a-b}{2}+2}\otimes \mathrm{Sym}_{b-1}$.
In particular, $\psi_{3,1}$ corresponds to the trivial representation of the compact real form of $\GSp(4)$.
The elliptic endoscopic group $(\GL(2)\times\GL(2))/\{(x,x^{-1})\mid x\in\mathbb{G}_m\}$ comes from the centralizer of $\diag(1,-1,1,-1)$ in $\GSp(4,\C)$.

\subsubsection{Discrete series of $G_2(\R)$}\label{sec:g2str}

Recall the notations $Q$, $\alpha$, $\beta$, and $T$ for $G_2$ 
given in \cref{sec:reductive}.
Choose an element $(a,b)\in (2 \Z)^{\oplus 2}$.
We denote by $\tilde \psi_{a,b}:W_\R\to G_2(\C)$ a tempered discrete relevant $L$-parameter of $\GSp(4,\R)$ which is of the form
\[
\tilde \psi_{a,b}(z)=(  (z/\overline{z})^{a/2} , (z/\overline{z})^{b/2} )\in T(\C)\subset G_2(\C) \subset \SO(Q,\C)
\]
and $\tilde\psi_{a,b}(j)=Q-E_{44}$.
Note that we have $\tilde\psi_{a,b}\cong\tilde\psi_{-a,a+b}$ and $\tilde\psi_{a,b}\cong\tilde\psi_{-a-b, b}$ up to conjugation.
When $a>b>0$, the $L$-parameter $\tilde\psi_{a,b}$ corresponds to an $L$-packet of a discrete series for $G_2(\R)$. 
In particular, $\tilde\psi_{4,2}$ corresponds to the trivial representation of the compact real form of $G_2$, and also corresponds to the quaternionic discrete series of $G_2(\R)$ whose minimal $(\SU(2)\times\SU(2))/\{\pm 1\}$-type is $\mathrm{Sym}_4\otimes 1$.
The elliptic endoscopic group $\SO(2,2)$ of $G_2$ arises from the centralizer of $(1,-1)$ in $G_2(\C)$.
For odd integers $c$ and $d$, the $L$-embedding of $\psi_c\oplus\psi_d$ for $\SO(4,\C)=(\SL(2,\C)\times\SL(2,\C))/\{\pm 1\}$ is $\tilde\psi_{c+d,c-d}(\cong \tilde\psi_{d-c,2c} \cong \tilde\psi_{2c,d-c}\cong\cdots)$.

\subsubsection{Satake parameters}

Now we review Satake transforms and parameters of spherical representations.
Let $F$ be a nonarchimedean local field of characteristic 0, and $G$ a connected reductive group over $F$.
For simplicity, we write $G$ for the group $G(F)$ of $F$-rational points in $G$.
Suppose that $G$ is unramified.
Then $G$ has a hyperspecial maximal compact subgroup $K$ and a Borel subgroup $B$ containing a maximal $F$-torus $T$.  We have an Iwasawa decomposition $G=BK$ and a Cartan decomposition $G=KTK$.
Also, $T\cap K$ is the maximal compact subgroup of $T$.
Denote by $X^\un(T)$ the group of unramified characters of $T$, i.e., characters 
trivial on $T\cap K$.
We have $T/(T\cap K) \cong \Z^n$, where $n$ is the rank of $G$.
The Weyl group $W:=N_G(T)/T$ acts on $X^\un(T)$ via $(w\chi)(t)=\chi(w^{-1}tw)$ for $t\in T$, $w\in W$.  As usual, $N_G(T)$ denotes the normalizer of $T$ in $G$.

We write $\cH(G,K)$ for the spherical Hecke algebra of $G$ and $K$, that is,
\[
\cH(G,K):=\{  f\in C_c^\inf(G) \mid f(h x h')=f(x) \quad \text{for } x\in G , \;\; h,h'\in K   \}.
\]
The convolution $(f_1*f_2)(x):=\int_G f_1(xy^{-1}) f_2(y) \, \d y$ defines a product on $\cH(G,K)$, where $\d y$ is the Haar measure on $G$ normalized by $\int_K \d y=1$.
The Satake transform $\scS:\cH(G,K)\to \cH(T,T\cap K)^W$ is a $\C$-algebra isomorphism defined by
\[
\scS f(t):=\delta_B(t)^{1/2}\int_N f(t u)\, \d u = \delta_B(t)^{-1/2}\int_N f(ut)\, \d u  \qquad (f\in\cH(G,K))
\]
where $\delta_B$ is the modulus function of $B$, $N$ is the unipotent radical of $B$, and $\d u$ denotes the Haar measure on $N$ satisfying $\int_{N\cap K}\d u=1$.
Let  $\omega_\chi:\cH(G,K)\to \C$ be the linear map defined by
\[
\omega_\chi(f):=\int_T \scS f(t) \, \chi(t) \, \d t \qquad (f\in\cH(G,K))
\]
where $\d t$ is a Haar measure on $T$ such that $\int_{T\cap K}\d t=1$.
It is known that, for any nonzero $\C$-algebra homomorphism $\lambda:\cH(G,K)\to \C$, there is a unique character $\chi\in X^\un(T)/W$ such that $\lambda=\omega_\chi$. 
We denote by $\mathrm{Irr}^K(G)$ the set of equivalence classes of $K$-spherical irreducible representations of $G$.
Each $\pi\in \mathrm{Irr}^K(G)$ defines a $\C$-algebra homomorphism $\lambda_\pi :\cH(G,K)\to \C$ by $\Tr \, \pi(f)$, so we may associate a character $\chi\in X^\un(T)/W$
to $\pi$ via $\omega_\chi=\lambda_\pi$.
This gives a bijection from $\mathrm{Irr}^K(G)$ to $X^\un(T)/W$.

There is a maximal $F$-split torus $A$ in $T$.
The inclusion mapping $A\to T$ induces isomorphisms $A/(A\cap K) \cong T/(T\cap K)$ and $X^\un(A)\cong X^\un(T)$.
Choose a basis $(a_1,a_2,\dots,a_n)$ in $A/A\cap K$.
For each $\chi\in X^\un(A)$, one has an $n$-tuple
\[
\alpha=(\chi(a_1),\chi(a_2),\dots,\chi(a_n)) .
\]
Write $\omega_\alpha = \omega_\chi$.
The $n$-tuple $\alpha$ determines a unique spherical representation 
corresponding to $\omega_\alpha$ as above, and $\alpha$ is called the 
Satake parameter.  By abuse of notation, we also denote the spherical
representation by $\omega_\alpha$.
Note that in the case of $\alpha_0=(\delta_B(a_1)^{1/2},\dots,\delta_B(a_n)^{1/2})$,  
$\omega_{\alpha_0}$ corresponds to the trivial representation, that is,
\begin{equation}\label{eq:trivial}
\omega_{\alpha_0}(f)=\int_G f(g) \, \d g.
\end{equation}

Throughout this section, to consider mod $p$ congruences between Hecke eigenvalues we always assume that, for each double coset $KaK$ $(a\in A)$, the Hecke eigenvalue $\omega_\alpha(f_{KaK})$ is an algebraic integer for the characteristic functions $f_{KaK}$ of $KaK$.
Here, mod $p$ is used in the same sense as above (mod $\frakp$ for a suitable 
prime $\frakp$ above $p$).


\subsection{Local congruence lifts} \label{sec:32}

Now we consider local lifts of Eisenstein congruences.
Let $F$ denote a nonarchimedean field of characteristic 0, and $|\;|$ the normalized valuation of $F$.
The integer ring of $F$ is given by $\fO:=\{ x\in F\mid |x|\leq 1\}$.
We choose a prime element $\varpi$ of $\fO$ and put $q:=|\fO/\varpi\fO|$, that is, $|\varpi|=q^{-1}$.

\subsubsection{Symmetric functions}

We shall review some symmetric function theory.
For details, we refer to \cite{Macdonald}.

A partition means any sequence $\lambda=(\lambda_1,\lambda_2,\dots)$ such that $\lambda_j\in\Z_{\geq 0}$, $\lambda_j\geq \lambda_{j+1}$, and $l(\lambda):=|\{ j \in \Z_{>0} \mid \lambda_j>0\}|$ is finite.
We call $l(\lambda)$ the length of $\lambda$.
We write $\lambda \geq \mu$ if $\lambda_1+\cdots+\lambda_j\geq \mu_1+\cdots+\mu_j$ holds for all $j$, where $\lambda=(\lambda_1,\lambda_2,\dots)$ and $\mu=(\mu_1,\mu_2,\dots)$.
This relation is a partial ordering on partitions.
For each partition $\lambda$, we denote by $\lambda^c=(\lambda_1^c,\lambda_2^c,\dots)$ the conjugate of $\lambda$, i.e., $\lambda_j^c=|\{ k\in\Z_{>0} \mid \lambda_k\geq j\}|$.
Set $|\lambda|:=\lambda_1+\lambda_2+\cdots$.
Notice that $\lambda\geq \mu$ does not hold if $|\lambda|<|\mu|$.
If a partition $\lambda$ satisfies $l(\lambda)\leq n$, then we write $\lambda=(\lambda_1,\dots,\lambda_n)$ as a $n$-tuple.

Let $S_n$ denote the symmetric group of degree $n$.
We write $m_\lambda$ for the monomial symmetric function of $\lambda$, i.e.,
\[
m_\lambda(x_1,\dots,x_n) := \sum_{w\in S_n/S_n^\lambda }x_{w(1)}^{\lambda_1}\cdots x_{w(n)}^{\lambda_n}  , 
\]
where $S_n^\lambda$ denotes the stabilizer of $\lambda$ in $S_n$. 
We also write $e_r$ for the $r$-th elementary symmetric function, i.e.,
\[
e_r(x_1,\dots,x_n):=\sum_{1\leq m_1<m_2<\cdots<m_r\leq n} x_{m_1} x_{m_2}\cdots x_{m_r}.
\]
Obviously one has the relation $e_r=m_{(1^r,0^{n-r})}$. 

Let $\Lambda_n:=\Z[x_1,\dots,x_n]^{S_n}$.
It is obvious that $m_\lambda$ (for $l(\lambda)\leq n)$ form a $\Z$-basis of $\Lambda_n$.
By \cite[(2.3) on p.20]{Macdonald}, for a partition $\lambda=(\lambda_1,\dots,\lambda_n)$, there exist non-negative integers $a_{\lambda\mu}$ such that
\begin{equation}\label{eq:elemono}
e_{\lambda_1^c}\cdots e_{\lambda_r^c} = m_\lambda + \sum_{\mu<\lambda} a_{\lambda \mu} m_\mu
\end{equation}
where $r=\lambda_1$ and $|\mu|=|\lambda|$.
Hence, we also have $\Lambda_n=\Z[e_1,\dots,e_n]$.

\subsubsection{$\GL(n)$}\label{sec:GL(n)}

Set
\[
G:=\GL(n,F), \quad G^+:=G\cap M(n,\fO), \quad \text{and} \quad  K:=\GL(n,\fO).
\]
Let $T=T_n$ denote the usual maximal split torus in $G$ consiting of diagonal matrices, and $B=TN$ the standard Borel subgroup consisting of upper triangular matrices with
unipotent radical $N$.
We denote by $X_j$ the characteristic function of $\{ \diag(x_1,\dots,x_n) \in T \mid x_k\in\fO^\times \; (\forall k\neq j) $, $x_j\in\varpi\fO^\times  \}$.
For $u \in \Z$, $X_j^u$ denotes the characteristic function of $\{ \diag(x_1,\dots,x_n) \in T \mid x_k\in\fO^\times \; (\forall k\neq j) $, $x_j\in\varpi^u\fO^\times  \}$.
The Satake transform gives an isomorphism between $\cH(G^+,K)$ and $\C[X_1,\dots, X_n]^{S_n}$.

For each partition $\lambda$ with $l(\lambda)\leq n$, we set
\[
\varpi^\lambda:=\diag(\varpi^{\lambda_1},\varpi^{\lambda_2},\dots,\varpi^{\lambda_n}).
\]
The set $\{ \varpi^\lambda \mid \lambda$ is a partition with $l(\lambda)\leq n\}$ is a system of representative elements of $K\bsl G^+/K$.
Let $c_\lambda$ denote the characteristic function of $K\varpi^\lambda K$.
The functions $c_\lambda$ $(l(\lambda) \leq n)$ form a basis of $\cH(G^+,K)$ as a vector space over $\C$.
We set
\begin{equation}\label{eq:rhotrivial}
\rho=\rho_n:=\frac{1}{2}(n-1,n-3,\dots,1-n) ,
\end{equation}
\[
\langle (t_1,\dots,t_n),(t_1',\dots,t_n')\rangle:=t_1t_1'+\cdots+t_n t_n'.
\]
We write $\widehat{c}_\lambda$ for the Satake transform $\scS(c_\lambda)$ of $c_\lambda$, that is, $\widehat{c}_\lambda(X)$ is in $\C[X]^{S_n}$ where $X=(X_1,X_2,\dots,X_n)$.
It is known that $\widehat{c}_\lambda$ is described by the Hall--Littlewood symmetric function $P_\lambda(x_1,\dots,x_n;t)$, namely $\widehat{c}_\lambda(X)=q^{\langle\lambda,\rho\rangle}P_\lambda(X;q^{-1})$ (see \cite[(3.3) on p.299]{Macdonald}).
The following fact is well known.
\begin{lem}\label{lem:subdeterminant}
For any partition $\lambda=(\lambda_1,\dots,\lambda_n)$ and any $g\in G$, the element $g$ belongs to $K\varpi^\lambda K$ if and only if $d_r(g)=q^{-\lambda_n-\lambda_{n-1}-\cdots-\lambda_{n-r+1}}$ for all $1\leq r\leq n$, where $d_r(g)$ denotes the maximum of the valuations of all $r\times r$ minors of $g$.
\end{lem}
This lemma leads to the equality
\begin{equation}\label{eq:maxorder201921}
\int_N  c_\lambda( u\varpi^\mu    ) \, \d u=\begin{cases} 0 & \text{if $|\mu|\neq |\lambda|$, or  if $\mu \not \leq \lambda$ and $|\mu|=|\lambda|$,} \\  1 & \text{if $\mu = \lambda$.} \end{cases}
\end{equation}
\begin{lem}\label{lem:subringGL}
For any partition $\lambda$, $l(\lambda)\leq n$, the symmetric function $\widehat{c}_\lambda$ belongs to the subalgebra $\Z[Z_1,\dots,Z_n]$, where $Z_r:=\widehat{c}_{(1^r,0^{n-r})}=q^{\frac{(n-r)r}{2}} e_r$ $(1\leq r\leq n)$.
\end{lem}
\begin{proof}
It follows from \eqref{eq:maxorder201921} and the symmetry of $\widehat{c}_\lambda$ that, for some $a_{1,\lambda \mu} \in \Z_{\ge 0}$, we have
\[
\widehat{c}_\lambda=q^{\langle\lambda,\rho\rangle}m_\lambda + \sum_{\mu<\lambda} a_{1,\lambda \mu} \,  q^{\langle\mu,\rho\rangle}m_\mu
\]
where $|\mu|=|\lambda|$.
Note that  \cref{lem:subdeterminant} ensures that, if $\mu>\lambda$, then such $m_\mu$ does not appear in the above equality.
We have $\langle\lambda-\mu,\rho\rangle\in\Z_{>0}$ for any $\mu$ satisfying $\mu<\lambda$
 and $|\mu|=|\lambda|$, and $\langle\lambda,\rho\rangle=\sum_{j=1}^{\lambda_1}\langle(1^{\lambda_j^c},0^{n-\lambda_j^c}),\rho\rangle$.
Hence, this assertion follows from \eqref{eq:elemono}.
\end{proof}

Take the basis of $T/(T\cap K)$ given by
\[
\diag(\varpi,1,\dots,1),\;\; \diag(1,\varpi,\dots,1), \; \dots,  \; \diag(1,1,\dots,\varpi).
\]
As above, write $\omega_\alpha$ for the $K$-spherical representation of $G$ with the Satake parameter $\alpha=(\alpha_1,\alpha_2,\dots,\alpha_n)$.
Then, one has
\[
\widehat{c}_\lambda(\alpha)=\omega_\alpha(c_\lambda).
\]
Set
\[
\mathbb{L}=\mathbb{L}_n:=\rho+\{  (l_1,l_2,\dots,l_n)\in\Z^{\oplus n} \mid l_1\geq l_2 \geq \cdots \geq l_n \}.
\]
Choose an $n$-tuple $k=(k_1,k_2,\dots,k_n)\in \mathbb{L}$.
Here $k$ corresponds to the $L$-parameter $\C^\times\ni z\mapsto \diag( (z/\overline{z})^{k_1},(z/\overline{z})^{k_2},\dots,(z/\overline{z})^{k_n}) \in \GL(n,\C)$ for the discrete series of $\UU_{\C/\R}(n,\R)$.
We set
\[
q^k:=(q^{k_1},q^{k_2},\dots,q^{k_n}).
\]
Since such an $L$-parameter is associated to the local component of a global
representation induced from the Borel, we may think of $\widehat c_\lambda(q^k)$ 
as a local Eisenstein Hecke eigenvalue.
One can show $\widehat c_\lambda(q^k)$ is in $\Z[q^{-1}]$ using  \cref{lem:subringGL}.
Since $\omega_{q^\rho}$ is the trivial representation, one gets
\begin{equation}\label{eq:congvol}
 \widehat c_\lambda(q^\rho) = \vol(K\varpi^\lambda K) .
\end{equation}
If there exists a prime number $p$ such that $(p,q)=1$ and
\begin{equation}\label{eq:eigln}
\widehat c_\lambda (\alpha) \equiv \widehat c_\lambda(q^k) \mod p
\end{equation}
holds for any partition $\lambda$, then we say that the spherical representation 
$\omega_\alpha$ satisfies the Eisenstein congruence mod $p$ for the $L$-parameter (associated to) $k$.
Here, as in the global case, what we mean by this congruence notation
is that both quantities lie in some ring of integers $\mathcal O$ and they are
congruent modulo a prime ideal of $\mathcal O$ above $p$.
It follows from \eqref{eq:congvol} that $\omega_\alpha$ satisfies the Eisenstein congruence mod $p$ for the $L$-parameter $\rho$ if and only if
\[
\widehat c_\lambda(\alpha) \equiv \vol(K\varpi^\lambda K) \mod p
\]
holds for any partition $\lambda$.

\begin{prop}\label{pr:glcong1}
Let $\omega_{\alpha}$ (resp.\ $\omega_{\alpha'}$) be the spherical representation of $\GL(n,F)$ (resp.\ $\GL(n',F)$) with the Satake parameter $\alpha=(\alpha_1,\dots,\alpha_n)$ (resp.\ $\alpha'=(\alpha_1',\dots,\alpha_{n'}')$).
Assume that $\omega_{\alpha}$ (resp.\ $\omega_{\alpha'}$) satisfies the Eisenstein congruence mod $p$ for $k=(k_1,\dots,k_n)$ (resp.\ $k'=(k_1',\dots,k_{n'}')$).
For $m = n, n'$, choose $u_m \in \Z$, set
\[ 
\nu_m:= \begin{cases} u_m & \text{if $m$ is even,} \\ 1/2+u_m & \text{if $m$ is odd,}
 \end{cases}
\]
and let $\beta_m \in \C^\times$ such that $q^{-\nu_m}\beta_m$ is an
algebraic integer with $q^{-\nu_m}\beta_m\equiv 1 \bmod p$. 
We also set $\alpha \beta_{n'}=(\alpha_1 \beta_{n'},\dots,\alpha_n \beta_{n'})$, $\alpha' \beta_{n}=(\alpha_1' \beta_{n},\dots,\alpha_{n'}' \beta_{n})$, $k+\nu_{n'}=(k_1+\nu_{n'},\dots,k_n+\nu_{n'})$, and $k'+\nu_n=(k_1'+\nu_n,\dots,k_{n'}'+\nu_n)$.
Suppose that $(k+\nu_{n'},k'+\nu_n)$ is in $\mathbb{L}_{n+n'}$ by changing its ordering.
Then the spherical representation $\omega_{(\alpha \beta_{n'},\alpha' \beta_{n})}$ of $\GL(n+n',F)$ satisfies the Eisenstein congruence mod $p$ for $L$-parameter $(k+\nu_{n'},k'+\nu_n)$.
\end{prop}

This proposition is relevant for local endoscopic lifts of unitary groups 
at a split finite place.

\begin{proof}
By using  \cref{lem:subringGL}, it suffices to prove the Eisenstein congruence of $\widehat c_{(1^r,0^{n+n'-r})}(\alpha \beta_{n'},\alpha' \beta_{n})$ for
 $1\leq r\leq n+n'$.
It is obvious that  there exist integers $a_{2,j,j'}$ such that
\begin{multline*}
q^{\frac{(n+n'-r)r}{2}}e_r(\alpha \beta_{n'},\alpha' \beta_{n})=\\
\sum_{j+j'=r}a_{2,j,j'}q^{\frac{(n-j)j'+(n'-j')j}{2}+\nu_{n'}j+\nu_{n}j'} (q^{-\nu_{n'}} \beta_{n'})^j (q^{-\nu_{n}} \beta_{n})^{j'} \,  q^{\frac{(n-j)j}{2}}e_j(\alpha)\,  q^{\frac{(n'-j')j'}{2}}e_{j'}(\alpha') .
\end{multline*}
Hence, the statement follows from the facts that $\frac{(n-j)j'+(n'-j')j}{2}+\nu_{n'}j+\nu_{n}j'\in\Z$, $q^{\frac{(n-j)j}{2}}e_j(\alpha) \equiv q^{\frac{(n-j)j}{2}}e_r(q^k) \bmod p$, and $q^{\frac{(n'-j')j'}{2}}e_{j'}(\alpha') \equiv q^{\frac{(n'-j')j'}{2}}e_r(q^{k'}) \bmod p$.
\end{proof}

\begin{prop}\label{pr:glnsymcong}
Let $\omega_{\alpha}$ be the spherical representation of $\GL(2,F)$ with the Satake parameter $\alpha=(\alpha_1,\alpha_2)$.
Assume that $\omega_{\alpha}$ satisfies the Eisenstein congruence mod $p$ for $k=(k_1,k_2)$.
The spherical representation of $\GL(m+1,F)$ with the Satake parameter $(\alpha_1^m,\alpha_1^{m-1}\alpha_2,\dots,\alpha_2^{m})$ satisfies the Eisenstein congruence mod $p$ for $L$-parameter $(m k_1, (m-1)k_1+k_2 , \dots , m k_2)$.
Furthermore, when $k_1-k_2\ge m$, the spherical representation of $\GL(2m,F)$ with the Satake parameter $(\alpha_1 q^{\rho_m}, \alpha_2 q^{\rho_m})$ satisfies the Eisenstein congruence mod $p$ for $L$-parameter $((k_1,\dots,k_1)+\rho_m, (k_2,\dots,k_2)+\rho_m)$.
\end{prop}
\begin{proof}
The assertions follow from Lemma \ref{lem:subringGL}, since it is easy to prove the congruences for $Z_r$.
\end{proof}

\subsubsection{$\UU_{E/F}(N)$}\label{sec:U(N)}

Let $E$ be an unramified quadratic extension of $F$.
Then $G=\UU_{E/F}(N,F)$ is unramified.
Assume that $K$ is a hyperspecial maximal compact subgroup of $G$.
Let $T$ denote the maximal $F$-torus consisting of diagonal matrices.
The torus $T$ includes a maximal $F$-split torus $A$.
We set $n:=[N/2]$, so $N=2n$ or $2n+1$.
Then $A$ is isomorphic to $\mathbb{G}_m^n$ over $F$.
Write $X_j$ for the characteristic function of $\{ \diag(x_1,\dots,x_n,x_n^{-1},\dots,x_1^{-1}) \in A$ (resp.\ $\diag(x_1,\dots,x_n,1,x_n^{-1},\dots,x_1^{-1}) \in A$) $\mid x_k\in\fO^\times \; (\forall k\neq j) $, $x_j\in\varpi\fO^\times  \}$ if $N=2n$ (resp.\ $N=2n+1$).
The Satake transform gives an isomorphism between $\cH(G,K)$ and $\C[X_1^{\pm 1},X_2^{\pm 1},\dots, X_n^{\pm 1}]^{W_n}$. Here $W_n=S_n\rtimes \{\pm 1\}^n$ where
the $-1$ in the $j$-th component acts on $X_j$ by $X_j\mapsto X_j^{-1}$.
Hence we obtain the following isomorphism
\begin{equation*}
\cH(G,K) \cong \C[Y_1,Y_2,\dots,Y_n]^{S_n}
\end{equation*}
where $Y_j=X_j+X_j^{-1}$.
Here we have used
\begin{equation}\label{eq:polyunitary}
X_j^k+X_j^{-k}= (X_j+X_j^{-1})^k -\sum_{u=1}^{k-1} \begin{pmatrix} k \\ u \end{pmatrix} X_j^{k-2u}.
\end{equation}

For each partition $\lambda=(\lambda_1,\lambda_2,\dots,\lambda_n)$ with $l(\lambda)\leq n$, we set
\[
\varpi^\lambda:=\begin{cases}
\diag(\varpi^{\lambda_1},\varpi^{\lambda_2},\dots,\varpi^{\lambda_n},\varpi^{-\lambda_n},\dots, \varpi^{-\lambda_1}) & \text{if $N=2n$,} \\
\diag(\varpi^{\lambda_1},\varpi^{\lambda_2},\dots,\varpi^{\lambda_n},1,\varpi^{-\lambda_n},\dots, \varpi^{-\lambda_1}) & \text{if $N=2n+1$.} 
\end{cases}
\]
The set $\{ \varpi^\lambda \mid  l(\lambda)\leq n\}$ is a system of representative elements of $K\bsl G/K$.
Let $c_\lambda$ denote the characteristic function of $K\varpi^\lambda K$.
The functions $c_\lambda$ with $l(\lambda)\leq n$ form a basis of $\cH(G,K)$.
We set
\[
\widehat{c}_\lambda:=\scS(c_\lambda), \quad 
\varrho=\varrho_N:=\begin{cases}
(2n-1,2n-3,\dots,1) & \text{if $N=2n$,} \\
(2n,2n-2,\dots,2) & \text{if $N=2n+1$.} 
\end{cases}
\]
\begin{lem}\label{lem:unitsym2}
Set $X=(X_1,\dots,X_n)$ and $Y=(Y_1,\dots,Y_n)$ $(Y_j=X_j+X_j^{-1})$.
For any partition $\lambda$, $l(\lambda)\leq n$, the function $\widehat{c}_\lambda$ is in $\Z[\tilde Z_1,\dots,\tilde Z_n]$, where
\[
\tilde Z_r(X):=q^{Nr-r^2} \, e_r(Y)\in \C[X_1^{\pm 1},X_2^{\pm 1},\dots, X_n^{\pm 1}]^{W_n} \quad (1\leq r\leq n).
\]
We note $Nr-r^2=\langle (1^r,0^{n-r}),\varrho\rangle$.
\end{lem}
\begin{proof}
For each $1\leq r\leq n$, one can prove
\begin{equation}\label{eq:20190511}
\widehat{c}_{(1^r,0^{n-r})}(X)\in \tilde Z_r(X)+\Z[\tilde Z_1,\dots,\tilde Z_{r-1}]
\end{equation}
by \cref{lem:subdeterminant}, $\scS(c_\lambda)\in\C[Y_1,Y_2,\dots,Y_n]^{S_n}$ and a direct calculation for representative elements of left cosets of $K\varpi^\lambda K$.
Furthermore, for any partition $\lambda$, $l(\lambda)\leq n$,  using \cref{lem:subdeterminant} and \eqref{eq:polyunitary}, one gets
\[
\widehat{c}_\lambda(X)=q^{\langle \lambda,\varrho\rangle} m_\lambda(Y) +\sum_{\mu<\lambda, \, |\mu|=|\lambda|} a_{3,\lambda \mu}q^{\langle \mu,\varrho\rangle} m_\mu(Y) +\sum_{\mu<\lambda, \, |\mu|\leq |\lambda|-2} a_{4,\lambda \mu}q^{\langle \mu,\varrho\rangle} m_\mu(Y)
\]
for some $a_{3,\lambda\mu}\in\Z_{\geq 0}$ and $a_{4,\lambda \mu}\in\Z$.
Hence the proof is completed by \eqref{eq:elemono} and \eqref{eq:20190511}.
\end{proof}

Set $a_j:=\diag(\overbrace{1,\dots,1}^{j-1},\varpi,1,\dots,1,\varpi^{-1},\overbrace{1,\dots,1}^{j-1}) \in A$.
Then $a_1,a_2,\dots,a_n$ form a basis of $A/A\cap K$.
For the Satake parameter $\alpha=(\alpha_1,\dots,\alpha_n)$, let $\omega_\alpha$ denote the associated irreducible spherical representation of $G$.
We have $\widehat c_\lambda(\alpha)=\omega_\alpha(c_\lambda)$. 
We shall use the same notations $\rho=\rho_N$, $\mathbb{L}=\mathbb{L}_N$, and $q^k$ as in the previous subsection.
Choose an $N$-tuple $k=(k_1,k_2,\dots,k_N)\in\mathbb{L}$ and a prime number $p$ such that $(p,q)=1$.
We set
\[
\tilde k:=(k_1-k_N,k_2-k_{N-1},\dots,k_n-k_{N-n+1}),
\]
\[
k':=(k_1+k_N,k_2+k_{N-1},\dots,k_n+k_{N-n+1}).
\]
For $t=(t_1,\dots,t_n)$ and $s=(s_1,\dots,s_n)$, we set
\[
(-1)^t q^s:=((-1)^{t_1}q^{s_1},\dots,(-1)^{t_n}q^{s_n}).
\]
Then, defining $\tilde\rho$ and $\rho'$ in the same manner, one has $\varrho=\tilde\rho$, $\rho'=(0,\dots,0)$, and
\begin{equation} \label{eq:congvol-Un}
\widehat c_\lambda((-1)^{\rho'}q^{\tilde\rho}) = \vol(K\varpi^\lambda K).
\end{equation}
We say that the spherical representation $\omega_\alpha$ satisfies the Eisenstein congruence mod $p$ for the $L$-parameter (associated to) $k$ if
\begin{equation}\label{eq:eigunitary}
\widehat c_\lambda (\alpha) \equiv \widehat c_\lambda((-1)^{k'}q^{\tilde k}) \mod p 
\end{equation}
holds for any partition $\lambda$.
In particular, $\omega_\alpha$ satisfies the Eisenstein congruence mod $p$ for the parameter $\rho$ if and only if
\[
\widehat c_\lambda(\alpha) \equiv \vol(K\varpi^\lambda K) \mod p
\]
holds for any partition $\lambda$.
Notice that this local Eisenstein congruence was defined from the viewpoint of the $L$-parameter of $\omega_\alpha$ over $F$ (cf. \cite[Section 3.2.6]{Minguez}).

\begin{prop}\label{pr:unicong1}
Let $\omega_{\alpha}$ (resp.\ $\omega_{\alpha'}$) be the spherical representation of $\UU_{E/F}(N,F)$ (resp.\ $\UU_{E/F}(N',F)$) with the Satake parameter $\alpha=(\alpha_1,\dots,\alpha_n)$ (resp.\ $\alpha'=(\alpha_1',\dots,\alpha_{n'}')$), where $n=[N/2]$ (resp.\ $n'=[N'/2]$).
Assume that $\omega_{\alpha}$ (resp.\ $\omega_{\alpha'}$) satisfies the Eisenstein congruence mod $p$ for $k=(k_1,\dots,k_N)$ (resp.\ $k'=(k_1',\dots,k_{N'}')$).
For a natural number $m=N$, $N'$, choose an integer $u_m$, and set
\[
\mu_m:=\begin{cases}-1 & \text{if $m$ is odd}, \\ 1 & \text{if $m$ is even}, \end{cases}   \qquad  \nu_m:=\begin{cases}1/2+u_m & \text{if $m$ is odd}, \\ u_m & \text{if $m$ is even}. \end{cases}
\]
Further, we put
\[
\mu_m(\alpha_1,\dots,\alpha_n):=(\mu_m\alpha_1,\dots,\mu_m\alpha_n), \quad (k_1,\dots,k_n)+\nu_m:=(k_1+\nu_m,\dots,k_n+\nu_m).
\]

\vspace{1mm}
\noindent
{\bf Case 1.}
Suppose $N$ and $N'$ are not both odd.
Suppose that there exists a new sequence $r=(r_1,r_2,\dots,r_{N+N'})\in \mathbb{L}_{N+N'}$ such that $(r_1,r_2,\dots,r_{N+N'})$ is obtained from $(k+\nu_{N'},k'+\nu_N)$ by changing the ordering, for any $1\leq j\leq [(N+N')/2]$ we have $(r_j,r_{N+N'-j+1})=(k_u+\nu_{N'},k_{N-u+1}+\nu_{N'})$ for some $1\leq u\leq [N/2]$ or $(r_j,r_{N+N'-j+1})=(k'_{u'}+\nu_N,k_{N'-u'+1}'+\nu_N)$ for some $1\leq u'\leq [N'/2]$.
Under these assumptions, the spherical representation $\omega_{(\mu_{N'}\alpha,\mu_N\alpha')}$ of $\UU_{E/F}(N+N',F)$ satisfies the Eisenstein congruence mod $p$ for $L$-parameter $r$.

\vspace{1mm}
\noindent
{\bf Case 2.}
Suppose $N$ and $N'$ are both odd.
Suppose that there exists a new sequence $r=(r_1,r_2,\dots,r_{N+N'})\in \mathbb{L}_{N+N'}$ such that $(r_1,r_2,\dots,r_{N+N'})$ is obtained from $(k+\nu_{N'},k'+\nu_N)$ by changing the ordering, for any $1\leq j\leq (N+N')/2-1$ we have $(r_j,r_{N+N'-j+1})=(k_u+\nu_{N'},k_{N-u+1}+\nu_{N'})$ for some $1\leq u\leq [N/2]$ or $(r_j,r_{N+N'-j+1})=(k'_{u'}+\nu_N,k_{N'-u'+1}'+\nu_N)$ for some $1\leq u'\leq [N'/2]$, and we have $(r_{(N+N')/2},r_{(N+N')/2+1})=(k_{(N+1)/2}+\nu_{N'},k_{(N'+1)/2}+\nu_{N})$.
Furthermore, we suppose
\[
\widehat c_{(1)} (-1) \equiv \widehat c_{(1)}((-1)^{f'}q^{\tilde f}) \mod p \quad \text{where } f=(k_{(N+1)/2}+\nu_{N'},k_{(N'+1)/2}+\nu_{N}) .
\]
Under these assumptions, the spherical representation $\omega_{(\mu_{N'}\alpha,\mu_N\alpha',-1)}$ of $\UU_{E/F}(N+N',F)$ satisfies the Eisenstein congruence mod $p$ for $L$-parameter $r$.
\end{prop}
\begin{proof}
By  \cref{lem:unitsym2}, it is sufficient to prove the Eisenstein congruence of $\tilde Z_r(\mu_{N'}\alpha,\mu_{N}\alpha')$ $(1\leq \forall r\leq n+n')$ and $\tilde Z_r(\mu_{N'}\alpha,\mu_{N}\alpha',-1)$ $(1\leq \forall r\leq n+n'+1)$.
Hence, this assertion follows from the same argument as in the proof of  \cref{pr:glcong1}.
Notice that $(-1)^{2\nu_m}=\mu_m$ holds.
\end{proof}

\begin{prop}\label{pr:unisymcong}
Let $\omega_{\alpha}$ be the spherical representation of $\UU_{E/F}(2,F)$ with the Satake parameter $\alpha=(\alpha_1)$.
Assume that $\omega_{\alpha}$ satisfies the Eisenstein congruence mod $p$ for $k=(k_1,k_2)$.
Choose a natural number $m$, and we put $\delta_m:=2$ if $m$ is even, and $\delta_m:=1$ if $m$ is odd.
The spherical representation of $\UU_{E/F}(m+1,F)$ with the Satake parameter $(\alpha_1^m,\alpha_1^{m-2},\dots,\alpha_1^{\delta_m})$ satisfies the Eisenstein congruence mod $p$ for $L$-parameter $(m k_1, (m-1)k_1+k_2 , \dots , m k_2)$.
Further, when $k_1-k_2\ge m$, the spherical representation of $\UU_{E/F}(2m,F)$ with the Satake parameter $(\alpha_1 q^{\rho_m})$ satisfies the Eisenstein congruence mod $p$ for $L$-parameter $((k_1,\dots,k_1)+\rho_m, (k_2,\dots,k_2)+\rho_m)$.
\end{prop}
\begin{proof}
If we put $k=(m k_1, (m-1)k_1+k_2 , \dots , m k_2)$ (resp. $(k_1+\rho_m,k_2+\rho_m)$), then $k'=(k_1+k_2)(m,m,\dots,m)$ and $\tilde k=(k_1-k_2) (m,m-2,\dots,\delta_m)$ (resp. $k'=(k_1+k_2)(1,1,\dots,1)$ and $\tilde k=(k_1-k_2-m) (1,1,\dots,1)+\varrho_{2m}$).
Since we have $x^l+x^{-l}=(x+x^{-1})^l+c_1(x+x^{-1})^{l-2}+\cdots$ for some $c_1, \ldots\in\Z$, we obtain the assertions.
\end{proof}

\subsubsection{$\GL(2)$ and $\PGL(2)$}\label{sec:GL(2)}

Let $G:=\GL(2,F)$, $K:=G(\fO)$, and let $T$ denote the maximal split torus consisting of diagonal matrices in $G$.
We choose $\diag(\varpi,1)$ and $\diag(1,\varpi)$ as a basis of the group $T/T\cap K$.
For each spherical representation $\omega_\alpha$ with the Satake parameter $\alpha=(\alpha_1,\alpha_2)$, we say that $\omega_\alpha$ satisfies the Eisenstein congruence mod $p$ for weight $k$ and character $(\chi_1,\chi_2)$ if
\begin{equation}\label{eq:eigl2}
q^{(k-1)/2}(\alpha_1+\alpha_2)\equiv \chi_1(\varpi)q^{k-1}+\chi_2(\varpi) \mod p ,\quad  \alpha_1\alpha_2\equiv \chi_1(\varpi)\chi_2(\varpi) \mod p
\end{equation}
hold, where $\chi_1$ and $\chi_2$ are unramified characters on $F^\times$ and weight $k\in \Z_{\geq 1}$ means the $L$-parameter $\psi_{(k-1)/2}$ of $\GL(2,\R)$ cf. \eqref{eq:gl2r}.
Here, we assume that $q^{(k-1)/2}(\alpha_1+\alpha_2)$ and $\alpha_1\alpha_2$ are algebraic as always.
Note that \eqref{eq:eigl2} is equivalent to \eqref{eq:eigln}, $n=2$, $(k_1,k_2)=((k-1)/2,-(k-1)/2)$.
In particular, if $k=2$ and $\chi_1$, $\chi_2$ are trivial, then \eqref{eq:eigl2} means
\begin{equation}\label{eq:eigl2tr}
q^{1/2}(\alpha_1+\alpha_2)\equiv \vol(K\diag(1,\varpi)K) \mod p ,\quad  \alpha_1\alpha_2\equiv \vol(K) \mod p .
\end{equation}

For the group $\PGL(2,F)$, Eisenstein congruences can be defined by putting $\alpha_2=\alpha_1^{-1}$ and supposing that weight $k$ is even in the situation of $G$.
We choose $\diag(\varpi,1)$ as a generator on $T/(T\cap K)Z$, where $Z$ denotes the center of $G$.
For the spherical representation $\omega_\alpha$ of $\PGL(2,F)$ with the Satake parameter $\alpha$, we say that $\omega_\alpha$ satisfies the Eisenstein congruence mod $p$ for weight $k$ if
\begin{equation}\label{eq:eipgl2}
q^{(k-1)/2}(\alpha+\alpha^{-1}) \equiv q^{k-1} + 1 \mod p.
\end{equation}

\subsubsection{$\GSp(4)$}

Set $G:=\GSp(4,F)$, $K:=G(\fO)$, and $G^+:=G \cap M(4,\fO)$.
We write $T$ for the maximal split torus consisting of diagonal matrices in $G$.
Choose
\[
\diag(1,1,\varpi,\varpi), \quad \diag(\varpi,1,\varpi^{-1},1), \quad \diag(1,\varpi,1,\varpi^{-1}),
\]
as a basis of the group $T/T\cap K$.
The Hecke algebra $\cH(G^+,K)$ is generated by $c_1$ and $c_2$, where $c_1$ (resp. $c_2$) denotes the characteristic function of $K\diag(1,1,\varpi,\varpi) K$ (resp. $K\diag(1,\varpi,\varpi^2,\varpi) K$).
We write $\widehat{c}_1$ and $\widehat{c}_2$ for the Satake transform of $c_1$ and $c_2$ respectively.

Consider the spherical representation $\omega_\alpha$ with the Satake parameter $\alpha=(\alpha_0, \alpha_1,\alpha_2)$. 
Then, it is known that the values $\widehat c_1(\alpha)=\omega_\alpha(c_1)$ and $\widehat c_2(\alpha)=\omega_\alpha(c_2)$ are given by
\begin{equation}\label{eq:heckegsp1}
\widehat c_1(\alpha) = q^{3/2}(\alpha_0+\alpha_0\alpha_1+\alpha_0\alpha_2+\alpha_0\alpha_1\alpha_2),
\end{equation}
\begin{equation}\label{eq:heckegsp2}
\widehat c_2(\alpha) = q^2(\alpha_0^2\alpha_1+\alpha_0^2\alpha_2+\alpha_0^2\alpha_1^2\alpha_2+\alpha_0^2\alpha_1\alpha_2^2)+(q^2-1)\alpha_0^2\alpha_1\alpha_2,
\end{equation}
see, e.g., \cite[(3.15)]{Gross}, \cite[Section 6.1]{ralf-brooks}.
Furthermore, it is easy to prove that $\omega_\alpha(c)$ is in $\Z[\widehat c_1(\alpha),\widehat c_2(\alpha)]$ for each characteristic functions $c$ of double cosets $K\diag(\varpi^{c_1},\varpi^{c_2},\varpi^{-c_1+c_3},\varpi^{-c_2+c_3})K$, see, e.g., \cite[(2.9)]{Gross}.
Therefore, the Eisenstein congruences of $\omega_\alpha(c)$ are determined by those of $\widehat c_1(\alpha)$ and $\widehat c_2(\alpha)$, and so it is enough to consider congruences only for $\widehat c_1(\alpha)$ and $\widehat c_2(\alpha)$.

The parameter $(k_1,k_2)\in\Z^{\oplus 2}$ $(k_1+1\geq k_2\geq 2)$ corresponds to the $L$-parameter $\psi_{k_1+k_2-3,k_1-k_2+1}$, cf. \eqref{eq:gsp4r}.
If $k_1\geq k_2\geq 3$, then $(k_1,k_2)$ means the minimal $\UU(2)$-type $\det^{k_2}\otimes \mathrm{Sym}_{k_1-k_2}$ of the holomorphic discrete series for $\psi_{k_1+k_2-3,k_1-k_2+1}$.
We say that $\omega_\alpha$ satisfies the Eisenstein congruence mod $p$ for weight $(k_1,k_2)$ if
\begin{equation}\label{eq:eigsp1}
q^{\frac{k_1+k_2}{2}-3} \, \widehat c_1(\alpha) \equiv   q^{k_1+k_2-3}+q^{k_1-1}+q^{k_2-2}+1 \mod p ,
\end{equation}
\begin{equation}\label{eq:eigsp2}
q^{k_1-3}\, \widehat c_2(\alpha) \equiv q^{2k_1-2}+q^{k_1+k_2-3}+q^{k_1-k_2+1}+q^{k_1-1}-q^{k_1-3}+1 \mod p 
\end{equation}
hold.
When $(k_1,k_2)=(3,3)$, the congruences \eqref{eq:eigsp1} and \eqref{eq:eigsp2} are respectively equivalent to
\begin{equation}\label{eq:eigsp1tr}
\widehat c_1(\alpha) \equiv \vol(K\diag(1,1,\varpi,\varpi) K) \mod p ,
\end{equation}
\begin{equation}\label{eq:eigsp2tr}
\widehat c_2(\alpha) \equiv \vol(K\diag(1,\varpi,\varpi^2,\varpi) K) \mod p . 
\end{equation}
The following propositions follow from a direct calculation for \eqref{eq:eigl2}, \eqref{eq:heckegsp1}, \eqref{eq:heckegsp2}, \eqref{eq:eigsp1}, and \eqref{eq:eigsp2}.
\begin{prop}\label{pr:yoshida}
Let $k\geq k'\geq 1$ and $\omega_{\alpha}$ (resp. $\omega_{\alpha'}$) be the spherical representation of $\GL(2,F)$ with the Satake parameter $\alpha=(\alpha_1,\alpha_2)$ (resp. $\alpha'=(\alpha_1',\alpha_2')$).
Suppose that $k+k'$ is even, $\alpha_1\alpha_2=\alpha_1'\alpha_2'$ holds, and $\omega_\alpha$ (resp. $\omega_{\alpha'}$) satisfies the Eisenstein congruence mod $p$ for weight $k$ (resp. $k'$). 
Then, the spherical representation $\omega_\beta$ of $G$ with the Satake parameter $\beta=(\alpha_1,\frac{\alpha_1'}{\alpha_1},\frac{\alpha_2'}{\alpha_1})$ satisfies the Eisenstein congruence mod $p$ for weight $(\frac{k+k'}{2},\frac{k-k'}{2}+2)$.
\end{prop}
In this proposition, $\omega_\beta$ is the local Yoshida lift of $\omega_\alpha$ and $\omega_{\alpha'}$.
If one chooses $k'=2$ and $\alpha'=(q^{1/2},q^{-1/2})$, then $\omega_\beta$ is the
local Saito-Kurokawa lift.
\begin{prop}\label{pr:spsym}
Let $k \geq 1$ and $\omega_{\alpha}$ be the spherical representation of $\GL(2,F)$ with $\alpha=(\alpha_1,\alpha_2)$.
Suppose that $\omega_\alpha$ satisfies the Eisenstein congruence mod $p$ for weight $k$. 
Then, the spherical representation $\omega_\beta$ of $G$ with $\beta=(\alpha_1^3, \alpha_2/\alpha_1, (\alpha_2/\alpha_1)^2)$ satisfies the Eisenstein congruence mod $p$ for weight $(k-\frac{1}{2},\frac{k+1}{2})$.
\end{prop}
In this proposition, $\omega_\beta$ is the symmetric cube lift of $\omega_\alpha$.

\subsubsection{$G_2$}
Recall the notations in  \cref{sec:g2str} for $G_2$.
Set $G=G_2(F)$ and $K:=G_2(\fO)$. 
According to \cite[Section 5]{Gross} and \cite[Section 13]{GGS}, we write $\lambda_1$ and $\lambda_2$ for the fundamental co-weights, and then they satisfy $\lambda_1(t)=(t^3,t^{-2})$ and $\lambda_2(t)=(t^4,t^{-3})$ $(t\in\mathbb{G}_m)$.
The Hecke algebra $\cH(G,K)$ is generated by $c_1$ and $c_2$, which denote the characteristic functions of $K\lambda_1(\varpi)K$ and $K\lambda_2(\varpi)K$ respectively.
We choose $\lambda_1(\varpi)$ and $\lambda_2(\varpi)$ as a basis of the group $T/T\cap K$.
Let $\omega_\alpha$ be the spherical representation of $G$ with the Satake parameter $\alpha=(\alpha_1,\alpha_2)$. 
By \cite[(5.7)]{Gross}, the eigenvalues $\widehat c_1(\alpha)=\omega_\alpha(c_1)$ and $\widehat c_2(\alpha)=\omega_\alpha(c_2)$ satisfy
\begin{equation}
\widehat c_1(\alpha)=q^3-1+q^3(\alpha_1+\alpha_1^{-1}+\alpha_1^2\alpha_2^{-1}+\alpha_1^{-2}\alpha_2+\alpha_1^3\alpha_2^{-1}+\alpha_1^{-3}\alpha_2),
\end{equation}
\begin{multline}
\widehat c_2(\alpha)=2q^5-q^4-1-\widehat c_1(\alpha) +q^5(\alpha_1^2\alpha_2^{-1}+\alpha_1^{-2}\alpha_2+\alpha_1^3\alpha_2^{-2}+\alpha_1^{-3}\alpha_2^2\\
+\alpha_1\alpha_2^{-1}+\alpha_1^{-1}\alpha_2+\alpha_1+\alpha_1^{-1}+\alpha_1^3\alpha_2^{-1}+\alpha_1^{-3}\alpha_2+\alpha_2+\alpha_2^{-1}).
\end{multline}
The Eisenstein congruences mod $p$ of $\widehat c_1(\alpha)$ and $\widehat c_2(\alpha)$ provide those of any characteristic functions of double cosets $K\alpha K$.

The parameter $(a,b)\in (2\Z)^{\oplus 2}$ corresponds to the $L$-parameter $\tilde \psi_{a,b}$ over $\R$.
We say that $\omega_\alpha$ satisfies the Eisenstein congruence mod $p$ for weight $(a,b)$ if
\begin{equation}\label{eq:g2cong1}
\widehat c_1(\alpha) \equiv q^3-1+q^3(q^{a/2}+q^{-a/2}+q^{b/2}+q^{-b/2}+q^{(a+b)/2}+q^{-(a+b)/2}) \mod p,
\end{equation}
\begin{multline}\label{eq:g2cong2}
\widehat c_2(\alpha) \equiv 2q^5-q^4-1-\widehat c_1(\alpha)  +q^5(q^{\frac{a}{2}} + q^{-\frac{a}{2}} + q^{\frac{b}{2}} + q^{-\frac{b}{2}} \\
+ q^{\frac{a+b}{2}} + q^{-\frac{a+b}{2}} + q^{\frac{a-b}{2}} + q^{-\frac{a-b}{2}} +q^{a+\frac{b}{2}}+q^{-a-\frac{b}{2}} + q^{\frac{a}{2}+b} + q^{-\frac{a}{2}-b}    \mod p .
\end{multline}
When $(a,b)=(4,2)$, the congruences \eqref{eq:g2cong1} and \eqref{eq:g2cong2} are equivalent to
\begin{equation}\label{eq:g2cong1v2}
\widehat c_1(\alpha) \equiv q^6+q^5+q^4+q^3+q^2+q=\vol(K\lambda_1(\varpi)K) \mod p,
\end{equation}
\begin{equation}\label{eq:g2cong2v2}
\widehat c_2(\alpha) \equiv q^{10}+q^9+q^8+q^7+q^6+q^5=\vol(K\lambda_2(\varpi)K)    \mod p 
\end{equation}
cf. \cite[Section 7]{Gross}.
\begin{prop}\label{pr:g2cong}
Let $c$ and $d$ be odd natural numbers, and $\omega_{\beta_1}$ (resp. $\omega_{\beta_2}$) be the spherical representation of $\PGL(2,F)$ with the Satake parameter $\beta_1$ (resp. $\beta_2$).
Suppose that $\omega_{\beta_1}$ (resp. $\omega_{\beta_2}$) satisfies the Eisenstein congruence mod $p$ for weight $c+1$ (resp. $d+1$). 
Then, the spherical representation $\omega_\alpha$ of $G$ with the Satake parameter $\alpha=(\beta_1\beta_2,\beta_2^{-2})$ satisfies the Eisenstein congruence mod $p$ for weight $(c+d,c-d)$.
\end{prop}
\begin{proof}
This follows from a direct calculation, see \cite[Section 4.3]{lansky-pollack}
\end{proof}

If we choose $(c,d)=(3,1)$ or $(1,5)$ in \cref{pr:g2cong}, then one gets the Eisenstein congruence mod $p$ of $\omega_\alpha$ for weight $(4,2)$. (Note $\tilde\psi_{6,-4}\cong \tilde\psi_{4,2}$.)


\section{Lifting congruences} \label{sec:lifted-cong}


In this section, we explain how the results of the previous section
imply that various functorial lifts of unitary groups 
preserve Eisenstein congruences.  For simplicity we will restrict to the
case $F=\Q$, and sometimes assume $E$ has class number 1,
but this is not crucial.
We discuss lifts to $\SO(5)$ and $G_2$ in \cref{sec:lift-GL2,sec:lift-G2}.

\subsection{Eisenstein series and congruences for GL(2)}

Set $F=\Q$, and choose a natural number $N$, an integer $k\ge 2$, and a Dirichlet character $\chi$ on $\Z/N\Z$ such that $\chi(-1)=(-1)^k$.
We also denote by $\chi=\otimes_v \chi_v$ its lift on $F^\times\R_{>0}\bsl \A^\times$.
Let $S_k(N,\chi)$ denote the space of holomorphic cusp forms of weight $k$ and level $N$ (i.e., $\Gamma_0(N)$) with nebentypus $\chi$.
For simplicity, we put $S_k(N)=S_k(N,\chi)$ if $\chi$ is trivial. 
It is well known that each cusp form $f$ in $S_k(N,\chi)$ lifts to a function on $\GL(2,\Q)\bsl \GL(2,\A_F)$, and if $f$ is a Hecke eigenform, it generates 
an irreducible automorphic representation $\pi_f^{\GL(2)}=\otimes_v \pi_{f,v}^{\GL(2)}$ in $L^2_\mathrm{cusp}(\GL(2,F)\bsl \GL(2,\A_F),\chi)$.
The $L$-parameter of $\pi_{f,\inf}^{\GL(2)}$ is $\psi_{\frac{k-1}{2}}$, cf.\ \eqref{eq:gl2r}.
Let $N=N_1N_2$ with $(N_1,N_2)=1$, and choose Dirichlet characters $\chi_1$ on $\Z/N_1\Z$ and $\chi_2$ on $\Z/N_2\Z$ such that $\chi=\chi_1\chi_2$ on $\Z/N\Z$. 
For $k > 2$, we can define an Eisenstein series $E_{k,\chi_1,\chi_2}$ of weight $k$ and character $(\chi_1,\chi_2)$ by
\[
E_{k,\chi_1,\chi_2}(z)=\sum_{(m,n)\in\Z^2\setminus \{(0,0)\} } \chi_1(m)\chi_2(n)\, (mz+n)^{-k}.
\]
If $N \ne 1$, a regularized expression similarly gives a weight 2 Eisenstein
series $E_{2, \chi_1, \chi_2}$.
For $k \ge 2$, $E_{2,\chi_1, \chi_2}$ is a Hecke eigenform with $T_p$-
eigenvalue $\chi_1(p)p^{k-1} + \chi_2(p)$ for $p \nmid N$.
An eigenform $f \in S_k(N,\chi)$ is Hecke congruent to $E_{k,\chi_1,\chi_2} \bmod p$ away from a set of places $S$ if and only if $\pi_{f,v}^{\GL(2)}$ satisfies the local Eisenstein congruence mod $p$  in \eqref{eq:eigl2} for weight $k$ and character $(\chi_{1,v},\chi_{2,v})$ for all $v \not \in S$.
(Here we are fixing the a ring of integers and a prime ideal above $p$ 
for these congruences---we of course 
are not allowing the use of different prime ideals
above $p$ as we vary $v$ in the local congruence.)
In particular, for weight 2 and trivial characters $f \in S_2(N)$ being Hecke congruent to $E_{2,1,1} \bmod p$ corresponds to $\pi_f^{\GL(2)}$ is Hecke congruent to $\one \bmod p$.

\subsection{The Kudla lift and symmetric power lifts} \label{sec:kudla}

Choose a fundamental discriminant $D<0$, and set $F=\Q$ and $E=\Q(\sqrt{D})$.
Suppose that the class number of $E$ is one.
We write $\chi_D$ for the quadratic Dirichlet character on $\Z/D\Z$ corresponding to $E/F$.

We will consider the following two cases:
\begin{itemize}
\item[(I)] $k\ge 2$, $k$ is even, $N$ is arbitrary, and $\chi$ is trivial. 
\item[(II)] $k\ge 3$, $k$ is odd, $N$ is divisible by $D$, and $\chi=\chi_D$.
\end{itemize}

Choose a character $\omega=\otimes_v \omega_v$ on $\UU(1,F)\bsl \UU(1,\A_F)$ as follows.
In case  (I), take $\omega$ to be trivial.
In case (II), let $\Omega=\otimes_v \Omega_v$ denote a character on $E^\times\bsl \A_E^\times$ such that $\Omega|_{\A_F^\times}=\chi_D$ and $\Omega_\inf(z)=|z|/z$ $(\forall z\in E_\inf^\times =\C^\times)$, and we define $\omega=\otimes_v \omega_v$ to be the restriction of $\Omega$ to $\UU(1,\A)$.

The mapping $\UU(1)\times \SL(2) \ni (z,h) \mapsto zh\in G=\UU(2)$ gives rise to the isomorphism $(\C^1\times \SL(2,\R))/\{\pm 1\}\cong \UU(2,\R)$.
In addition, one has $G(\Q)\cap \prod_{v<\inf} G(\Z_v)=(\fO_E^\times \times \SL(2,\Z))/\{\pm 1\}$, where $G(\Z_v):=G(\Q_v) \cap M(2,\fO_E\otimes_\Z\Z_v)$.
Hence, each cusp form $f$ in $S_k(N,\chi)$ is identified with a smooth function $\varphi_{f,\R}$ on $\UU(2,\R)$ satisfying
\[
\varphi_{f,\R}(\gamma z g k)= \omega_\inf(z)\, \chi(\gamma) \, e^{ik\theta}\, \varphi_{f,\R}(g)
\]
for any $\gamma\in\Gamma_0(N)$, $z \in \C^1$, $g\in\SL(2,\R)$, $k=\begin{pmatrix} \cos\theta&  \sin\theta \\ -\sin\theta & \cos\theta \end{pmatrix} \in \SO(2)$.
Furthermore, $\varphi_{f,\R}$ can be lifted as a smooth function $\varphi_{f,\A}$ in $L^2_\mathrm{cusp}(G(F)\bsl G(\A_F),\omega)$ by the assumption that 
$E$ has  class number 1.
Let $\pi_f$ be an irreducible component of the
 automorphic representation of $G(\A_F)$
generated by $\phi_{f,\A}$.
In case (I) (resp.\ (II)), the $L$-parameter $\psi_\inf$ of $\pi_{f,\inf}$ is 
given by $\psi_\inf(z)=\diag((z/\overline{z})^{(k-1)/2},(z/\overline{z})^{(-k+1)/2})$ (resp.\ $\psi_\inf(z)=\diag((z/\overline{z})^{(k-2)/2},(z/\overline{z})^{-k/2})$). 
\begin{lem}\label{lem:1418}
Let $v$ be a finite place of $F$ which is non-split in $E$, and suppose that $G(F_v)$ is unramified, and $\pi_{f,v}$ and $\pi_{f,v}^{\GL(2)}$ are spherical.
Suppose that
\begin{itemize}
\item $(k_1,k_2)=((k-1)/2,(-k+1)/2)$ and $\chi_v$ is trivial in case (I); 
\item $(k_1,k_2)=((k-2)/2,-k/2)$ and $\chi_v$ is the unramified quadratic character on $F_v^\times$ in case (II).
\end{itemize}
Then, $\pi_{f,v}$ satisfies the Eisenstein congruence mod $p$ in \eqref{eq:eigunitary} for $(k_1,k_2)$ if $\pi_{f,v}^{\GL(2)}$ satisfies the Eisenstein congruence mod $p$ in \eqref{eq:eigl2} for weight $k$ and character $(\chi_v,1)$ or $(1,\chi_v)$.
\end{lem}
\begin{proof}
Let $q$ denote the prime number corresponding to $v$.
The assertion can be proved by considering the action of the Hecke operator of $\Gamma_0(N)\diag(q^{-1},q)\Gamma_0(N)$ on $f$, and its lifts to $G(\A_F)$ and $\GL(2,\A_F)$.
\end{proof}

For split places, we only consider local Eisenstein congruences \eqref{eq:eigl2}  for trivial character. 

\begin{lem}\label{lem:2418}
Let $v$ be a finite place of $F$ split in $E$, and suppose that $\pi_{f,v}$ and $\pi_{f,v}^{\GL(2)}$ are spherical.
In case (I), $\pi_{f,v}$ satisfies the Eisenstein congruence mod $p$ in  \eqref{eq:eigln} for $((k-1)/2,(-k+1)/2)$ if and only if $\pi_{f,v}^{\GL(2)}$ satisfies the Eisenstein congruence mod $p$ in \eqref{eq:eigl2} for weight $k$.
In case (II), $\pi_{f,v}\boxtimes \omega_v^{-1}\circ\det$ satisfies the Eisenstein congruence for weight $k$ in \eqref{eq:eigl2} (not \eqref{eq:eigln}) if and only if $\pi_{f,v}^{\GL(2)}$ satisfies the Eisenstein congruence mod $p$ in \eqref{eq:eigl2} for weight $k$.
\end{lem}
\begin{proof}
Let $q$ denote the prime number corresponding to $v$.
Since we are now supposing the class number of $E$ is one, for each a place $w$ of $E$ dividing $v$, there exists a prime element $\varpi_w\in \fO=\fO_E$ of $\fO_w$ such that $q=\varpi_w\overline{\varpi_w}$.
Then, from the viewpoint of the action on $f$, the Hecke operator of $\Gamma_0(N)\diag(\varpi_w^{-1},\varpi_w^{-1}q)\Gamma_0(N)$ on $G$ is identified with that of $\Gamma_0(N)\diag(1,q)\Gamma_0(N)$ on $\GL(2)$.
Hence, the assertion can be proved by considering their lifts and the central character $\omega$.
\end{proof}

The following is a global congruence lift by the Kudla lift.
\begin{thm}\label{th:Kudlalift}
Suppose that the class number of $E$ is one.
Let $N=N_1N_2$ where $N_1$ and $N_2$ be natural numbers such that $(N_1,N_2)=1$.
Choose real-valued Dirichlet characters $\chi_1$ on $\Z/N_1\Z$ and $\chi_2$ on $\Z/N_2\Z$ such that $\chi_D=\chi_1\chi_2$ on $\Z/N\Z$, and a finite set $S$ of places of $F$ such that $\inf\in S$ and $G$ is unramified for $v\not\in S$. 
Assume that there exists a Hecke eigen cusp form $f$ in $S_3(N,\chi_D)$ such that $f$ is Hecke congruent to $E_{3,\chi_1,\chi_2}(z) \bmod p$.
Then, there exists an automorphic representation $\pi=\otimes_v \pi_v$ of $\UU(3)$ such that $\pi$ is Hecke congruent to $\one \bmod p$ and the $L$-parameter of $\pi_\inf$ is associated with $(1,0,-1)$.
\end{thm}
\begin{rem}
There is also an endoscopic lift from $S_2(N)$ to an automorphic representation of $\UU(3)$ with the same real component, but it does not preserve the Eisenstein congruences at places $v$ which are non-split over $E$ in general, see  \cref{sec:U(N)}. 
\end{rem}
\begin{proof}
By the Kudla lift, one has a lifting of $f$ to an automorphic form $F_f$ in $L^2_\mathrm{cusp}(\UU(3,F)\bsl \UU(3,\A_F))$, see \cite{Kudla} and \cite{MS}.
The automorphic representation of $F_f$ corresponds to the $L$-parameter of $\pi_f$ of $\UU(2)$ and the trivial representation of $\UU(1)$ via the $L$-embedding ${}^L\!\UU(2)\times {}^L\!\UU(1)\to {}^L\!\UU(3)$ with $\Omega^{-1}$, cf. \cite[Section 4.8]{rogawski}.
Therefore, the assertion follows from  \cref{pr:unicong1},  \cref{lem:1418,lem:2418}, and the proof of  \cref{pr:glcong1}. 
Notice that the condition for $\pi_{f,v}$ in  \cref{lem:2418} is different from  \cref{pr:glcong1}, but the local congruence lift can be proved by the same manner as in the proof of  \cref{pr:glcong1}. 
\end{proof}

The following tells us that symmetric power lifts, when they exist, preserve
Eisenstein congruences.  For non-CM forms cusp forms $f$, the
$m$-th symmetric power lift to $\GL(m+1)$ is known for $m \le 8$
(see \cite{clozel-thorne3} and references therein), which can be
transferred to $\UU(m+1)$ using \cite{mok}.  (In fact
\cite{clozel-thorne3} first constructs the representation on
a definite $\UU(m+1)$.)

\begin{thm}
Suppose that the class number of $E$ is one.
Assume $f \in S_2(N)$ is an eigenform which is Hecke congruent to 
$E_{2,1,1} \bmod p$.
If there exists an $m$-th symmetric power lift $\pi$ of $\pi_f$ from $\UU(2)$ to $\UU(m+1)$, then $\pi$ is Hecke congruent to $\one \bmod p$.
\end{thm}
\begin{proof}
This follows from \cref{pr:glnsymcong,pr:unisymcong} and  \cref{lem:1418,lem:2418}. 
\end{proof}

\subsection{Congruences on $\UU(1)$}\label{sec:gcu1}

Since endoscopic lifts for $\UU(n)$ involve characters of $\UU(1)$, to
understand when these lifts preserve Eisenstein congruences, we need
to know about congruences of characters with different infinity types.  
Let $E$ be an imaginary quadratic field over $F=\Q$ of class number 1.

First we briefly discuss characters of even infinity type.
For any $u\in\Z$, one can define a Hecke character $\Omega=\otimes_v \Omega_v$ on $E^\times\A_F^\times\bsl\A_E^\times$ by
\[
\Omega_\inf(z)=(z/\overline{z})^{-u} .
\]
Suppose that $\Omega_\inf(\varepsilon)=1$ for any $\varepsilon\in\fO_E^\times$, i.e., suppose that $\frac{|\frako_E^\times|}2
| u$.
Then $\Omega$ is identified with a character on $\UU(1,F)\bsl \UU(1,\A)$ by the isomorphism $\mathbb{G}_m\bsl R_{E/F}(\mathbb{G}_m) \ni z \mapsto z/\overline{z}\in \UU(1)$ over $F$.
Note the corresponding character on $\C^1=\UU(1,\R)$ is $e^{i\theta}\mapsto e^{-iu\theta}$.
When a prime number $q$ is not split in $E$, $\UU(1,F_q)$ is compact, and so the spherical representation $\Omega_q$ is trivial.
When a prime number $q$ of $F$ is split in $E$, we choose a place $w|q$ and we take a prime element $\varpi_w\in \fO_E$ of $\fO_w$.
Then one has
\[
\Omega_q(\varpi_w) = \Omega_\inf(\varpi_w^{-1})=(\varpi_w/\overline{\varpi_w})^u=q^{-u} \times \varpi_w^{2u}.
\]
Therefore, if $\varpi_w^{2u}\equiv 1 \bmod p$ holds for the place $q$, then $\Omega$ satisfies the local Eisenstein congruence mod $p$ of $L$-parameter $-u$ for the split place $q$, that is, $q^{u}\Omega_w(\varpi_w)\equiv 1 \bmod p$.
This local congruence holds for all (split) $q$ whenever $(p-1) | 2u$
(in particular for any $u$ if $p=2, 3$), in which case we get that
$\Omega$ satisfies a global Eisenstein mod $p$ congruence with
archimedean parameter $-u$.

Next we give an example of a congruence for odd infinity type.
Let $E=\Q(\sqrt{-7})$, $\chi_{-7}$ denote the quadratic Dirichlet character on 
$\Z/7\Z$ corresponding to $E/F$.
A character $\Omega=\otimes_v \Omega_v$ on $E^\times\bsl \A_E^\times$ is defined by \[
\Omega_\inf(z)=(z/\overline{z})^{-u} \, |z|/z \;\; (z\in\C^\times)   , \quad \Omega_7(x+y\sqrt{-7})=\chi_{-7}(x) \;\; (x\in\Z_7^\times,  \;\; y\in\Z_7)  .
\]
Note that $\Omega|_{\A_F}=\chi_{-7}$ holds.
When a prime number $q_1$ of $F$ is non-split in $E$ and not $7$, we have $\Omega_{q_1}(q_1)=-1$. 
Further, for each prime $q$ split in $E$, one has
\[
\Omega_q(\varpi_w)=q^{-u-\frac12} \times \Omega_7(\varpi_w) \times \varpi_w^{2u+1}.
\]
Since $a^3\equiv \chi_{-7}(a) \bmod 7$ $(a\in(\Z/7\Z)^\times)$, $\Omega_q$ satisfies the local Eisenstein congruence mod $7$ of $L$-parameter $-u-\frac{1}{2}$ if $2u+1\equiv 3 \bmod 6$.

\subsection{Endoscopic lifts for $\UU(n)$}\label{sec:20190530lift}

Here we briefly discuss elliptic endoscopic lifts of congruences
for $\UU(n)$.
Let $E$ be an imaginary quadratic field over $F=\Q$, and $n_1$, $n_2$ be nutural numbers.
Suppose that $n_2$ is even.
We set $n=n_1+n_2$ and consider endoscopic lifts from $\UU(n_1)\times \UU(n_2)$ to $\UU(n)$ for $E/F$.
Let $\pi_1=\otimes_v \pi_{1,v}$ be an automorphic representation of $\UU(n_1)$ such that $\pi_{1,\inf}$ corresponds to the $L$-parameter $k_1=\rho_{n_1}$ (see \eqref{eq:rhotrivial} and Section \ref{sec:20190531} for $\rho_{n_1}\in \mathbb{L}_{n_1}$).
Choose an integer $u$ if $n_1$ is odd.
Let $\pi_2=\otimes_v \pi_{2,v}$ be an automorphic representation of $\UU(n_2)$ such that $\pi_{2,\inf}$ corresponds to the $L$-parameter
\[
k_2=\begin{cases} \frac{1}{2}(n-2-2u,\dots, n_1-2u,-n_1-2-2u,\dots, -n-2u) & \text{if $n_1$ is odd,} \\ \frac{1}{2}(n-1,\dots,n_1+1,-n_1-1,\dots, -n+1) & \text{if $n_1$ is even.} \end{cases}
\]
Assume that $\pi_1$ and $\pi_2$ are associated with global generic parameters in the sense of \cite[Section 2.3]{mok}.
When $n_1$ is odd, choose a character $\chi=\otimes_v \chi_v$ on $E^\times\bsl \A_E^\times$ such that $\chi_\inf(z)=(z/\overline{z})^u z/|z|$ and $\chi|_{\A_F^\times}$ is the quadratic character of $F^\times\bsl \A_F^\times$ corresponding to $E/F$.
When $n_1$ is even, we set $\chi=1$. 
We also choose a finite set $S$ of places of $F$ containing $\infty$
such that $\pi_{1,v}$, $\pi_{2,v}$, and $\chi_v$ are unramified for any $v\notin S$.
In the case that $n_1$ is odd, we suppose that $\chi_v(\varpi_v) q_v^{-u-\frac12}\equiv 1  \bmod p$ holds for each place $v\not\in S$ when $v$ is split in $E$, and $\chi_v(\varpi_v)\equiv -1 \bmod p$ holds for each place $v\not\in S$ when $v$ is non-split in $E$, see  \cref{sec:gcu1} for an example.
Assume that $\pi_{1,v}$ (resp. $\pi_{2,v}$) satisfies the Eisenstein congruences mod $p$ of $L$-parameter $k_1$ (resp. $k_2$) for any $v\not\in S$.
Then, there exists an automorphic representation $\pi$ of $\UU(n)$, which is an endoscopic lift of $\pi_1$ and $\pi_2$ by the $L$-embedding of $\chi$ (see \cite{mok}), and it follows from \cref{pr:glcong1,pr:unicong1} that $\pi$ is Hecke congruent to $\one$ mod $p$.

One can also consider global congruence lifts as above when both $n_1$ and 
$n_2$ are odd.  However it is more complicated due to character considerations
(cf.\ \cref{pr:unicong1}) and we do not discuss it here. 

On the other hand, it is easier to lift congruences from
from $\UU(n-2)\times \GL(2)$, generalizing \cref{th:Kudlalift} and the above argument with $n_1=n-2$ and $n_2=2$. 

\begin{thm}\label{thm:20190601}
Let $\rho_m$ denote the $L$-parameter $\psi$ of $\UU_{\C/\R}(m,\R)$ for $l_1=\cdots=l_m=0$ in \eqref{eq:LDS}.
Let $n \ge 2$, and $\pi_1$ be an automorphic representation of $\UU(n-2)$ which has the $L$-parameter $\rho_{n-2}$ at $\infty$, is associated with a global generic parameter, and is Hecke congruent to $\one$ mod $p$.
Let $f \in S_n(N,\chi_D^n)$ be an eigenform which is Hecke congruent to $E_{n,1,\chi_D^n}(z) \bmod p$.
Then, there is an automorphic representation $\pi$ of $\UU(n)$ such that $\pi$ is an endoscopic lift of $\pi_1$ and $\pi_f$, $\pi$ has $L$-parameter $\rho_n$ at $\infty$, and $\pi$ is Hecke congruent to $\one$ mod $p$.
\end{thm}

\cref{thm:main-Un} implies the existence of $\pi_1$ for $n-2$ prime in the above theorem. 
The assumption of a global generic parameter ensures that the endoscopic lift has the $L$-parameter $\rho_n$ at $\infty$.

\subsection{The Ikeda lift}
Finally, we discuss congruence lifts obtained from the Ikeda lifts to $\UU(2m)$, cf. Ikeda \cite{ikeda} and Yamana \cite{yamana}.
The global $A$-parameters of the Ikeda lifts are of the form $\mu\boxtimes \nu$, where $\mu$ is the base change lift of a cuspidal automorphic representation of $\GL(2,\A_F)$ whose component at $\infty$ is a holomorphic discrete series, and $\nu$ is the $m$-dimensional irreducible
representation of $\SU(2)$.
Hence, almost all local conponents of the Ikeda lifts are non-tempered.
\begin{thm}\label{th:oddunitikeda}
As in \cref{sec:kudla}, assume $F=\Q$, $E=\Q(\sqrt{D})$, and the class number of $E$ is one.
Choose natural numbers $N$ and $m$.
Let $f$ be a Hecke eigen cusp form in $S_{2m}(N)$ which is Hecke congruent to 
$E_{2m,1,1} \bmod p$.
Then, there exists a cuspidal automorphic representation $\pi=\otimes_v \pi_v$ of $\UU(4m-2)$ such that $\pi$ is Hecke congruent to $\one \bmod p$ and $\pi_\inf$ is the holomorphic discrete series whose the $L$-parameter is associated with $\rho_{4m-2}$.
\end{thm}
\begin{proof}
By \cite[Theorem 1.1]{yamana}, we get an automorphic representation $\pi_1$ of $\mathrm{GU}(4m-2)$ which contains a nonzero Hermitian cusp form in the lift of $f$.
We can choose a nonzero irreducible constitutent $\pi$ of the restriction of $\pi_1$ to $\UU(4m-2)$, whose component is the holomorphic discrete series.
Hence, this theorem follows from \cref{pr:glnsymcong,pr:unisymcong} and \cref{lem:1418,lem:2418}.
\end{proof}



\section{Unitary groups} \label{sec:unitary}


Let $E/F$ be a CM extension of number fields, and $G' = \UU(n)$ be the
associated quasi-split unitary group over $F$ in $n$ variables
defined in \eqref{eq:unitarygroup}.

Let $G$ be an inner form of $G'$.  We can realize $G$ as follows.
There exist (i) a central simple algebra $A/E$ of degree $n$, i.e., $\dim_E A=n^2$,
and (ii) an involution $\alpha \mapsto \alpha^*$ of $A$ of the second kind 
with $\alpha^* = \bar \alpha$ for $\alpha \in E$, such that 
\[
G = \{ g \in A^\times : g^* g = 1 \}.
\]  
We remark that $G$ is the automorphism
group of the Hermitian form $\langle \alpha, \beta \rangle = \alpha^* \beta$ on
$A$.  
The center of $G$ is $E^\times \cap G$ (viewing $E^\times$ as the algebraic group 
$\mathrm{Res}_{E/F} \mathbb G_m$), 
which we may identify with $\UU(1) = E^1 = \{ a \in E^\times : a
\bar a = 1 \}$.

To specify $A$ and/or $*$ below, we will also denote
$G = \UU_A(n) = \UU_{A,*}(n)$.  (The isomorphism class depends on both
$A$ and $*$, but as we will typically only be concerned about specifying $A$ 
we often just write $\UU_A(n)$.)  Landherr's theorem on the classification of
involutions of the second kind tells us that if $v$ is inert or ramified in $E/F$, then 
$A_v$ is split.
Moreover, if $v$ splits in $E/F$ as $v = ww'$, then $*$ interchanges the factors
of $A_w$ and $A_{w'}$, giving an isomorphism $A_w \simeq A_{w'}^{\mathrm{opp}}$ and
$G_v = \UU_A(n, F_v) \simeq A_w^\times \simeq A_{w'}^\times$.

We will now assume $G = \UU_{A,*}(n)$ is a definite unitary group, i.e.,
the associated Hermitian form is totally definite.  
This means $G_v$ is compact for all $v | \infty$.
Note that one can make a definite involution on $A$ from any involution by conjugation (see \cite[Remark 10.6.11]{Sch}).

Let $\det$ denote the reduced norm on $A$.  By restriction to $G = \UU_A(n)$,
we may view $\det$ as a homomorphism of algebraic groups $\det: G \to \UU(1)$.
The derived subgroup $\SU_A(n)$ of $G$ is the kernel of 
$\det$, so 
any 1-dimensional automorphic representation of $G(\A)$ factors through $\det$.

\begin{lem} \label{det-lem}
The map $\det: G(k) \to \UU(1,k)$ is a surjective map of rational
points for any $k=F_v$ and for $k=F$.
\end{lem}

\begin{proof} By the Hasse principle for the norm map of unitary groups
(\cite[Theorem 6.28]{platonov-rapinchuk}), the result for $F$ follows
from the result for each $F_v$.  If $v$ is split in $E_v/F_v$,
the local result follows from surjectivity of reduced norm for central
simple algebras over $p$-adic fields.  
Otherwise, $G(F_v)$ is an honest unitary group, and it
is clear $\det$ restricted to the diagonal torus surjects onto $\UU(1,k)$.
\end{proof}

\subsection{Endoscopic classification}\label{sec:classification}

Here we briefly explain certain aspects of the endoscopic classification
for unitary groups as asserted in \cite[Theorem* 1.7.1]{KMSW},
and refer the reader to \emph{op.\ cit.}\ and \cite{mok}
for more precise details.

The endoscopic classification was treated by
Rogawski \cite{rogawski} for $\UU(3)$ and its inner forms 
(as well as quasi-split $\UU(2)$), by Mok \cite{mok} for quasi-split
$\UU(n)$, and by Kaletha--Minguez--Shin--White \cite{KMSW}
for inner forms of $\UU(n)$ under some hypotheses.  (See 
\cite[Section 2.6]{mok} for a summary of some intermediary results.)
These latter results rely on the stabilization of the twisted trace formula
which was established in \cite{stab:final}, and also require the
general weighted fundamental lemma which is expected to be
finished by Chaudouard and Laumon.
Work in progress of Kaletha--Minguez--Shin is expected to complete
the proof of \cite[Theorem* 1.7.1]{KMSW}, and we will
assume this in our subsequent congruence results.  

In fact the cases that we need are in some sense easier than cases already
established in the literature (e.g., \cite{harris-taylor}, \cite{labesse},
\cite{shin}, \cite{mok}), as the only non-quasi-split forms we
consider are certain compact forms, where the trace formula analysis is
simpler and one does not have endoscopic contributions.  
However, to our knowledge the cases we use (definite unitary
groups over division algebras) have not been explicitly dealt with in the literature.

To describe the classification, in this section we let
$G$ be an arbitrary inner form of $G' = \UU(n)$.
In particular, we allow $G=G'$.

As in \cite{mok}, the set of formal global parameters for $G'$ is the
set $\Psi(G')$ consisting of formal sums (up to equivalence)
$\psi = \psi_1 \boxplus
\dots \boxplus \psi_m$ of formal tensors $\psi_i = \mu_i \boxtimes \nu_i$,
where $\mu_i$ is a cuspidal automorphic representation of
$\GL_{n_i}(\A_E)$ and $\nu_i$ is the $r_i$-dimensional irreducible
representation of $\SU(2)$, such that $\sum n_i r_i = n$ and
the parameter $\psi$ is conjugate self-dual.  If $m=1$, we call $\psi$ simple.  If each $\nu_i = 1$, we call $\psi$ generic.
Set $\dim \psi_i = n_i r_i$.

According
to the Moeglin--Waldspurger classification, $\mu_i \boxtimes \nu_i$
corresponds to a discrete automorphic representation $\sigma_{\psi_i}$
of $\GL_{n_ir_i}(\A_E)$, which is cuspidal if $r_i=1$.
Thus by Langlands theory of Eisenstein series, $\psi$ corresponds
to an automorphic representation $\sigma_\psi$ of $\GL_n(\A_E)$.
Let $\Psi_2(G')$ denote the subset of square-integrable parameters,
which are of the form $\psi = \psi_1 \boxplus \dots \boxplus \psi_m$
where the $\psi_i$'s are all distinct and each $\psi_i$ is conjugate
self-dual.  Let $\Psi_2(G', \mathrm{std})$ be the subset of
$\Psi_2(G')$ which ``factor through'' the standard $L$-embedding
$\mathrm{std}: {}^LG' \to {}^L \mathrm{Res}_{E/F}(G')$
(this set is denoted $\Psi_2(G', \xi_1)$ in \cite[Definition 2.4.5]{mok}).

Let $\psi = \psi_1 \boxplus \dots \boxplus \psi_m \in \Psi_2(G')$.
One associates to $\psi$ a component group 
$\mathcal S_\psi \simeq (\Z/2\Z)^{m'}$ 
(denoted $\bar{\mathcal S}_\psi$ in \cite{KMSW}), and
a canonical sign character $\epsilon_\psi$ of $\mathcal S_\psi$.  
Here $0 \le m' \le m$---see \cite[(2.4.14)]{mok}
for a precise description of $m'$.
We note $\epsilon_\psi = 1$ if $\psi$ is generic.  
Then there is a global packet 
$\Pi_\psi(G) = \Pi_\psi(G, \xi, \epsilon_\psi)$ of representations 
attached to an inner twist $(G, \xi)$ that is a
certain subset of a restricted product of local packets consisting of 
elements which are globally compatible with $\epsilon_\psi$. 
 (Here $\xi$ is an 
$\bar F$-isomorphism from $G$ to $G'$ exhibiting $G$ as an inner 
form  of $G'$.)  The role of $\epsilon_\psi$ is to give a parity condition
for a product of members of local packets to lie in the global packet.

The packet $\Pi_\psi(G)$ is necessarily empty if $\psi$ is not 
locally relevant
everywhere for $G$.  Specifically, if $v$ is split in $E/F$, and
$G(F_v) \simeq M_{r_v}(D_v)$ where $D_v$ is a central $F_v$-division algebra of degree $d_v$, then for 
$\psi = \psi_1 \boxplus \dots \boxplus \psi_m$ 
to be relevant it is necessary that $d_v | \dim \psi_i$ for each $i$.

For $\psi \in \Psi_2(G', \mathrm{std})$ and $\pi \in 
\Pi_{\psi}(G)$, we call the associated automorphic
representation $\pi_E := \sigma_\psi$ of $\GL_n(\A_E)$ the (standard)
base change of $\pi$.  Note that $\pi_E$
is cuspidal if and only if $\psi = \pi_E \boxtimes 1$, 
i.e., if and only if $\psi$ is simple generic.
If $\pi_E$ is cuspidal and $v=ww'$ is a split place for $E/F$, then
$\pi_v \simeq \pi_{E,w}$ when $A_v$ is split, and more generally
$\pi_v$ corresponds to $\pi_{E,w}$ via the
Jacquet--Langlands correspondence for $\GL(n)/E$.

Then the $\kappa = 1$ and $\chi_\kappa = 1$ case of
 \cite[Theorem* 1.7.1]{KMSW} states that we have a $G(\A)$-module
 isomorphism:
\begin{equation} \label{ECU} \tag{EC-U}
L^2_{\mathrm{disc}}(G(F) \bs G(\A)) \simeq
\bigoplus_{\psi \in \Psi_2(G', \mathrm{std})} \bigoplus_{\pi \in \Pi_{\psi}(G)} \pi.
\end{equation}
\renewcommand*{\theHequation}{notag.\theequation}

A consequence of this is a generalized Jacquet--Langlands 
correspondence for unitary groups.  Namely,  fix an inner form
$G$ of $G'$, so $G(F_v) \simeq G'(F_v)$ for almost all $v$.
For simplicity, assume $\psi \in \Psi_2(G', \mathrm{std})$ is
simple generic, so we may view $\psi$ as a conjugate self-dual
 cuspidal representation of  $\GL_n(\A_E)$.  
Then the packet $\Pi_\psi(G')$ is non-empty---in fact it contains a cuspidal 
generic representation of $G'$ \cite[Corollary 9.2.4]{mok}.
If $\pi \in \Pi_\psi(G)$ we write $\JL(\pi) = \Pi_\psi(G')$ for 
the Jacquet--Langlands
correspondent to the packet $\Pi_\psi(G)$.
For $v$ split in $E/F$ and $\pi' \in \JL(\pi)$, $\pi_v$ and $\pi'_v$
correspond via the local Jacquet--Langlands correspondence for $\GL_n(F_v)$,
and necessarily $\pi'_v \simeq \pi_v$ if $G(F_v) \simeq G'(F_v) \simeq \GL_n(F_v)$.

It is expected that generic packets are tempered.
If $\psi$ is cohomological, then 
Shin \cite{shin} (together with \cite{chenevier-harris} when $n$ is
even and $\psi_\infty$ is not Shin-regular)
guarantees that $\psi_v$ is tempered at all finite $v$.
Now let us also assume $\psi$ is cohomological.

For $\pi' \in \Pi_\psi(G')$, the local packets $\Pi_{\psi_v}$ for $\pi$ and $\pi'$
are the same at almost all places.  But, by definition, elements of the global
packets correspond locally to the trivial character of the component group
(and thus unramified local parameters $\psi_v$)
almost everywhere.  
Since $\psi_v$ is generic and bounded (tempered), 
the local packet $\Pi_{\psi_v}(G'(F_v))$
is in bijection with the dual of the component group 
at nonarchimedean $v$ (\cite[Theorem 2.5.1(b)]{mok}).
Consequently, $\pi_v \simeq \pi'_v$ for almost all $v$.

In fact we can say more.  Since $\psi$ is simple generic,
we have $|\mathcal S_\psi| =1$ (\cite[(2.4.14)]{mok}). 
This means there is no parity condition associated to $\epsilon_\psi$
required for a product $\pi' = \otimes \pi'_v$ of local components of packets
to lie in the global packet $\Pi_\psi(G')$.  Hence, given $\pi$, we may always
choose $\pi' \in \JL(\pi)$ such that $\pi'_v \simeq \pi_v$ whenever $G(F_v) \simeq
G'(F_v)$.  Moreover, at all other $v$, we can choose $\pi'_v$ freely within
the local packet $\Pi_{\psi_v}(G'(F_v))$.

\subsection{A cuspidality criterion for base change}
For the remainder of this section, we return to our assumption that 
$G = \UU_{A, *}(n)$  is a definite unitary group.

\begin{prop} \label{prop:cusp} Assume $\eqref{ECU}$.  Suppose
$n$ is prime and $A_{w}$ is a division algebra for some
finite prime $w$ of $E$.  If $\pi$ occurs in $\calA(G, K)$,
and $\pi$ is not 1-dimensional, then $\pi_E$ is cuspidal.
\end{prop}

\begin{proof}
Necessarily, there is a finite prime $v$ of $F$ which splits
as $v=ww'$ for some $w'$.  Then $G(F_v) \simeq A_w^\times$
is the multiplicative group of a degree $n$ division algebra.
Let $\psi \in \Psi_2(G', \mathrm{std})$ be the parameter associated
to $\pi$.  Then for $\psi$ to be relevant, we need $\psi$ to be simple,
i.e., $\psi = \mu \boxtimes \nu$ for some cuspidal automorphic
representation $\mu$ of $\GL_{m}(\A_E)$ and $\nu$ of dimension
$r = \frac nm$.  

Since $n$ is prime, either $m=1$ or $m=n$.  If $m=n$, we are done.
Otherwise, the proposition follows from the following lemma, which
was kindly explained to us by Sug Woo Shin.
\end{proof}

\begin{lem} Assume \eqref{ECU}.  Suppose $\pi$
is an automorphic representation of $G$ associated
to a simple parameter $\psi = \mu \boxtimes \nu$ where
$\mu$ is a representation of $\GL_1(\A_E)$.
Then $\pi$ is 1-dimensional.
\end{lem}

\begin{proof}
Suppose $v=ww'$ is split in $E$.
The local base change $\pi_{E,w}$ is a
1-dimensional representation of $\GL_n(E_w)$.  
Then $\pi_v \simeq \pi_{E,w}$, so $\pi_v$ is 1-dimensional.
Since the strong approximation property with respect to $v$ is satisfied by $G^1 = \{ g \in G : \det g = 1 \}$ (see \cite[Theorem 7.12]{platonov-rapinchuk}),
$\pi_v$ trivial on $G^1(F_v)$ implies $\pi$ is trivial on $G^1(\A)$. 
Thus $\pi$ is 1-dimensional.
\end{proof}

\subsection{Eisenstein congruences} \label{sec:eis-Un}
Let $K = \prod K_v \subset G(\A)$ be a compact open subgroup which is
hyperspecial and maximal at almost all $v$.  We assume that $K_v = G_v$
for $v | \infty$, and place the following assumptions on $K_v$ for $v < \infty$.

First suppose $v$ splits in $E/F$.  Then we can write $G_v = \GL_{r_v}(D_v)$
for some division algebra $D_v$ of degree $d_v$ with $d_v r_v = n$.  Let
$\calO_v$ be an order of $D_v$ containing the unramified field extension of $F_v$
of degree $d_v$ (e.g., $\calO_v$ is the maximal order in $D_v$).  We assume
the diagonal subgroup $(\calO_v^\times)^{r_v} \subset K_v$.
This holds, for instance, when $K_v$ is the stabilizer of a lattice of the
form ${\mathcal I}_1 \oplus \dots \oplus {\mathcal I}_{r_v} \subset D_v^{r_v}$ where each
${\mathcal I}_i$ is left $\calO_v$-ideal on $D_v$.

Next suppose $v$ is ramified or inert in $E/F$, so $A_v$ is split.  
Assume $G_v$ has a maximal torus $T_v \simeq 
(E_v^\times)^r \times (E_v^1)^s$ for some $r, s$ with $2r+s = n$,
 such that the integral points of $T_v$ are contained in $K_v$,
 i.e.,  $(\frako_{E_v}^\times)^r \times (E_v^1)^s \subset K_v$.
This holds for instance if $K_v$ is the stabilizer of a lattice of the form
${\mathcal I}_1 \oplus \dots \oplus {\mathcal I}_{n} \subset E_v^{n}$ where each
${\mathcal I}_i$ is a $\frako_{E_v}$-ideal (in which case $s=n$).

The above assumptions guarantee that for all $v < \infty$, 
(i) $K_v \cap Z(G_v) = \UU(1, \frako_v) = \{ a \in \frako_{E,v}^\times : a \bar a = 1 \}$,
and (ii) $\det K_v = \UU(1, \frako_v)$. 
Note for $v < \infty$, if $E_v/F_v$ is a field
then $\UU(1, \frako_v) = \UU(1, F_v) = E_v^1$, 
whereas if $E_v/F_v$ is split
then $\UU(1, \frako_v) \simeq \frako_v^\times$.

Consequently, if $\pi$ occurs in $\calA(G, K, \omega)$, then $\omega$
is a character of $\UU(1,\A)$ which is invariant under
$\UU(1,F)$ and $K \cap Z(\A) = \UU(1, \hat \frako) \UU(1, F_\infty)$.
Thus the relevant central characters for us will be characters $\omega$
of the class group $\Cl(\UU(1)) = 
\UU(1,F) \bs \UU(1, \A_f) / \UU(1, \hat \frako)$.

Any 1-dimensional representation $\pi$
occurring in $\calA(G, K)$ is of the form $\pi = \chi \circ \det$, where 
$\chi$ is a character of $\UU(1, \A)$. 
From \cref{det-lem} and our assumptions on $K$, we in fact see that $\chi$
must be a character of $\Cl(\UU(1))$.

We can apply \cref{gen-prop} or \cref{gen-cor} to construct
congruences on $\calA(G, K)$.
However, since $\calA(G, K)$ admits many 1-dimensional representations in general,
even with trivial central character, we need more to guarantee
we get congruences with non-abelian forms.

\subsubsection{Congruence modules} \label{sec:cong-mod}

Fix a finite abelian group $H$ and let $L$ be a number field which contains
all character values for $H$.  Let $X(R)$ be the set of $R$-valued class functions
for $R = \Z$ or $R=L$.  Endow $X(L)$ with the usual inner product $( \cdot , \cdot )$.
Decompose $X(L) = X_\one(L) \oplus X_0(L)$ where $\one$ is the trivial character
of $H$ and $X_\one(L) = L \one$.  
Let $X_\one(\Z) = X_\one(L) \cap X(\Z) = \Z \one$ and 
$X_0(\Z) = X_0(L) \cap X(\Z)$.
Also, let $X^\one(\Z)$ (resp.\ $X^0(\Z)$) 
be the image of the orthogonal projection $X(\Z) \to X_\one(L)$ (resp.\ $X(\Z) \to X_0(L)$).
Then $X_\one(\Z) \oplus X_0(\Z) \subset X(\Z) \subset X^\one(\Z) \oplus X^0(\Z)$.
We consider the congruence module $C_0(H) = X(\Z)/(X_\one(\Z) \oplus X_0(\Z))$.  One readily sees that the projection $X(L) \to X_\one(L)$
induces an isomorphism 
$C_0(H) \simeq X(\Z)/(X_\one(\Z) \oplus X_0(\Z)) \simeq X^\one(\Z)/X_\one(\Z)$.
One similarly has an isomorphism with $X^0(\Z)/X_0(\Z)$.

\begin{lem} For a positive integer $n$, there exists $\phi \in X_0(\Z)$ such that
$\phi \equiv \one \mod n$ if and only if $C_0(H)$ contains an element of order $n$.
\end{lem}

\begin{proof}
First note if $\phi \in X_0(\Z)$ such that $\phi \equiv \one \mod n$, 
then the projection of
$\frac 1n ( \phi - \one )$ to $X_\one(L)$ is the element $- \frac 1n \one \in
X^\one(\Z)$, and thus gives an element of order $n$ in $C_0(H)$.
Conversely, suppose $\psi \in X(\Z)$ is an element of order $n$ in 
$C_0(H)$.  Then we can write $\psi = \frac an \one - \frac 1n \phi$
where $a \in \Z$ and $\phi \in X_0(\Z)$.  Since projection
gives the isomorphism $C_0(H) \simeq X^\one(\Z)/\Z \one$,
$\frac an$ has order $n$ mod $\Z$.  Thus after scaling $\psi$
(and correspondingly $\phi$) we may assume
$a \equiv 1 \mod n$.  Then $\phi \equiv  \one \mod n$.
\end{proof}

\begin{lem} As $\Z$-modules, $C_0(H) \simeq H$.
\end{lem}

\begin{proof} First suppose that $H = H_1 \times H_2$. 
For $i=1, 2$, write $X(R; H_i)$, $X_0(R; H_i)$, etc.\ for the corresponding
objects for the group $H_i$.  It is not hard to see that 
$X(\Z) = X(\Z; H) = \{ \phi_1 \otimes \phi_2 : \phi_i \in X(\Z; H_i) \}$.
Thus we may identify $X(\Z; H) = X(\Z; H_1) \oplus X(\Z; H_2)$.  
This identifies the $\Z$-submodule
$X_\one(\Z; H) \oplus X_0(\Z; H)$ with $X_\one(\Z; H_1) \oplus X_0(\Z; H_1)
\oplus X_\one(\Z; H_2) \oplus X_0(\Z; H_2)$.
 Hence $C_0(H) \simeq C_0(H_1) \oplus C_0(H_2)$.
 This reduces the proof to the case that $H = \< g \>$ is cyclic of order $n$, which
we assume now.

If $\chi_1, \ldots, \chi_n$ are the irreducible characters of
$H$, then $\frac 1n (\chi_1 + \dots + \chi_n) \in X(\Z)$.  Hence $\frac 1n \one \in
X^\one(\Z)$.  Conversely, suppose $\frac 1m \one \in X^\one(\Z)$.  Then
there exists $\phi \in X_0(\Z)$ such that $\phi \equiv \one \mod m$.  Let
$a_j = \phi(g^j)$ for $1 \le j \le n$.  Then $n \cdot (\chi, \one) = \sum a_j = 0$ but
also $\sum a_j \equiv n \mod m$, hence $m | n$.  Therefore $C_0(H) \simeq
X^1(\Z)/\Z \one \simeq H$.
\end{proof}

The relevant consequence for us is the following.  Let $e(H)$ denote
the exponent of a finite group $H$: if $p^r \nmid e(H)$ then there is no
congruence mod $p^r$ between the trivial character of $H$ and any
$\Z$-valued linear combination of the non-trivial characters of $H$.

\begin{prop} \label{prop:non-abel} Let $h^1_E = |\Cl(\UU(1))|$ and $e^1_E$ be the
exponent of $\Cl(\UU(1))$.
Suppose $p | \frac{m(K)}{\gcd(n, e^1_E) h^1_E}$ and $n$ is odd.
Then there is a non-abelian eigenform $\phi \in \calA(G, K, 1)$ such that
$\phi$ is Hecke congruent to $\one \bmod p$.
\end{prop}

\begin{proof}
By our assumptions on $K$, we have $K_Z = \UU(1, \hat \frako) U(1, F_\infty)$,
so $m(K_Z) = |\Cl(\UU(1))|$.
Thus  \cref{gen-cor} says there exists an eigenform $\phi \in \calA_0(G, K, 1)$
which is Hecke congruent to $\one \bmod p$.  We want to show we can
take $\phi$ to be non-abelian.  

Let $\bar G = G/Z$ and $\bar K = Z(\A) K/Z(\A)$.
Note the abelian elements of $\calA(G, K, 1) = \calA(\bar G, \bar K)$ are
generated by the characters $\chi \circ \det$ where $\chi$ is a character of 
$\Cl(\UU(1))$ of order dividing $n$.  We may view such $\chi$ as factoring through
the largest quotient $H$ of $\Cl(\UU(1))$ of exponent dividing $n$.

Recall that the existence of such a $\phi$ arose from an integral element
$\phi' \in \calA_0^\Z(\bar G, \bar K)$ such that $\phi' \equiv \one \bmod p^r$ where $r=v_p(m( \bar K)) = v_p(m(K))-v_p(h_E^1)$.
For a suitable rationality field $L$, decompose $\calA_0^L(\bar G, \bar K) = X_1(L) \oplus
X_2(L)$ where $X_1(L)$ consists of the abelian forms orthogonal to $\one$ 
and $X_2(L)$ is spanned
by the non-abelian eigenforms.   

We claim $\phi' \not\in X_1(L)$.  
Since $\det: G(\A) \to \UU(1,\A)$  is surjective, our assumptions on $K$
imply that $\det$ induces a surjective map $\Delta: \Cl(\bar K) \to H$. 
Thus if we had $\phi' \in X_1(L)$, 
composing it with $\Delta$ gives $\Z$-valued class function $\psi$ on $H$ such that 
$\psi \equiv \one \bmod p^r$.  But this is impossible by the above lemmas as
$v_p(\frac{p^r}{\gcd(n,e_E^1)}) > 0$ implies $p^r$ does not divide the exponent of $H$.

Hence $\phi'$ has nonzero projection to $X_2(L)$.
Therefore applying the lifting lemma, \cref{lift-lem}, with $W=X_2(L)$, we obtain
an eigenform $\phi \in X_2(L)$ which is Hecke congruent to $\one \bmod p$.
\end{proof}

\subsubsection{Non-endoscopic congruences} \label{sec:non-endo-cong}
We now define the notion of congruences on the quasi-split form $G'$.
For convenience, we talk about congruences of representations.
Suppose $K' = \prod K_v'$ is an open compact subgroup of $G'$ which is 
hyperspecial at all $v \not \in S$, and $\pi$ and $\pi'$ are automorphic
representations of $G'(\A)$ which are $K_v'$-unramified at all $v \not \in S$.
For $\alpha_v \in G'_v$, we let $\lambda_{\alpha_v}(\pi)$ be the
eigenvalue of the local Hecke operator $K_v' \alpha_v K_v'$ on $\pi_v^{K_v'}$.
We say $\pi$ and $\pi'$ are Hecke congruent (away from $S$) mod $p$ if 
$\lambda_{\alpha_v}(\pi) \equiv \lambda_{\alpha_v}(\pi') \bmod \frakp$
for some prime $\frakp$ of $\bar \Q$ above $p$ and all $v \not \in S$,
$\alpha_v \in G'_v$.

Consider the simple parameter $\psi_0 = 1 \boxtimes \nu(n) \in \Psi_2(G', 
\mathrm{std})$, where $\nu(n)$ is the irreducible $n$-dimensional representation of
$\SU(2)$.  This is the parameter of the trivial representations $\one_G$ and
$\one_{G'}$ of $G$ and $G'$. 
The base change of $\one_{G'}$
to $\GL_n(\A_E)$ is the residual contribution of the Eisenstein
series induced from the $\delta_{\GL(n)}^{-1/2}$ of the Borel.  
The unramified Hecke eigenvalues for $\one_{G'}$ are given by
\eqref{eq:congvol} and \eqref{eq:congvol-Un}.

\begin{thm} \label{thm:main-Un}
Suppose $n$ is an odd prime and assume \eqref{ECU} for $n$.
Let $A/E$ be a degree $n$ central simple algebra which is division at a non-empty
set $\mathrm{Ram}_0(A)$ of finite places of $F$ which split in $E/F$, and let $S_0 \subset \mathrm{Ram}_0(A)$.
Consider a definite unitary group $G=\UU_A(n)$ over $A$ as above.
Let $K = \prod K_v \subset G(\A)$ be a compact open subgroup satisfying 
the assumptions at the beginning of this section, and also assume that
$K_v = G^1(F_v)$ for $v \in S_0$.

Suppose that $p | \frac{m(K)}{\gcd(n,e_E^1)h_E^1}$.  
Then there exists a cuspidal automorphic representation $\pi$ 
of $G'(\A)$ with trivial central character 
such that: (i) the base change $\pi_E$ is cuspidal, (ii)
$\pi_{v_0}$ is an unramified twist of Steinberg for $v_0 \in S_0$; (iii) $\pi_v$ has a nonzero $K_v$-fixed vector
when $G(F_v) \simeq G'(F_v)$;
(iv) $\pi_v$ is a holomorphic weight $n$ discrete series for $v | \infty$;
and (v) $\pi$ is Hecke congruent to $\one_{G'} \bmod p$.
\end{thm}

Note that by the classification of central simple algebras over number fields
and Landherr's theorem, given $E/F$ and any non-empty finite set $\Sigma$
of finite places of $F$ split in $E/F$, there exists $G=\UU_A(n)$ as in 
Theorem \ref{thm:main-Un} with $\Ram_0(A) = \Sigma$.

We make a few remarks on such a $\pi$ as in the theorem.
First, it  cannot arise as an endoscopic lift
from smaller unitary groups, so this congruence is ``native'' to $\UU(n)$.  
Second, by the central character condition, (ii) means $\pi_{v_0}$ is a twist of
Steinberg by an unramified character of order dividing $n$.
Also, (iii) implies $\pi_v$ will be unramified whenever
$K_v$ is hyperspecial.  Moreover, if every finite place $v \not \in S_0$
satisfies $G(F_v) \simeq G(F'_v)$, and if $K_v$ is good special maximal compact
subgroup at all of these places, then we have strong control over $\pi$ at all
places: (iv) describes $\pi_\infty$ completely; (ii) says $\pi_v$ is an unramified
twist of Steinberg for $v \in S_0$, and 
(iii) says $\pi_v$ is $K_v$-spherical at all remaining $v$.

\begin{proof}
First \cref{prop:non-abel} tells us there exists a non-abelian eigenform
$\phi \in \calA(G, K, 1)$ which is Hecke congruent to $\one \bmod p$.
Let $\sigma$ be the associated automorphic representation of $G(\A)$.
By \cref{prop:cusp}, we know $\sigma_E$ is cuspidal.
We may take $\pi \in \JL(\sigma)$ such that (iv) holds
 and $\pi_v \simeq \sigma_v$
when $G(F_v) \simeq G'(F_v)$.  For $v \in S_0$,
since $K_{v_0} = G_{v_0}^1$ we must have that 
$\sigma_{v_0}= \chi_{v_0} \circ \det$, where $\chi_{v_0}$ is an unramified 
character of $F_{v_0}^\times$, so
the local Jacquet--Langlands correspondent $\pi_{v_0}$ is Steinberg twisted
by $\chi_{v_0}$.
Finally, $\pi$ satisfies (v) because $\one_{G'}$ has the same Hecke eigenvalues
as $\one \in \calA(G,K,1)$ at almost all places.
\end{proof}

We now describe $m(K)$ for nice maximal compact subgroups $K$
using \cite{GHY}.  For simplicity we restrict to odd $n$. 
If desired, one can obtain masses for smaller
compact subgroups $K' \subset K$ by recalling  that $m(K') = [K:K'] m(K)$.
Let $\chi_{E/F}$
be the quadratic idele class character of $F$ associated to $E/F$.

\begin{prop} \label{prop:mass-Un}
Let $G=\UU_A(n)$ be a definite unitary group over $A$ where $n$ is odd.
Let $\Ram_f(E)$ (resp.\ $\Ram_f(A)$) denote the set of finite primes of
$F$ above which $E$ (resp.\ $A$) is ramified.  Assume $A_w$ is division for
each $w$ above $v \in \Ram_f(A)$.  Let $S = \Ram_f(E) \cup \Ram_f(A)$.
Take $K = \prod K_v$ such that $K_v$ is maximal hyperspecial for finite $v \not \in S$,
$K_v = G^1(F_v)$ for $v \in \Ram_f(A)$, 
$K_v$ is the stabilizer of a maximal lattice for $v \in \Ram_f(E)$,
and $K_v = G(F_v)$ for $v | \infty$.
 Then
\begin{equation} \label{eq:mass-Un}
m(K) = 2^{1-nd-|\Ram_f(E)|}  \times 
\prod_{r=1}^n L(1-r, \chi_{E/F}^r) \times \prod_{v \in \Ram_f(A)} 
\left( \prod_{r=1}^{n-1} (q_v^{r} - 1) \right),
\end{equation}
where $d = [F:\Q]$. 
\end{prop}

\begin{proof}
A general mass formula is given in \cite[Proposition 2.13]{GHY}, which is
explicated for definite odd unitary groups over fields in Proposition 4.4 of \emph{op.\ cit.}
From those calculations, it follows that
\[ m(K) = 2^{1-nd} \times \prod_{r=1}^n L(1-r, \chi_{E/F}^r) \times \prod_{v \in S} \lambda_v, \] 
where $\lambda_v$ is as follows.  For a finite place $v$, let $\underline H_v'$ be Gross's canonical integral model of $H_v:=G'_v$.  Let $\underline G_v$ be the smooth
integral model associated to a parahoric such that $K_v = \underline G_v(\frako_v)$.
By our hypotheses, $S$ is the set of finite places such that 
$\underline G_v \not \simeq \underline H_v^0$.  
Let $\bar G_v$ and $\bar H^0_v$ be the maximal reductive quotients of the
special fibers of $\underline G_v$ and $\underline H_v^0$, which are reductive groups
over $k_v = \frako_v/\frakp_v$, with $\bar G_v$ possibly being disconnected.
Then for $v \in S$,
\[ \lambda_v = \frac{q_v^{-N(\bar H_v^0)} | \bar H_v^0(k_v) |}
{q_v^{-N(\bar G_v)} | \bar G_v(k_v) |}, \]
where $N( \cdot)$ denotes the number of positive roots over $\bar k_v$.
When $G_v$ is quasi-split, \emph{loc.\ cit.}\ tells us $\lambda_v = \frac 12$ if $E_v/F_v$
is ramified.  

So we need only to compute $\lambda_v$ for $v \in \Ram_f(A)$.
In this case $v$ splits in $E/F$ so $\bar H_v^0 \simeq \GL_n(k_v)$.
Let $\calO_v = A_v$ and $\mathfrak P_v$ the prime ideal of $\calO_v$.  Then
$G_v \simeq \calO_v^\times/(1+ \mathfrak P_v) \simeq \F_{q_v^n}^\times$, which gives
 $\lambda_v = q_v^{-n(n-1)/2} \prod_{r=1}^{n-1} (q_v^n - q_v^r) =
\prod_{r=1}^{n-1} (q_v^{r} - 1)$. 
\end{proof}

\begin{rem} By \cite{GHY}, we can extend the formula \eqref{eq:mass-Un} to include 
finite places $v$ such that $G_v$ is quasi-split and $K_v$ 
is a special but not hyperspecial maximal
compact.  Each such place will contribute a factor of
$\lambda_v = \frac{q^n+1}{q+1}$ to $m(K)$.
\end{rem}

Consequently, \cref{thm:main-Un} gives non-endoscopic Eisenstein congruences
mod $p$ which are Steinberg at $v$ whenever
$p$ is a sufficiently large (depending on $n$ and $E/F$)
prime dividing some $q_v^r -1$ (for $1 \le r \le n-1$).

\begin{ex} \label{ex:Qi} Suppose $F=\Q$, $E=\Q(i)$.  Then
$|\Cl(\UU(1))| = 1$.  Let $A/E$ be a central division algebra of
odd prime degree $n=2m+1$  which is 
ramified only at the primes of $E$ above a fixed rational 
prime $\ell \equiv 1 \bmod 4$ (so necessarily division at $w | \ell$).
Write $\chi = \chi_{E/F}$.  
It is well known that $L(1-r,\chi^r)=-\frac 1r B_{r,\chi^r}$ (generalized
Bernoulli number).
Thus taking $G$ and $K$ as in \cref{prop:mass-Un}, we get
\[ m(K) = \frac {1}{2^{n} n!} \prod_{r=1}^m B_{2r} \times
\prod_{r=1}^m B_{2r+1, \chi} \times \prod_{r=1}^{n-1} (\ell^r - 1). \]

Suppose $p > n$ is a prime dividing some $\ell^r - 1$ where $1 \le r \le n-1$.
Since $p > n$, the von Staudt--Clausen theorem tells us that
 $p$ does not divide the denominators of any of the Bernoulli numbers
$B_2, B_4, \dots, B_{2m}$.  Also
$B_{1, \chi}, B_{3, \chi}, \dots, B_{n, \chi}$ all have denominator $2$.
Hence \cref{thm:main-Un} yields a non-endoscopic
holomorphic weight $n$ cuspidal representation
$\pi$ of $G'(\A) = \UU(n,\A)$ Hecke congruent to $\one_{G'} \bmod p$ 
such that $\pi$ is (i) unramified at each odd finite $v \ne \ell$, (ii)
spherical at $v=2$, and (iii) an unramified twist of Steinberg at $v=\ell$.
(By working with smaller compact subgroups $K$, one can remove the condition
$p > n$.)

The same result is true for some additional values of $p$, independent of $\ell$,
coming from numerators of Bernoulli numbers.  For instance,
we can always take $p=61$ for $7 \le n \le 59$ as $61 | B_{7,\chi}$; 
we can take $p \in \{ 277, 2659 \}$ if $11 \le n < p$ as $277 \cdot 2659 | B_{9,\chi}$; 
we can take $p = 19$ if $n= 13, 17$ as $19 | B_{11,\chi}$;
or we can take $p\in \{43, 691, 967 \}$ if $13 \le n < p$ as $691 | B_{12}$ and
$43 \cdot 97 | B_{13,\chi}$.
\end{ex}

\subsubsection{Congruences from definite unitary groups over fields}
\label{sec:Un-field}

The reason to work with definite unitary groups over division
algebras $A$ is to guarantee we get a representation $\pi$ with an Eisenstein
congruence such that $\pi$ is non-endoscopic.  
For instance, the congruences coming from Bernoulli numbers in \cref{ex:Qi}
also occur on the definite unitary group $G = \UU_A(3)$ where $A = M_3(\Q(i))$.
Namely, if $p > n=2m+1$ and $p | B_{2r}$ or $p | B_{2r+1, \chi}$ for some
$1 \le r \le m$, we still get an Eisenstein congruence mod $p$ 
for some non-abelian $\pi$ occurring in 
$\calA_0(G,K, 1)$ where $K_v$ is maximal everywhere.  
However, it may be that $\pi$ is a lift from a smaller group, and that
this congruence may be explained either as
the symmetric square lift of a weight 2 Eisenstein congruence
or as the Kudla lift of a weight 3 Eisenstein congruence as in  \cref{sec:kudla}.

First we state a sample result using definite unitary groups over fields. Unspecified notation is as in \cref{thm:main-Un}.

\begin{prop} \label{field-cong} Let $E/F$ be a CM extension of number fields,  
$n$ a positive
integer, and $G$ the definite unitary group preserving $\Phi = I \in M_n(E)$.
Take $\mathfrak l$ to either be $\frako_E$ or a prime ideal of $E$ not ramified in
in $E/F$.  Let $\ell \in \mathbb N$ be the absolute norm of $\mathfrak l$.
Let $K = \prod K_v$ where $K_v = G_v$ for $v | \infty$, $K_v$ 
is the stabilizer of a maximal lattice for all finite $v \ne \mathfrak l$  (so all finite $v$ if $\ell = 1$) such that $K_v$ is hyperspecial at all such unramified $v$, 
and $K_{\mathfrak l}$ be the subgroup of $G(\frako_{\mathfrak l})$
fixing the (localization of the) lattice $\Lambda = \frako_E \oplus \dots \oplus \frako_E \oplus \mathfrak l \frako_{E} \subset E^n$.
If $n$ is odd, put $\lambda_v = \lambda_v(n) = \frac 12$.  If $n$ is even,
put $\lambda_v = \lambda_v(n) = 1$ or $\frac 12 \cdot \frac{q^n-1}{q+1}$ according to
whether $(-1)^{n/2} \in N(E_v^\times)$ or not. 
Suppose $p$ divides 
\begin{equation} \label{Unfieldmass}
 m(K) = 2^{1-nd}  \times\frac{\ell^n - \chi_{E/F}(\ell)^n}{\ell - \chi_{E/F}(\ell)} \times \prod_{v \in \Ram_f(E)} \lambda_v \times \prod_{r=1}^n L(1-r, \chi_{E/F}^r).
\end{equation}
Then there exists an eigenform $\phi \in \calA_0(G, K, 1)$ which is Hecke congruent
to $\one \mod p$.  The factor $\frac{\ell^n - 1}{\ell - 1}$ is interpreted
to be $1$ if $\ell = 1$.
\end{prop}

\begin{proof} 
The only thing to do is check the above formula for $m(K)$.  This
follows from the formulas in \cite{GHY} together with the
fact that $[G(\frako_{\mathfrak l}) : K_{\mathfrak l}] = (\ell^n-1)/(\ell - 1)$
(resp.\ $(\ell^n-(-1)^n)/(\ell+1)$) if
$\mathfrak l$ is split (resp.\ unramified inert) in $E$.  Note 
$[G(\frako_{\mathfrak l}) : K_{\mathfrak l}]$ can be computed as 
the index of an $(n-1, 1)$-type maximal parabolic inside 
$\GL_n(\mathbb F_\ell)$ (resp.\ $\UU_n(\mathbb F_\ell)$).
\end{proof}

We can numerically compute examples in Magma \cite{magma} on
definite $\UU(3)$ (or more generally $\UU(n)$) over a field
using Markus Kirschmer and David 
Lorch's code for Hermitian lattices.\footnote{Available at: \url{http://www.math.rwth-aachen.de/~Markus.Kirschmer/magma/lat.html}}
We compute unramified eigenvalues on $\UU(n)$
at both split and inert places using the lattice method described in
\cite{greenberg-voight}.  For convenience, we calculate with
the standard definite form $G$ of $\UU(n)$ over $F=\Q$, taking
$\Lambda = \frako_E \oplus \dots \oplus \frako_E \oplus  \mathfrak l \frako_{E}$, 
 $K_v$ to be stabilizer of $\Lambda_v$ at all finite $v$, and $K_\infty =
 \UU_3(\R)$.  In this case, denote $K(\ell) := K = \prod K_v$ where $\ell$
 is the absolute norm of $\mathfrak l$.
 Note $K_v$ is of the form in given the proposition except possibly
 at the finite $v \in \Ram_f(E)$ as now $K_v$ need not be maximal.
 But since $K(\ell)$ is contained in some $K'$ of the form in the proposition, 
 we can still compute congruence examples
as in the proposition on the potentially larger spaces of level $K(\ell)$. 
 
Our calculations suggest that
the Eisenstein congruences from the $n=3$ case of \cref{field-cong} 
can also be explained by \cref{th:Kudlalift} as Kudla lifts of weight 3 
Eisenstein congruences for $\GL(2)$.
We numerically check if a form $\phi$ on $\UU(3)$ is a Kudla lift of a
weight 3 elliptic form $f$ by comparing Hecke eigenvalues.
Specifically, from computing local Hecke operators in \cref{sec:local}, it follows that the $p$-th Hecke eigenvalue
of $\phi$ is $a_p(f)+p$ at split primes $p$ and $a_p(f)^2+2p^2+p-1$
at unramified inert places $p$.

\begin{ex} Taking $n=3$, $E/F=\Q(i)/\Q$, and $K=K(\ell) = K(5)$
we get $m(K)=\frac{31}{384}$, which gives an eigenform
$\phi \in \calA_0(G, K, 1)$ with an Eisenstein congruence mod $31$.
(Here $K(5)$ is of the form in the proposition.)
We compute that $\calA_0(G, K, 1)$ is
spanned by $2$ forms, say $\phi$ and $\bar \phi$, 
which are Galois conjugate over $\Q(\sqrt 5)$.
We may choose $\phi$ so that the first few eigenvalues at split primes $v$
for $\one$ and $\phi$ are given by:
\begin{center}
\begin{tabular}{c|cccc}
$v$ & $13$ & $17$ & $29$ & $37$ \\
\hline
$\one$ & $183$ & $307$ & $871$ & $1407$ \\ 
$\phi$ & $9-2\sqrt 5$ & $11+ 8 \sqrt 5$ & $27 - 4 \sqrt 5$ &
$41 - 10 \sqrt 5$
\end{tabular}
\end{center}
and the first few eigenvalues at unramified inert primes $v$ are given by
\begin{center}
\begin{tabular}{c|cccc}
$v$ & $3$ & $7$ & $11$ & $19$ \\
\hline
$\one$ & $84$ & $2408$ & $14652$ & $130340$ \\ 
$\phi$ & $10+2\sqrt 5$ & $50+10\sqrt 5$ & $52-88 \sqrt 5$ &
$580-32 \sqrt 5$
\end{tabular}
\end{center}
These eigenvalues for $\one$ and $\phi$ are congruent mod 
$\frakp = (\frac{1+5\sqrt 5}2)$.
Note this congruence comes from the factor $\frac{\ell^3-1}{\ell-1}$ in the
mass formula.

Numerically, $\phi$ agrees the Kudla lift of a newform 
$f \in S_3(20, \chi_{-4})$.  The rationality field of $f$
is a CM extension of $\Q(\sqrt{5})$ and $f$ has a mod $31$ congruence 
with $E_{3,1,\chi_{-4}} \in M_3(20,\chi_{-4})$.
\end{ex}

\begin{ex} Taking $n=3$, $E/F=\Q(\sqrt{-19})/\Q$ and $K = K(1)$,
we get $m(K) = \frac{11}{48}$.
Here $\calA_0(G,K,1)$ is spanned by $2$ rational eigenforms. One of these forms,
which we call $\phi$, has a congruence mod $11$ with $\one$.  Note this
congruence comes from the numerator of the Bernoulli number
$B_{3,\chi_{-19}}$.

The space $S_3(19,\chi_{-19})$ has dimension $3$, spanned by 
newforms $f, \bar f, g$, where $f$ has rationality field $\Q(\sqrt{-13})$.
Numerically, $\phi$ agrees with the Kudla lift of $f$ (and of $\bar f$),
and $f$ has mod $11$ congruences with the Eisenstein
series $E_{3,1,\chi_{-19}}$ and $E_{3,\chi_{-19},1}$.

(In this example $K(1)$ is not of the form in the proposition, as the global lattice
$\frako_E^{\oplus 3}$ is not maximal, but it has the same mass as a maximal lattice $\Lambda'$.  
However, taking $K'$ to be the stabilizer of $\prod_q \Lambda'_q$, 
one only sees the Kudla lift of $g$ in $\calA_0(G,K')$ and not that of $f, \bar f$.)
\end{ex}

Additional calculations together with known results and expectations
about GL(2) Eisenstein congruences suggest that the congruences
in \cref{field-cong} always arise as endoscopic lifts.  Specifically, we expect
the following GL(2) Eisenstein congruences.
For simplicity, we only explicate this $F=\Q$.

\begin{conj} \label{conj:wtk}
Let $k \ge 2$ be an integer, $E$ be an imaginary quadratic field of discriminant $D$, and $\ell = 1$ or $\ell \nmid D$ a prime.  Define $\lambda_q = \lambda_q(k)$ as in \cref{field-cong}.
If $p$ divides 
\begin{equation} \label{eq:wtk-mass}
\frac 1{2^{k - 1}k!} \times \frac{\ell^k - \chi_{D}(\ell)^k}{\ell - \chi_{D}(\ell)} 
\times \prod_{q \in \Ram_f(E)} \lambda_q \times \prod_{r=1}^k B_{r,\chi_D^r},
\end{equation}
then there exists
an eigenform $f \in S_{k'}(\ell |D|, \chi_{D}^{k'})$ which is Hecke congruent mod
$p$ to the Eisenstein series $E_{k', 1, \chi_D^{k'}} \in M_{k'}(\ell |D|, \chi_{D}^{k'})$ for some $2 \le k' \le k$.

Moreover we may take $k=k'$ in the following situations.
If $\Ram_f(E)$ contains a single prime and $k$ is odd, 
denote this prime by $p_D$; otherwise
let $p_D = 1$.  Let $N$ be the conductor of $\chi_D^k$, and put $\delta = |\Ram_f(E)|$ if $k$ is odd and $\delta = 0$ if $k$ is even.

\begin{enumerate}[(i)]
\item If $\ell = 1$, and $p$ divides
\[ \frac 1{2^{k - 1 + \delta} p_D^{(k-1)/2}} \prod_{m=1}^{\lfloor (k-1)/2 \rfloor} (\prod_{(q-1) | 2m} \frac 1q ) \times \frac{B_{k,\chi_D^k}}k, \] then we may take $k'=k$ and
 $f \in S_{k}(N, \chi_{D}^k)$ to be a newform.
 
 \item If $k \ge 3$, $\ell$ is prime, and $p$ divides
 \[  \frac 1{2^{k - 1 + \delta} p_D^{\lfloor (k+1)/2 \rfloor}}
 \prod_{m=1}^{\lfloor k/2 \rfloor} (\prod_{(q-1) | 2m} \frac 1q ) \times \frac{\ell^k - \chi_{D}(\ell)^k}{\ell - \chi_{D}(\ell)}, \] 
 then we may take $k'=k$ and
  $f \in S_{k}(\ell N, \chi_{D}^k)$ to be a newform.
\end{enumerate}
\end{conj}

Note if $k$ is even, then $N=1$ and there is no dependence on
$D$ in (i) or (ii)---e.g., the final factor in (ii) is $\frac{\ell^k - 1}{\ell - 1}$.

The first part of the conjecture corresponds to the expectation
that the congruences in \cref{field-cong} can also be produced as lifts
of $\GL(2)$ Eisenstein congruences of the above type.  
Specifically, when $k=k'$ above, the posited GL(2) congruence
corresponds to a congruence in \cref{field-cong} via an
endoscopic lift as in \cref{thm:20190601}.  However, there are many
lifts from GL(2) to $\UU(n)$ (e.g., the Ikeda lift in \cref{th:oddunitikeda}) 
and not all of the congruences in \cref{field-cong} can directly be explained
by \cref{thm:20190601}.
For instance, when $D=-19$ and we have mod 11 Eisenstein congruences
on both $\UU(3)$ and $\UU(4)$ with $\ell = 1$ in \cref{field-cong}
because $11 | L(-2,\chi_{-19})$.  Correspondingly, there is a mod 11
congruence in $S_3(19,\chi_{-19})$ but not one in $S_4(19)$.  It
appears that both the $\UU(3)$ and $\UU(4)$ congruences arise as lifts
of the weight 3 congruence (the former being the Kudla lift, and the latter
not of the type considered in \cref{thm:20190601}).

The reasoning for the second part of the conjecture is as follows.
The condition in (i) means that a new factor
of $p$ arises in \eqref{eq:wtk-mass} when increasing the weight from
$k-1$ to $k$.
(In other words, the depth of the congruence in the sense of
\cref{rem:depth} should increase when going from
weight $k-1$ to weight $k$.)
Here we explicated denominators (and ignored numerators) of Bernoulli numbers, using that
the denominator of $\frac{B_{2n}}{2n}$ is $\prod_{(q-1) | 2n} q$ (where $q$ denotes
a prime) and the denominator of $\frac{B_{2n+1,\chi_D}}{2n+1}$ is $p_D$.
Similarly, for the $\ell$ prime case the divisibility condition in (ii) forces a
new factor of $p$ in \eqref{eq:wtk-mass} when compared with the $\ell = 1$
case.

For $k \ge 4$ even, we also applied the following reasoning:
if one has a mod $p$ congruence in $S_k(\ell D)$ for all $D$, almost surely
there is one in $S_k(\ell)$.  
(Note $k=2$ is special: one does not have congruences in
$S_2(\ell)$ for $p | (\ell + 1)$, but rather for $p | (\ell - 1)$, and
one really should look at the spaces $S_2(\ell D)$---cf.\
\cite{mazur}, \cite{yoo}, and \cref{sec:GL2}).  

This explicit conjecture is mostly known by now,
and our primary reason for stating the above conjecture was to
make precise the apparent connection between
mass formulas on $\UU(k)$ and weight $k$ GL(2) congruences.
For instance, generalizing work of Billerey
and Menares \cite{billerey-menares} (cf.\ \cref{sec:lift-GL2}),
Spencer \cite{spencer} showed that for a Dirichlet character $\chi$ of
conductor $N$, there is a mod $p$ Hecke congruence with $E_{k,1,\chi}$
in $S_k(\ell N, \chi)$ if $p \nmid 6N$, $\ell \nmid N$ and $p$ divides
$(p^k - \chi(p))\frac{B_{k,\chi}}k$.  Thus apart from the specification
of the exact level of $f$ in (i) and (ii), the above conjecture is known
if $p > k+1$ and $p \nmid N$.  However, at least when $k$ is even, 
one can typically deduce the exact level of $f$ in the conjecture by
comparing \cite{spencer} with work of Ribet \cite{ribet} which says that 
(for $p > k+1$) one has a mod $p$  Eisenstein congruence in $S_k(1)$ 
if and only if $p | \frac{B_k}k$.  See also \cite[Proposition 5.2]{berger-klosin}
for a similar congruence result with an argument that is valid for small primes,
but with less control over the level of $f$.  We also checked the
above conjecture numerically for small choices of parameters.

\begin{rem} \label{rem:wtk}
\cref{conj:wtk} can be extended to different ways.
First, there is an obvious extension to Hilbert modular forms using
\eqref{Unfieldmass}.
Second, one can use the same idea of checking when new factors
divide the mass to extend these conjectures to newforms of higher levels
(e.g., replace $\ell$ with a product $\ell_1 \cdots \ell_t$ by
allowing more general compact subgroups in \cref{field-cong}).  E.g., see
\cref{conj:wt4} and \cref{rem:wt2new}.
To our knowledge, there are no general results along the lines of
\cref{conj:wtk} for higher levels or for Hilbert modular forms when $k > 2$.
\end{rem}

We note one more phenomenon.
While the above conjecture says that the congruences seen in the mass
formula in \cref{field-cong} should also be seen on GL(2), we can
find examples to see the converse of this statement is not true, at least for small
$p$.

\begin{ex} \label{ex:endo-nomass}
There is a unique CM newform $f \in S_3(11, \chi_{-11})$.  It has
rational coefficients and is congruent to $E_{3,1,\chi_{-11}}$ mod $3$.
Taking $n=3$, $E=\Q(\sqrt{-11})$ and $K=K(1)$, we compute
$\dim \calA(G,K) = 2$.  Say $\calA_0(G,K) = \langle \phi \rangle$.
Numerically $\phi$ agrees with the Kudla lift of $f$
and is congruent to $\one$ mod $3$.  However, the
mass here is $\frac 1{16}$, so we do not see
this lifted congruence in the mass formula.
\end{ex}


\section{Symplectic and odd orthogonal groups} \label{sec:orthog}

Here we will discuss our construction of Eisenstein congruences
for groups of type $B_n$ and $C_n$, i.e., odd orthogonal and symplectic
groups.  For simplicity we will focus on the rank 2 case where these 
coincide---namely $\SO(5) \simeq \PGSp(4)$---and we have a more explicit
understanding of the representations.  Then we briefly discuss the higher rank
situation in \cref{sec:BCn}.  

\subsection{$\SO(5)$} 
Let $F$ be a totally real number field of degree $d$ and $B/F$ a totally definite quaternion
algebra.  Consider the projective similitude group $G = \mathrm{PGU}(2,B)$, which is an
inner form of the quasi-split $G' = \SO(5) \simeq \PGSp(4)$ that is compact at infinity.
We consider a maximal compact subgroup $K = \prod K_v$ as follows.  Let $S$
be a finite set of finite places containing all finite places where $B$ is ramified.  
Take $K_v = G(F_v)$ for $v | \infty$ and $K_v$
hyperspecial maximal compact for finite $v \not \in S$.  Let $S_1$ (resp.\ $S_2$)
be the set of finite places which are ramified in $B$ such that $K_v$ stabilizes
a maximal lattice in the principal genus (resp.\ non-principal genus).  Let $S_3$
be the set of places in $S$ which are split in $B$, and assume $K_v$ is not hyperspecial
for $v \in S_3$, i.e., $K_v$ is the level $\frakp_v$ paramodular subgroup of $\PGSp_4(F_v)$.

\begin{prop} \label{SO5-mass}
Suppose $p$ divides the numerator of
\[ m(K) = 2^{1-2d-|S|} | \zeta_F(-1) \zeta_F(-3)| \prod_{v \in S} \lambda_v, \]
where
\[ \lambda_v = 
\begin{cases}
(q_v+1)(q_v^2+1) & \text{if } v \in S_1, \\ 
q_v^2 - 1 & \text{if } v \in S_2, \\
q_v^2 + 1 & \text{if }  v \in S_3.
\end{cases} \]
Then there exists an eigenform $\phi \in \calA_0(G,K)$ which is Hecke congruent
to $\one \mod p$.
\end{prop}

Note that in all cases the local factors $\lambda_v$ are factors of $q_v^4 - 1$.

\begin{proof}
By \cref{gen-prop}, it suffices to prove the mass formula.  This mass formula
follows from the mass formula for $\SO(5)$ in \cite{GHY}.  In the case of $F=\Q$,
the mass formula was given earlier in the context of $\mathrm{PGU}(2,B)$ by 
Hashimoto and Ibukiyama  (\cite{hashimoto-ibukiyama2}, \cite{hashimoto-ibukiyama3}).
\end{proof}

As with the case of unitary groups, there should be a generalized Jacquet--Langlands
correspondence associating automorphic representations of $G(\A)$ to automorphic
representations of $G'(\A)$.  This would follow from the endoscopic classification of
representations of $G(\A)$, as conjectured in \cite[Chapter 9]{arthur:book} (see also
\cite{ralf:packet} for an explicit description of this classification for $\SO(5)$).
However, even for $\SO(5)$ this has not been completed, though some instances
of the Jacquet--Langlands correspondence are known by \cite{sorensen:2009}.

We would like to know whether this method can give cuspidal, 
and if possible non-CAP, representations of $G'(\A)$ with Eisenstein congruences.
The local archimedean parameter corresponding to the trivial representation for $G(\R)$
matches the local parameter $\rho_3$ corresponding to the packet containing the holomorphic discrete series $\mathcal D_3$ of scalar weight 3 for $G'(\R)$.  This packet
has size 2, and the other (generic) member $\mathcal D_{(3,-1)}$ is the large discrete series with minimal $K$-type $(3,-1)$.
Hence we expect the Eisenstein congruences from \cref{SO5-mass} to transfer to Eisenstein congruences for holomorphic or non-holomorphic Siegel modular forms of archimedean type $\rho_3$---specifically, discrete automorphic representations 
$\pi'$ of $G'(\A)$ such that either $\pi'_v \simeq \mathcal D_3$ for all $v | \infty$,
or for all but one $v | \infty$ and the remaining archimedean
component is $\mathcal D_{(3,-1)}$.

To be more explicit about the weight, 
for a global Arthur parameter $\psi$ for $G'(\A)$, one in general
needs a parity condition to be satisfied for a global tensor product $\pi' = \bigotimes \pi'_v$
to lie in the discrete spectrum of $G'(\A)$.  This parity condition tells us whether
we can take $\pi'$ to be holomorphic or not.
We will say $\pi'$ (or $\psi$) is of Saito--Kurokawa type if the Arthur parameter $\psi$ 
is of the
form $\mu \boxplus (\chi \otimes \nu(2))$, where $\mu$ is a cuspidal representation of
$\GL_2(\A)$, $\chi$ a quadratic idele class character of $\GL_1(\A)$ and $\nu(2)$
is the irreducible 2-dimensional representation of $\SU(2)$.  It follows from
\cite{ralf:packet} that, if $\pi'$ has a paramodular fixed vector for all finite $v$,
then the parity condition means we can take $\pi_\infty$ to be holomorphic
if and only if $\psi$ is not of Saito--Kurokawa type $\mu \boxplus (\chi \otimes \nu(2))$
for some $\mu, \chi$ such that $\eps(1/2, \mu \otimes \chi) = -1$.

While the behavior of level structure (existence of fixed vectors under explicit
compact subgroups) 
along the conjectural Jacquet--Langlands
correspondence for $\SO(5)$ is not understood even conjecturally in general,
Ibukiyama (\cite{ibukiyama:1985}, \cite{ibukiyama:2018}) 
has made global conjectures that suggest how the level 
structure behaves for maximal compact $K$ as above.
Indeed, it follows from the tables \cite{ralf:appendix} and \cite{ralf-brooks}
that the (principal genus stabilizer) $K_v$-level structure on $G(F_v)$ 
for $v \in S_1$ corresponds to being new of Klingen level $\frakp_v$ on $G'(F_v)$,
and the (non-principal genus stabilizer) $K_v$-level structure on $G(F_v)$ for
$v \in S_2$ corresponds to paramodular level structure on $G'(F_v)$.
We remark the paramodular case of Ibukiyama conjectures were recently
proven by \cite{vanHoften}.
Hence \cref{SO5-mass} suggests the following congruence statement about
Siegel modular forms.

\begin{conj} \label{SO5-conj}
For $p | m(K)$ in the notation of \cref{SO5-mass}, there exists
a holomorphic or non-holomorphic degree $2$ 
Siegel modular form $f$ of archimedean type $\rho_3$ such that (i) $f$ is unramified outside
$S$; (ii) $f$ has Klingen level $\frakp_v$ for $v \in S_1$; (iii) $f$ has paramodular
level $\frakp_v$ for $v \in S_2 \cup S_3$; and (iv) $f$ is Hecke congruent mod $p$ to
the trivial representation $\one$ 
of $\PGSp(4)$.
\end{conj}

In the next section, we will suggest that, at least when $S_1 = \emptyset$,
there is always a holomorphic Saito--Kurokawa form $f$ satisfying the above 
congruence.  

We expect the trivial representation has the same spherical
Hecke eigenvalues as a suitable weight 3 degree 2
Siegel Eisenstein series.

\subsection{Lifting GL(2) Eisenstein congruences} \label{sec:lift-GL2}

Here we will explain how, under certain conditions, one can also construct 
congruences predicted in \cref{SO5-conj} via Saito--Kurokawa lifts.
First we state recent results on GL(2) Eisenstein congruences in higher weight.

\begin{prop}[\cite{billerey-menares}, \cite{gaba-popa}] \label{thm:BM}
Let $k \ge 2$, $p \ge \min\{ 2k-1, 5 \}$, and $\ell$ a prime.  If $p | (\ell^{2k}-1)$, then
there exists a newform $f \in S_{2k}(\Gamma_0(\ell))$ which is congruent mod $p$ to a 
weight $2k$ Eisenstein series $E_{k,\ell} \in M_{2k}(\Gamma_0(\ell))$.
Moreover, if $p | (\ell^{k} \pm 1)$, then we may take $f$ to have Atkin--Lehner sign
$\pm 1$ at $p$.
\end{prop}

In fact, these works also give congruences if $p | (\ell^{2k-2}-1)$ with additional divisibility
criteria of Bernoulli numbers, and \cite{billerey-menares} proves an 
if-and-only-if statement (without the Atkin--Lehner sign condition) under the 
stronger hypothesis that $p \ge 2k+3$.  The specification of Atkin--Lehner sign
and the extension to $p \ge 2k-1$ was done in \cite{gaba-popa}.
  In \cite{billerey-menares}, the authors also
conjecture that their congruence result extends to squarefree level in a natural way.  Note \cref{conj:wtk}(iii) suggests an extension of this result to small
primes $p$.

We recall some facts about the Saito--Kurokawa lift.
Let $\mu$ be a discrete $L^2$-automorphic representations of 
of $\PGL_2(\A)$.  We define the Saito--Kurokawa lift
$SK(\mu)$ of $\mu$ to be a certain discrete automorphic representation of 
$\SO_5(\A)$ associated to the formal global parameter 
$\mu \boxplus (\one \boxtimes \nu(2))$ as in \cite{arthur:book}.
This corresponds to the
endoscopic functorial lifting of $(\mu, \one \boxtimes \nu(2))$
associated to the
standard embedding $\SL_2(\C) \times \SL_2(\C) \to \Sp_4(\C)$ 
from the dual group of $\PGL(2) \times \PGL(2)$ to the dual group of
$\SO(5)$.  We have not specified $SK(\mu)$ as a unique representation
in a global packet $\Pi(\mu)$.   
Under a parity condition, $SK(\mu)$ can be defined as in \cite{ralf:SK1}.

For simplicity, we just specify some information about $SK(\mu)$
in the special case $F=\Q$ and $\mu$ corresponds to an elliptic newform
$f \in S_{4}(\Gamma_0(N))$ with $N$ squarefree.
Then we may take $SK(f) := SK(\mu)$ to be unramified at primes
$p \nmid N$ and paramodular of level $p$ for $p | N$  \cite{ralf:SK2}.  
Let $\omega(N)$ be the number of primes dividing $N$.
If $\eps(1/2, f) = (-1)^{\omega(N)}$, then we may take $SK(f)$ to be a
holomorphic Siegel cusp form of scalar weight 3 \cite{ralf:SK1}.  
If $\eps(1/2, f) = 1$, then an examination of the parity condition
shows we cannot take $SK(f)$ to be holomorphic, but we can
take $SK(f)$ to be generic of weight $(3,-1)$.
(For the product $\otimes \pi_v$ of local components to occur
discretely, one needs the associated character of the component group to have order $\eps(1/2, f)$. Both the non-archimedean 
local paramodular representations and the archimedean representation $\mathcal D_{(3,-1)}$  correspond to the trivial characters
of the local component groups---cf.\ \cite{ralf:packet}.) 

Let $E_4$ be the weight 4 Eisenstein series in $M_4(\Gamma_0(1))$.
Then the trivial representation $\one$ of SO(5) locally agrees with the
local Saito--Kurokawa lifts of (the local principal series representations
attached to) $E_4$.
Thus it follows from \cref{pr:yoshida} ($k=4$, $k'=2$)
that if an eigenform 
$f \in S_{4}(\Gamma_0(N))$ is Hecke congruent to $E_{4}$ mod $p$, then
the Saito--Kurokawa lift $SK(f)$ will be Hecke congruent to $\one$ mod $p$.

Hence \cref{thm:BM}, together with the above discussion of 
Saito--Kurokawa lifts, gives us a special case of
 \cref{SO5-conj}, namely when whe $F=\Q$, 
 $S = S_2 \cup S_3 = \{ \ell \}$ and $p \ge 5$:

\begin{cor} Let $\eps = \pm 1$.  Suppose $\ell$ is a prime and $p \ge 5$ such that
$p | (\ell^k + \eps)$.  Then there exists a degree $2$ archimedean type
$\rho_3$ Saito--Kurokawa type
Siegel modular form $SK(f)$ Hecke congruent to $\one$ mod $p$, such that
$f$ is cuspidal and holomorphic if $\eps = -1$ and $SK(f)$ is non-holomorphic if 
$\eps = +1$.
\end{cor}

\begin{rem} By considering the action of ramified Hecke operators on $\one$,
it should be possible to show that, for $p$ odd, the $\phi$ from \cref{SO5-mass} has local
epsilon sign $-1$ for $v \in S_2$ and sign $+1$ for $v \in S_3$.  
(It is not hard to prove the analogous statement for GL(2) assuming $p \ne 2$---for $p=2$ it is not true as one cannot distinguish signs mod $2$.)
Consequently, looking at the parity condition for the global packet, 
we expect that one can take the $f$ in \cref{SO5-conj} to be
holomorphic when $S_1 = \emptyset$ and $|S_2| \equiv [F:\Q] \mod 2$.  Since
$B$ is ramified at an even number of places, the condition $S_1 = \emptyset$
should be sufficient.
\end{rem}

Since local conditions, such as $p | (q_v^2 - 1)$, are often sufficient to produce
congruences, and the local factor in \cref{thm:BM} for weight 4 matches with the
local factors in the mass formula in \cref{SO5-mass}, 
it is reasonable to expect that the congruences predicted in \cref{SO5-conj}
are always achieved by Saito--Kurokawa lifts.  As with the case of 
unitary groups over fields in \cref{sec:Un-field}, 
this suggests
that the mass formula does not give ``new'' Eisenstein congruences for $\SO(5)$, but does provide
a way to see higher weight Eisenstein congruences on $\GL(2)$.
Specifically, \cref{SO5-conj} suggests the following generalization
of \cref{thm:BM} in weight 4.

\begin{conj} \label{conj:wt4} Let $N_1, N_2$ be coprime squarefree positive integers with 
$\omega(N_1)$ odd.  Suppose $p$ is an odd prime dividing
$\frac 1{45} \prod_{\ell | N_1} (\ell^2 - 1) \prod_{\ell | N_2}(\ell^2 + 1)$.
Then there exists an eigenform $f \in S_4(\Gamma_0(N_1 N_2))$ which is
Hecke congruent to $E_4$ mod $p$.  If $p \ge 7$, and $p | (\ell^2-1)$ for each $\ell | N_1$ and $p | (\ell^2+1)$ for each $\ell | N_2$, then we may take $f$ to be new of
level $N_1 N_2$ such that $f$ has Atkin--Lehner sign
$-1$ (resp.\ $+1$) if $\ell | N_1$ (resp.\ $\ell | N_2$).
\end{conj}

In general, the $f$ in the conjecture need not be new for level $N_1N_2$.
For instance if
$N_1 = 2$ and $N_2=5$, it predicts an Eisenstein congruence mod 13 in level 10.  
However there is no new congruence at level 10, but rather an old one coming
from level 5.  The reason to expect that we may take $f$ to be new in the
situation prescribed in the second part of the conjecture is because then 
each local factor contributes to the depth of the congruence as
mentioned in \cref{rem:depth} (see also \cref{rem:wt2new}).
The same reasoning suggests the obvious analogue
of \cref{conj:wt4} for parallel weight 4 Hilbert modular forms.
We also remark that by considering mass formulas for $K$ which are non-maximal,
one is led to expect weight 4 Eisenstein congruences mod $p$ in levels divisible 
by $p^2$, analogous to what happens in weight 2 (see \cref{thm:GL2}).

\subsection{Higher rank} \label{sec:BCn}
Here we briefly discuss the analogue of the
mass formulas for the higher rank analogues of $\Sp(4)$ and $\SO(5) \simeq \PGSp(4)$.

Suppose $G$ is a definite form of $\Sp(2n)$ or $\SO(2n+1)$ over $\Q$ and
$K = \prod K_v$ where $K_v = G_v$ for $v | \infty$ and $K_v$ is hyperspecial
maximal compact outside of a finite set of finite places $S$.  Then Shimura's
mass formula as stated in \cite[Propositions 7.4 and 9.4]{GHY} is of the form
\begin{equation}
m(K) = \frac{\tau(G)}{2^n} \prod_{r=1}^n \zeta(1-2r) \prod_{v \in S} \lambda_v.
\end{equation}
Here $\tau(G) = 1$ or $2$ according to whether $G = \Sp(2n)$ or $\SO(2n+1)$.

If $G(F_v)$ is quasi-split and we take $K_v$ to be the Iwahori subgroup at some finite quasi-split $v$, we simply have  $\lambda_v = [G(\frako_v) : K_v]$.  Putting $q=q_v$,
these indices are, respectively, 
$|\Sp(2n,\mathbb F_q)| = q^{n^2} \prod_{r=1}^n (q^{2r}-1)$ and (for $q$ odd)
$|\SO(2n+1,\mathbb F_q)| = q^{2n^2+3n}\prod_{r=1}^n (q^{2r}-1)$.
Alternatively, if $K_v$ is not hyperspecial but stabilizes a maximal integral lattice,
then the precise form of Shimura's mass formulas tells 
us that the numerator of
$\lambda_v$ is a divisor of $q_v^{2n}-1$ if $G=\SO(2n+1)$ and 
$\lambda_v = \prod_{r=1}^n (q_v^r + (-1)^r)$ if $G = \Sp(2n)$.
See \cite{GHY} for details.

If $G=\Sp(2n)$ and $G'$ is the quasi-split inner form, then
the Eisenstein congruences on \cref{gen-prop} should, via a generalized
Jacquet--Langlands correspondence between $G$ and $G'$,
give rise to Eisenstein congruences for
scalar weight $n+1$ holomorphic degree $n$ Siegel modular forms
(at least under suitable parity conditions for the local parameters).

However, there are conjecturally many functorial lifts from $\GL(2)$
to $\Sp(2n)$ and $\SO(2n+1)$, and there seems to be no reason
to expect that the Eisenstein congruences constructed by \cref{gen-prop}
will not arise from lifts of $\GL(2)$ congruences.  For instance,
when $n$ is odd, one has an Ikeda lift from weight 2 forms on $\PGL(2)$
to scalar weight $n+1$ Siegel modular forms on $\Sp(2n)$.
Thus we expect congruences coming from the factors $q_v\pm 1$
in the mass formula for $\Sp(2n)$ to already be explained by 
Ikeda lifts of $\GL(2)$ congruences in weight 2.
This suggests suitable generalizations of \cref{conj:wt4} to higher weight
for both elliptic and Hilbert modular forms.



\section{$G_2$} \label{sec:G2}


\subsection{Eisenstein congruences for $G_2$}

Let $F$ be a totally real number field of degree $d$ and $G/F$ be the form of
$G_2$ which is compact at each infinite place.  Necessarily $G_v$ is quasi-split
at each finite place.  Let $K = \prod K_v$ be a maximal compact subgroup of $G(\A)$
which is hyperspecial at almost all places.  At a finite place $v$, by Bruhat--Tits
theory, the conjugacy classes maximal compact subgroups correspond to 
(complements of) vertices in the extended Dynkin diagram for $G_2$.  There
are 2 types of non-hyperspecial maximal compact subgroups: (i) those corresponding
to Dynkin diagram $A_1 \times A_1$, and (ii) those corresponding to Dynkin diagram
$A_2$.  Let $S_1$ (resp.\ $S_2$) be the set of finite places $v$ where $K_v$ is of type
(i) (resp.\ (ii)).

\begin{thm} \label{thm:G2} Suppose $p$ divides
\[ m(K) = \frac 1{2^d} \zeta_F(-1) \zeta_F(-5) \prod_{v \in S_1} (q_v^4+q_v^2+1) \times \prod_{v \in S_2} (q_v^3 + 1). \]
Then there exists an eigenform $\phi \in \calA_0(G,K)$ such that $\phi$ is
Hecke congruent to $\one$ mod $p$.
\end{thm}

In particular, if $F=\Q$, we have $m(K) = \frac 1{2^5 \cdot 3^3 \cdot 7}
\prod_{\ell \in S_1} (\ell^4+\ell^2+1) \cdot \prod_{\ell \in S_2} (\ell^3 + 1)$.
Note that each factor for $v \in S_1 \cup S_2$ is a factor of $\frac{q_v^6 - 1}{q_v-1}$,
which would be the local factor if we took $K_v$ to be Iwahori.

\begin{proof}Again it suffices, to prove the mass formula for $m(K)$.
This is proved in \cite[Theorem 3.7 and (2)]{CNP}.
\end{proof}

\subsection{Lifts from GL(2)} \label{sec:lift-G2}

Conjecturally, there are three functorial lifts from PGL(2) to $G$,
corresponding to three classes of embeddings of dual groups
$\iota: \SL_2(\C) \to G_2(\C)$, and are described in
\cite{lansky-pollack}.  (There are 4 classes of embeddings
from $\SL_2(\C)$ to $G_2(\C)$,
but one is not relevant for compact $G_2$.) 
We can label these by $\iota_s$, $\iota_l$ and $\iota_r$,
which respectively send a unipotent element of $\SL_2(\C)$
to a short root for $G_2$, a long root for $G_2$ and regular
unipotent element for $G_2$.  The main aspect that is relevant here
is that $\iota_s$ (resp.\ $\iota_l$, resp.\ $\iota_r$) 
associates the weight 4 (resp.\ weight 6, resp.\ weight 2) 
discrete series of $\PGL_2(\R)$ to the trivial representation of $G(\R)$.

We remark that corresponding lifts to split $G_2$ associated
to $\iota_s$ and $\iota_l$ have been
studied, e.g., see \cite{GG:2006} and \cite{GG:2009}.

Now, by  \cref{pr:g2cong}, if $f$ is an elliptic
or Hilbert modular form of (parallel) weight $k$
 and is congruent to the Eisenstein series $E_k$
mod $p$ for $k \in \{ 2, 4, 6 \}$, then the appropriate functorial lift to $G(\A)$,
if automorphic, will be Hecke congruent to $\one$ mod $p$.
As in the case of $\SO(5)$ and unitary groups over fields, we expect that the Eisenstein congruences
from \cref{thm:G2} arise from PGL(2) Eisenstein congruences
coming from a suitable generalization of \cref{thm:BM} in weight 6.
Moreover, we expect such congruences to arise as lifts from PGL(2)
cusp forms which have Atkin--Lehner sign $+1$ at each $v \in S_2$.


\section{$\GL(2)$} \label{sec:GL2}


In this section, we discuss weight 2 Eisenstein congruences in the case
of GL(2) (or rather PGL(2)).  This was treated in \cite{me:cong}
over totally real number fields $F$ originally under the assumption that 
$h_F = h_F^+$.  However, as pointed out to us by Jack Shotton, 
the published argument
only gives cuspidal congruences mod $p$ when $p \nmid h_F$ and
$h_F$ is odd.\footnote{See \texttt{arXiv:1601.03284v4} for a corrected
version of \cite{me:cong}.}

Here we explain how to remove this class number condition 
by working with PGL(2) rather than GL(2)
and using congruence modules as in \cref{sec:cong-mod}. 
Moreover, even in the case that $p \nmid h_F$ and $h_F$ is odd,
we slightly refine our earlier result by making use of \cite{me:basis}
together with congruence modules.

Let $F$ be a totally real number field of degree $d$, 
and $B/F$ be a definite quaternion
algebra.  Let $\calO$ be a special order of $B$
(in the sense of Hijikata--Pizer--Shemanske) of the following type.
For a prime $v$ split in $B$, assume $\calO_v$ is an Eichler order
of level $\frakp_v^{r_v}$ (with $r_v = 0$ for almost all $v$).
For $v$ a finite  prime at which $B$ ramifies, assume 
$\calO_v$ is of the form $\frako_{E,v} + \mathfrak P_v^{2m}$
where $m$ is a non-negative integer, $\frako_{E,v}$ is the ring of integers
of the unramified quadratic extension $E_v/F_v$ and $\mathfrak P_v$
is the unique prime ideal for $B_v$.  In the latter case we say
$\calO_v$ is a special order of level $\frakp_v^{2m+1}$ (of unramified
quadratic type).
Let $\mathfrak N_1$ (resp.\ $\mathfrak N_2$) be
$\prod_v \frakp_v^{r_v}$ where $v$ ranges over the finite primes such that
$B/F$ splits (resp.\ ramifies) and $\frakp_v^{r_v}$ is the level of $\calO_v$.
Let $\mathfrak N = \mathfrak N_1 \mathfrak N_2$.
Let $E_{2, \mathfrak N}$ be a parallel weight 2 Eisenstein eigenform
over $F$ of level $\mathfrak N$ which has Hecke eigenvalue
$q_v$  (resp.\ 1) for $v | \mathfrak N_1$ (resp. $v | \mathfrak N_2$),
and Hecke eigenvalue $q_v + 1$ for finite $v \nmid \mathfrak N$.

\begin{thm} \label{thm:GL2}
Suppose $p$ is a rational prime which divides
\begin{equation}  \label{eq:GL2-mass}
2^{1-d-e-| \{ v | \mathfrak N_1 \}| } |\zeta_F(-1)|
 \prod_{v | \mathfrak N_1} q_v^{r_v-1}(q_v-1)
 \prod_{v | \mathfrak N_2} q_v^{r_v-1}(q_v+1),
 \end{equation}
 where $e$ is the $2$-exponent of the narrow class group
 $\Cl^+(F)$.
 Then there exists a parallel weight $2$ cuspidal Hilbert eigenform
 $f$ of level $\mathfrak N$ and trivial nebentypus such that
 $f$ is Hecke congruent to $E_{2, \mathfrak N} \bmod p$
 at all finite $v$ such that $r_v \le 1$.
 Moreover, for $v | \mathfrak N_1$ we may take $f$ such that the $v$-part of
 the exact level of $f$ is $\frakp_v^{s_v}$, where
 (i) $s_v$ is odd; (ii) $s_v = 1$ if $p \nmid q_v$; and (iii) $s_v = r_v$
 for any single chosen $v | \mathfrak N_1$ lying above $p$ (if such a $v$ exists).
 \end{thm}

\begin{proof}
Let $G = PB^\times$ and $K = \prod K_v$, where 
$K_v$ the image of $\calO_v^\times$ in $PB^\times$ for
$v < \infty$ and $K_v = G_v$ for $v | \infty$.  From the SO(3)
case of the mass formula in \cite{GHY}, one deduces that
\eqref{eq:GL2-mass} is $2^{-e} m(K)$ (compare with the mass formula
in \cite{me:cong}).
As explained in \cite{me:cong}, the constant function $\one$ on $\Cl(K)$ 
is a Hecke eigenfunction of all Hecke operators $T_v$ ($v$ finite),
with the same Hecke eigenvalues as the modular form $E_{2,\mathfrak N}$
for any $v$ with $r_v \le 1$.
Then by \cref{gen-prop}, there exists an eigenform $\phi \in \calA_0(G,K)$
such that $\phi$ is Hecke congruent to $\one$ mod $p$.  This congruence
is also valid for ramified Hecke eigenvalues when $r_v = 1$ (again,
see \emph{op.\ cit.}).

Now we want to show we can take $\phi$ to be non-abelian.  The abelian
forms in $\calA_0(G, K)$, viewed as functions on 
$\A^\times \bs B^\times(\A)/B^\times(F_\infty)$, are generated by the forms
$\psi \circ N$, where $N: B^\times \to F^\times$ is the reduced norm and
$\psi$ is a quadratic character of $\Cl^+(F)$.
Necessarily, such a form can only be congruent to $\one \bmod p$ if $p = 2$.
Using the same argument as in \cref{prop:non-abel} (the relevant congruence
module for the space of abelian forms orthogonal to $\one$ 
has $2$-exponent $e$, whereas
the congruence module for $\calA_0(G,K)$ has $2$-exponent $v_2(m(K))$),
gives such a non-abelian $\phi$.

Let $\mathcal S(G,K)$ be the orthogonal complement of the abelian subspace
of $\calA(G,K)$.  By the Jacquet--Langlands correspondence for  
modular forms from \cite{me:basis}, we have an isomorphism of Hecke modules,
for the Hecke algebras away from the set of $v | \mathfrak N_1$ with $r_v > 1$,
\[ \mathcal S(G,K) \simeq \bigoplus S_2^{\mathfrak M\text{-new}}(\mathfrak M
\mathfrak N_2), \quad \mathfrak M = \prod_{v | \mathfrak N_1} \frakp_v^{2m_v+1},
\, 1 \le 2m_v+1 \le r_v. \]
The spaces on the right are the spaces of parallel weight 2 Hilbert cusp forms
of level $\mathfrak M \mathfrak N_2$ which are locally new at each
$v | \mathfrak M$ (the associated local representation of $\PGL_2(F_v)$ has conductor
$\frakp_v^{2m_v+1}$), and $\mathfrak M$ runs over divisors of $\mathfrak N_1$
which have odd exponent at every $v | \mathfrak N_1$.  This shows (i).

Let $f$ be a Hilbert
modular form corresponding to $\phi$.
If $v | \mathfrak N_1$ such that $p \nmid q_v$, then if necessary
we may enlarge $K$ by taking $K_v = \calO_{B,v}^\times$ at $v$ so
that $r_v = 1$.  This forces $f$ to have exact level
$\frakp_v$ at $v$, i.e., we may assume (ii).

For (iii), suppose there exists $v | \mathfrak N_1$ such that $p | q_v$.
If $r_v = 1$, there is nothing to show, so assume $r_v \ge 3$.
Then we may use the above decomposition of $\mathcal S(G,K)$ together with
the argument from \cref{prop:non-abel}.  Namely, for a sufficiently large rationality
field $L$, we may decompose 
$\mathcal A^L_0(G,K) = X_1(L) \oplus X_2(L)$, where $X_1(L)$ is generated
by abelian forms together with cuspidal eigenforms which have level at most
$\frakp_v^{r_v-2}$ at $v$, and $X_2(L)$ is generated by cuspidal eigenforms 
which have exact level $\frakp_v^{r_v}$ at $v$.  Now $X_1(L) = A^L_0(G,K')$
where $K'$ is defined in the same way as $K$ except replacing $r_v$ with $r_v-2$.
Then the $p$-exponent of the congruence module for $X_1$ is simply 
$v_p(m(K'))$.  But this is strictly less than $v_p(m(K))$, so the argument of
\cref{prop:non-abel} gives an eigenform in $X_2(L)$ which is Hecke congruent 
to $\one$ mod $p$.
\end{proof}

\begin{rem}  \label{rem:wt2new}
If $v | \mathfrak N_2$ such that $p | (q_v+1)$ if $r_v = 1$ (resp.\
$p | q_v$ if $r_v > 1$), we expect that we can also assume the $f$ in
the theorem is locally new at $v$.  Similarly, we expect we can impose (iii) for
all $v$ such that $p | q_v$.
This is because then the local factor at 
$v$ contributes to the $v_p(m(K))$, i.e., contributes to the $p$-exponent
of the relevant congruence module.  Alternatively, this factor contributes to
the depth of the congruence mentioned in \cref{rem:depth}.  In order to prove
this along the lines of our argument for (iii), we would need to know
the $p$-exponent of the congruence module for the $v$-old forms.
We do not attempt to study this here.
\end{rem}

\begin{rem}
Ribet and Yoo (see \cite{yoo}) have studied weight 2 Eisenstein congruences
with fixed Atkin--Lehner signs for elliptic modular forms of squarefree level  
under some conditions.  If $p > 2$ and $\mathfrak N$ is squarefree,
then $f$ as in the theorem necessarily has Atkin--Lehner sign $-1$
at each $v | \mathfrak N_1$, and Atkin--Lehner sign $+1$ at each
$v | \mathfrak N_2$ such that the $v$-part of the exact level of $f$
is $\frakp_v$.
\end{rem}

\begin{cor} Let $F=\Q$ and $p$ be prime.  Then for any $m \ge 1$ (resp.\ $m \ge 3$)
if $p$ is odd (resp.\ $p=2$), there exists a newform $f \in S_2(p^{2m+1})$
which is Hecke congruent to $E_{2,p} \bmod p$ away from $p$.
\end{cor}


\section{Special mod $p$ congruences for $\UU(p)$} \label{sec:special}


Given a weight 2 cuspidal newform $f$ on $\PGL(2)$ 
whose $p$-th Fourier coefficient is $-1$ for a $p$ dividing the level
(i.e., locally is the unramified quadratic twist of Steinberg at $p$), one can
use quaternionic modular forms to construct a newform $g$ of the same weight and
level which is congruent to $f$ mod 2 and has Fourier coefficient $+1$ at $p$
(i.e., locally is the untwisted Steinberg at $p$), at least in the case that 
the level is a squarefree product of an odd number of primes \cite{me:cong2}. 
In general $g$ may be Eisenstein, but under some simple explicit conditions
it can be chosen to be cuspidal.
Here we extend this to higher rank in the setting of unitary groups.   

Let $E/F$ be a CM extension of number fields.
Let $S$ be a non-empty finite set of finite places of $F$ which split in $E$.
Consider a definite unitary group $G = \UU_A(n)$, 
where $A/E$ is a degree $n$ central division algebra such that, for each finite $v \in S$,
$G(F_v) \simeq D_v^\times$ for some division algebra $D_{v}/F_{v}$.
Let $K \subset G(\A)$ be as in the beginning of \cref{sec:eis-Un} such that 
$K_{v} \simeq \calO_{D_{v}}^\times$ for $v \in S$.

If $\pi$ occurs in $\calA(G, K; 1)$, then, for $v \in S$, 
$\pi_{v}$ is 1-dimensional, and thus of the
form $\mu_{v} \circ \det$ for some unramified character 
$\mu_{v} : F_v^\times \to \C^\times$
such that $\mu_{v}^n = 1$.  (Here $\det$ denotes the reduced norm from $D_{v}$
to $F_{v}$.)  Consider a collection $\mu_S = (\mu_v)_{v \in S}$ of such $\mu_v$.
We denote by $\calA(G, K; 1)^{\mu_S}$ the subspace
of $\calA(G, K; 1)$ generated by $\pi^K$ where $\pi$ runs over all
$\pi$ contributing to $\calA(G,K; 1)$ such that
$\pi_{v} \simeq \mu_{v} \circ \det$ for all $v \in S$.  
When $\mu_{v} = 1$ for all $v \in S$, we write this as $\calA(G, K; 1)^{1_S}$.
Let $\zeta_m = e^{2 \pi i/m}$.

\begin{lem} \label{lem:special}
Fix $p | n$.
Suppose $\mu_{v}$ has prime power order $p^{r_v} | n$ for all $v \in S$.  
Let $\calO$ be the ring of integers of some number field containing $\zeta_{p^r}$,
and $\frakp$ a prime of $\calO$ above $p$.  Then for any
nonzero $\phi \in \calA^\calO(G, K; 1)^{\mu_S}$, there exists 
a nonzero $\phi' \in \calA^\calO(G, K; 1)^{1_S}$
such that $\phi' \equiv \phi \bmod \frakp$.
\end{lem}

\begin{proof}
Let $\bar G = G/Z$ and $\bar K = Z(\A) K/ Z(\A)$.  Then we may view $\phi$
as a function on $\Cl(\bar K)$.  For $v \in S$, fix a uniformizer $\varpi_{D, v}$ of 
$D_{v}$ such that $\det \varpi_{D,v} = \varpi_{v}$.
Then $\varpi_{D,v}$ acts on $\Cl(\bar K)$ via right multiplication with order dividing 
$n$.  Denote this action by $\sigma_v$.  Let $Y_1, \ldots, Y_t$ be the orbits of the
ensuing action of $\Gamma = \prod_{v \in S} \langle \varpi_{D,v} \rangle$
on $\Cl(\bar K)$. 

Note that for $\phi \in \calA(G, K; 1)$, we have $\phi \in  \calA(G, K; 1)^{\mu_S}$ 
if and only if $\phi(\sigma_v(y)) = \mu_{v}(\varpi_{v}) \phi(y)$ for all 
$y \in \Cl(\bar K)$, $v \in S$.
Fix some orbit $Y_i$ and write $Y_i = \{ y_1, \ldots, y_s \}$.  
Then for any $y_j \in Y_i$, there is some sequence of $\sigma_v$'s (with $v \in S$) 
whose composition sends $y_1$ to $y_j$.  Hence $\phi(y_j) = \zeta \phi(y_1)$ 
for some $p$-power root of unity $\zeta$.  Since $\zeta \equiv 1 \bmod \frakp$,
defining $\phi'(y_j) = \phi(y_1)$ for $1 \le j \le s$ gives a function on $Y_i$
which is congruent to $\phi$ mod $\frakp$.  Defining $\phi'$ this way on each orbit
completes the proof.
\end{proof}

The following is a partial analogue of \cite[Theorem 1.3]{me:cong2} in higher rank,
and the proof is similar in spirit.  

\begin{thm}  Let $n=p$ be an odd prime, and assume $\eqref{ECU}$ for $n$.
Let $S$ be a finite set of finite places of $F$ which are split in 
$E/F$.    Suppose $p$ does not divide $|\Cl(\UU(1))|$ nor
\[  \prod_{r=1}^p L(1-r, \chi_{E/F}^r) \times \prod_{v \in S} 
\left( \prod_{r=1}^{p-1} (q_v^{r} - 1) \right). \]
For each $v \in S$, let $\mu_v$ be
an unramified character of $F_{v}^\times$ of order $1$ or $p$. 
For finite $v \not \in S$, assume $K_v$ is a hyperspecial
maximal compact open subgroup of $\UU_p(F_v)$.

Let $\pi$ be an automorphic representation of $G'(\A) = \UU_p(\A)$ holomorphic
of parallel weight $p$ with trivial central character such that $\pi_E$ is cuspidal, 
$K_v$-spherical for all finite $v \not \in S$, and
$\pi_{v} \simeq \St_{v} \otimes \mu_{v}$ for all $v \in S$.  
Then there exists an automorphic representation of $\pi'$ of $G'(\A)$, also
holomorphic of parallel weight $p$ with trival central character and $\pi_E'$ cuspidal, 
such that $\pi_v'$ is $K_v$-spherical for all finite $v \not \in S$,
$\pi'_{v} \simeq \St_{v}$ for all $v \in S$ 
and $\pi$ is Hecke congruent to $\pi' \bmod p$.
\end{thm}

\begin{proof}
Let $G = \UU_A(p)$ be a totally definite inner form of $G'$ which is
locally isomorphic to $G'$ at all finite places outside of $S$ and compact
at each $v \in S$.  Now $\pi$ corresponds to a simple generic formal 
parameter $\psi$, which we may think of as the cuspidal representation
$\pi_E$ of $\GL_p(\A_E)$.  Then there exists an automorphic representation
$\sigma \in \Pi_\psi(G)$ such that $\sigma_v \simeq \pi_v$ for all finite $v \not \in S$,
$\sigma_v \simeq \mu_v \circ \det$ for $v \in S$, and $\sigma_v$ is trivial for $v | \infty$.

For $v \in S$, let $D_v/F_v$ be a division algebra isomorphic to $A_w/E_w$
for some $w | v$ and put $K_v = \calO_{D_v}^\times$.  For $v | \infty$, put $K_v = G_v$.
Set $K = \prod K_v$.  Then $\sigma$ occurs
in $\calA_0(G, K; 1)$ and we may take a nonzero $\phi \in \sigma^K$ to have values
in the ring of integers $\calO$ of some number field $L$.  Let $\frakp$ be a prime of 
$\calO$ above $p$.  

If $\phi \equiv 0 \bmod \frakp$, we may consider the Hilbert class field $H_L$ of $L$
so that $\frakp$ is unramified and principal in $H_L$.
Thus we may scale $\phi$ by an element of $H_L$ to assume that $\phi \not \equiv 0
\bmod \frakp$, and moreover $\phi \not \equiv 0 \bmod \mathfrak P$ for some prime
$\mathfrak P$ of $H_L$ above $p$.  Hence by replacing $L$ with $H_L$ and $\frakp$
with $\mathfrak P$ if necessary, we may and will assume 
$\phi \not \equiv 0 \bmod \frakp$.

By \cref{lem:special}, there
exists a nonzero $\phi' \in \calA^{\calO}(G,K;1)^{1_S}$ such that $\phi' \equiv \phi \bmod \frakp$.
We claim $\phi'$ is non-abelian.  First note that, since
$p \nmid |\Cl(\UU(1))|$, the only non-abelian forms in $\calA(G,K;1)$ are constant
functions.  However, if $\phi' = c \one$ for some $c \in \calO$, then
$\phi \in \calA_0(G,K; 1)$ implies $0 = (\phi, \one) \equiv  c (\one, \one) \equiv
 c m(K) \bmod \frakp$.  This would mean $\frakp | m(K)$,
 since $\phi' \equiv \phi \not \equiv 0
\bmod \frakp$ implies $c \not \equiv 0 \bmod \frakp$.
But this is impossible by our indivisibility assumption
together with \cref{prop:mass-Un}.  

Then, as in the proofs of \cref{gen-prop} and \cref{thm:main-Un}, we can transfer this
to a mod $\frakp$ Hecke congruence with a non-abelian eigenform $\phi''$ on $G$,
and obtain a congruent $\pi'$ on $G'$ as asserted.
\end{proof}

\begin{rem} It is clear from the proof that one can allow $K_v$ to be a finite index
subgroup of a hyperspecial maximal compact $K_v^0$ at a finite number of $v \not
\in S$ by also imposing the conditions $p \nmid [K_v^0 : K_v]$.  At such $v$,
then the appropriate statement is that both $\pi_v$ and $\pi_v'$ have nonzero
$K_v$-fixed vectors.
\end{rem}

\begin{rem} In the case of weight 2 elliptic modular forms of squarefree level,
we showed in \cite{me:ref-dim} 
that there is a strict (though small) bias towards local ramified
factors being Steinberg as opposed to the unramifed quadratic twist of Steinberg.
In \cite{me:cong2}, this bias was shown to be related to the existence of mod 2 
congruences of forms which are twisted Steinberg at certain places to untwisted
Steinberg at these places.  Similarly, the above congruence result suggests
a bias towards local untwisted Steinberg representations on $\UU_p(\A)$.  
Specifically, in the notation
of the proof, we expect that the number of representations occurring in 
$\calA(G, K; 1)^{1_S}$ is always at least the number of representations
occurring in $\calA(G, K; 1)^{\mu_S}$.  The above result implies the analogous
statement is true for mod $p$ Hecke congruence classes of representations.
\end{rem}


\end{document}